\documentclass[11pt,a4paper,twoside]{amsart}
\usepackage{amssymb}
\usepackage{amsfonts}
\usepackage{mathrsfs}
\usepackage{amsmath,amscd}
\usepackage{amssymb}
\usepackage{amsthm}
\usepackage{enumerate}
\usepackage[english,francais]{babel}
\usepackage{hyperref}

\usepackage{tikz}
\usetikzlibrary{calc}
\usetikzlibrary{fadings,arrows}

\usepackage[all]{xy}
\usepackage{t1enc}

\newtheorem{e-proposition}[theorem]{Proposition}

\newtheorem{e-definition}[theorem]{Definition\rm}

\newtheorem{theoreme}{Th\'eor\`eme}[section]
\newtheorem{lemme}[theoreme]{Lemme}

\newtheorem{proposition}[theoreme]{Proposition}
\newtheorem{corollaire}[theoreme]{Corollaire}

\newtheorem{remarque}{\it Remarque}

\newcommand{\bm}[1]{\mbox{\boldmath{$#1$}}}

\newcommand{\DD}[0]{\ensuremath{\mathbb{D}}}
\newcommand{\EE}[0]{\ensuremath{\mathbb{E}}}
\newcommand{\RR}[0]{\ensuremath{\mathbb{R}}}
\newcommand{\AF}[0]{\ensuremath{\mathbb{A}}}
\newcommand{\ZZ}[0]{\ensuremath{\mathbb{Z}}}
\newcommand{\QQ}[0]{\ensuremath{\mathbb{Q}}}
\newcommand{\NC}[0]{\ensuremath{\mathcal{N}}}
\newcommand{\KK}[0]{\ensuremath{\mathbf{k}}}
\newcommand{\CO}[0]{\ensuremath{\mathcal{C}}}
\newcommand{\EN}[1]{\ensuremath{\operatorname{\Gamma^{\star}(#1)}}}

\newcommand{\spec}[0]{\ensuremath{\operatorname{Spec}}}
\newcommand{\B}[2]{\ensuremath{\operatorname{S(#1,h_{#2})}}}

\newcommand{\Bo}[0]{\ensuremath{\operatorname{\mathcal{B}}}}
\newcommand{\Dl}[0]{\ensuremath{\operatorname{D}}}
\newcommand{\Bl}[0]{\ensuremath{\operatorname{B}}}
\newcommand{\N}[0]{\ensuremath{\operatorname{N}}}
\newcommand{\M}[0]{\ensuremath{\operatorname{M}}}
\newcommand{\pp}[0]{\ensuremath{\operatorname{PPr}}}
\newcommand{\fhq}[0]{\ensuremath{\operatorname{h_q}}}
\newcommand{\hq}[2]{\ensuremath{\operatorname{h_q}( #1,#2)}}
\newcommand{\GD}[1]{\ensuremath{\operatorname{\Gamma^{\star}(#1)_{\mathcal{G}}}}}
\newcommand{\SD}[0]{\ensuremath{\operatorname{S_{\mathcal{G}}}}}

\newcommand{\pgcd}[2]{\ensuremath{\operatorname{p.g.c.d(#1,#2)}}}
\newcommand{\FI}[0]{\ensuremath{\operatorname{FI}}}
\newcommand{\sing}[0]{\ensuremath{\operatorname{Sing}}}
\newcommand{\Ess}[0]{\ensuremath{\operatorname{Ess}}}
\newcommand{\Hom}[0]{\ensuremath{\operatorname{Hom}}}

\newcommand{\ord}[0]{\ensuremath{\operatorname{Ord}}}
\newcommand{\ordt}[1]{\ensuremath{\operatorname{Ord}_{t} #1}}
\newcommand{\degt}[1]{\ensuremath{\operatorname{Deg}_{t} #1}}
\newcommand{\p}[0]{\ensuremath{\operatorname{p}}}
\newcommand{\pinf}[0]{\ensuremath{\operatorname{p}_{\infty}}}
\newcommand{\f}[0]{\ensuremath{\operatorname{f}}}
\newcommand{\e}[0]{\ensuremath{\operatorname{e}}}
\newcommand{\g}[0]{\ensuremath{\operatorname{g}}}
\newcommand{\m}[0]{\ensuremath{\operatorname{m}}}
\newcommand{\E}[0]{\ensuremath{\operatorname{E}}}
\newcommand{\D}[0]{\ensuremath{\operatorname{D}}}
\newcommand{\Div}[1]{\ensuremath{\operatorname{Div(#1)}}}
\newcommand{\sch}{\mathcal{S}ch}
\newcommand{\ens}{\mathcal{E}ns}
\newcommand{\com}[1]{\ensuremath{#1^{\star}}}

\newcommand{\vspde}{\vspace{0.2cm}}
\begin{document}
\selectlanguage{francais}
\title[Une famille d'hypersurfaces  avec application de Nash bijective]
{R\'esolution du probl\`eme des arcs de Nash pour une famille d'hypersurfaces  quasi-rationnelles}

\author{Maximiliano LEYTON-ALVAREZ}
\address{Universit\'e
Grenoble I, Institut Fourier, UMR 5582 CNRS-UJF, BP 74, 38402
St.\ Martin d'H\`eres c\'edex, France}

\email{leyton@ujf-grenoble.fr}

\date{\today}

\thanks{}

\begin{abstract}
\selectlanguage{francais}
 Le probl\`eme des arcs de Nash pour les singularit\'es normales de surfaces affirme qu'il y aurait autant de familles d'arcs sur un germe de surface singulier $(S,O)$ que de diviseurs essentiels sur $(S,O)$. Il est connu que ce probl\`eme se r\'eduit \`a \'etudier les singularit\'es quasi-rationnelles. L'objet de cet article est de r\'epondre positivement au probl\`eme de Nash pour une famille d'hypersurfaces  quasi-rationnelles non rationnelles. On applique la m\^eme m\'ethode pour r\'epondre positivement \`a ce probl\`eme dans les cas de singularit\'es de type $\EE_6$ et $\EE_7$  et pour fournir une nouvelle preuve dans le cas de singularit\'es de type  $\DD_n$, $n\geq 4$. 

\vskip 0.5\baselineskip

\selectlanguage{english}
\noindent{A{\tiny BSTRACT}.}
\noindent {\bf  A solution to the Nash problem on arcs for a family of  quasi-rational hypersurfaces.} The Nash problem on arcs for normal surface singularities states that there are as many arc families on a germ $(S,O)$ of a singular surface as there are essential divisors over $(S,O)$.  It is  known that this problem can be reduced to the study of quasi-rational singularities.  In this paper we give a positive answer to the Nash problem for a family of non-rational quasi-rational hypersurfaces. The same method is applied to answer positively to this problem in the case of $\EE_6$ and $\EE_7$ type singularities, and to provide new proof in the case of $\DD_n$, $n\geq 4$, type singularities.  

\end{abstract}
\maketitle

\selectlanguage{francais}
\section{Introduction}
\label{sec:int}

Soient $\KK$ un corps alg\'ebriquement clos de caract\'eristique nulle,  $V$ une vari\'et\'e alg\'ebrique normale sur $\KK$  et   $\pi:X\rightarrow V$  une d\'esingularisation divisorielle de $V$, c'est-\`a-dire $\pi$ est une d\'esingularisation ($\pi$ est un morphisme propre et birationnel tel que le morphisme  $\pi: X\backslash \pi^{-1}(\sing V)\rightarrow V\backslash \sing V$ est un isomorphisme, o\`u $\sing V$ est le lieu singulier de $V$ et $X$ est lisse) telle que  les composantes irr\'eductibles de  la fibre exceptionnelle $\pi^{-1}(\sing V)$ sont de codimension $1$ dans $X$. On sait, d'apr\`es le th\'eor\`eme de r\'esolution des singularit\'es  d'Hironaka, que  cette d\'esingularisation existe. 

Si $\pi':X'\rightarrow V$ est une autre d\'esingularisation de $V$, alors $(\pi')^{-1}\circ\pi:X\dashrightarrow X'$ est une application birationnelle. Soit $\E$ une composante irr\'eductible de la fibre exceptionnelle de $\pi$. Comme $\E$ est un diviseur et $X$ est une vari\'et\'e alg\'ebrique normale ($X$ est une vari\'et\'e lisse), il existe un ouvert $\E^0$ de $\E$  sur lequel l'application birationnelle $(\pi')^{-1}\circ\pi$ est bien d\'efinie.    
Le diviseur $\E$ est appel\'e  {\it diviseur essentiel  sur} $V$ si pour toute d\'esingularisation 
 $\pi':X'\rightarrow V$ de $V$  l'adh\'erence de  $(\pi')^{-1}\circ\pi(E^0)$ dans $X'$ est une composante irr\'eductible de la fibre exceptionnelle du morphisme  $\pi'$. On note $\Ess(V)$ l'ensemble de diviseurs essentiels sur $V$.

On remarque que, si $V$ est une surface alg\'ebrique normale  sur $\KK$, les diviseurs essentiels sur $V$ sont les composantes irr\'eductibles de la fibre exceptionnelle de la r\'esolution minimale de $V$.\\

Soit $K$ un corps d'extension de $\KK$. Un morphisme $\spec K[t]/(t^{m+1}) \rightarrow V$  (resp. $\spec K[[t]]\rightarrow V$) est appel\'e $(K,m)$-jet (resp. $K$-arc).\\  

Soit $\sch /\KK$  la cat\'egorie des sch\'emas sur $\KK$ et  $\ens$ la cat\'egorie des ensembles. 
On consid\`ere le  foncteur  contravariant suivant:
\begin{center}
$F_m:\sch/\KK \longrightarrow \ens$, $Y \mapsto \Hom(Y\times_{\KK}\spec \KK[t]/(t^{m+1}),V).$
\end{center}

Ce  foncteur est repr\'esentable de forme canonique par un sch\'ema $V_m$ de type fini sur $\KK$, c'est-\`a-dire on a:
\begin{center}
$\Hom(Y\times_{\KK} \spec \KK[t]/(t^{m+1}),V)\cong \Hom(Y,V_{m}),$
\end{center}
o\`u $Y$ est un sch\'ema  quelconque sur $\KK$. Le sch\'ema $V_m$  est appel\'e  {\it l'espace de m-jet sur} $V$.

L'homomorphisme surjectif canonique $\KK[t]/(t^{m+1})\rightarrow \KK[t]/(t^{m})$ induit un morphisme affine  $\p_m: V_m\rightarrow V_{m-1}$.  Les morphismes $\p_{m\;n}:V_m\rightarrow V_n$, o\`u $n<m$ et $\p_{m\;n}:=\p_{n+1}\circ\cdots\circ \p_m$, forment un syst\`eme projectif. Comme les morphismes $\p_m: V_m\rightarrow V_{m-1}$ sont affines (voir \cite{Ish07} ou \cite{EiMu09}), la limite projective existe. On note $V_{\infty}$ cette limite projective, c'est-\`a-dire   $V_{\infty}:= \displaystyle \lim_{\leftarrow m} V_m$.   La limite projective $V_{\infty}$ est un sch\'ema qui n'est pas en g\'en\'eral de type fini sur $\KK$. Le sch\'ema $V_{\infty}$ est appel\'e {\it l'espace d'arcs sur } $V$.\\

 L'espace d'arcs $V_{\infty}$ a {\it la propri\'et\'e fonctorielle} suivante (voir  \cite{IsKo03}):\
Le foncteur $Y\rightarrow \Hom(Y\widehat{\times}_{\KK}\spec \KK[[t]], V)$, 
 o\`u $Y$ est un sch\'ema quelconque  sur $\KK$ et $Y\widehat{\times}_{\KK}\spec \KK[[t]]$ est le compl\'et\'e formel  du sch\'ema $ Y\times_{\KK}\spec \KK[[t]] $  le long du sous-sch\'ema  $ Y\times_{\KK}\spec{\KK} $,  est repr\'esentable par le sch\'ema $V_{\infty}$.\\

D'apr\`es la  propri\'et\'e  fonctorielle de l'espace d'arcs $V_{\infty}$,  les $K$-points de $V_{\infty}$ sont en correspondance bijective avec les   $K$-arcs sur $V$. Par abus de notation, pour $\alpha\in V_{\infty}$, on note $\alpha$ son $\KK_{\alpha}$-arc correspondant, o\`u  $\KK_{\alpha}$ est le corps r\'esiduel du point $\alpha$.\\

Soit $\pinf:V_{\infty}\rightarrow V$ la projection canonique  $\alpha\mapsto \alpha(0)$, o\`u $0$ est le point ferm\'e de $\spec \KK_{\alpha}[[t]]$. On note $V_{\infty}^{s}:= \pinf^{-1}(\sing V)$, o\`u $\sing V$ est le lieu singulier de $V$, et $\mathcal{CN}(V)$ l'ensemble de composantes irr\'eductibles de  $V_{\infty}^{s}$.\\

 Nash a d\'emontr\'e  que l'application suivante est bien d\'efinie et injective (voir \cite{Nas95}): 
\begin{center}
$\mathcal{N}_{V}:\mathcal{CN}(V)\rightarrow \Ess(V), C\mapsto \overline{\{\widehat{\alpha}(0)\}}$, 
\end{center}
o\`u  le $\KK_{\alpha}$-arc $\alpha$ est le point g\'en\'erique de $C\in \mathcal{CN}(V)$, le $\KK_{\alpha}$-arc $\widehat{\alpha}$ est le rel\`evement \`a $X$ du $\KK_{\alpha}$-arc $\alpha$, c'est-\`a-dire $\widehat{\alpha}$ est l'unique morphisme tel que $\pi\circ\widehat{\alpha}=\alpha$, et   $\overline{\{\widehat{\alpha}(0)\}}$ l'adh\'erence  du point $\widehat{\alpha}(0)$. Cette application est appel\'ee {\it l'application de Nash associ\'ee \`a} $V$.\\

 Le probl\`eme de Nash consiste \`a \'etudier la surjectivit\'e  de l'application de Nash  $\mathcal{N}_{V}$.\\

  Dans plusieurs cas la  surjectivit\'e de cette application a \'et\'e prouv\'ee. Par exemple  pour les singularit\'es $\AF_n$ (\cite{Nas95}), les singularit\'es $\DD_n$ (\cite{Ple08}), les surfaces sandwichs (\cite{LeRe99}), les vari\'et\'es toriques (\cite{IsKo03}), les hypersurfaces quasi-ordinaires  (\cite{Gon07}), les vari\'et\'es toriques stables (\cite{Pet09}),  et  d'autres cas qu'on peut trouver dans les articles suivants (\cite{Ish05}, \cite{Ish06}, \cite{Gon07}, \cite{Gole97}, \cite{Lej80},\cite{Mor08}, \cite{PlPo06}, \cite{PlPo08}, \cite{Reg95}).\\

Dans l'article \cite{IsKo03}  Ishii et Kollar montrent que l'application de Nash associ\'ee \`a l'hypersurface de $\AF_{\KK}^5$ donn\'ee par l'\'equation $x_1^3+x_2^3+x_3^3+x_4^3+x_5^6=0$ n'est pas surjective. Or, en dimension deux ou trois il n'y a pas d'exemple publi\'e o\`u l'application de Nash  n'est pas surjective.\\

Soit $\E$ un diviseur essentiel sur $V$ ($\E\in \Ess(V)$). On note $N_{\E}$ l'adh\'erence dans  $V_{\infty}$ de  l'ensemble suivant:
\begin{center}
$\{\alpha\in V_{\infty}\backslash (\sing V)_{\infty}\mid \widehat{\alpha}(0)\in E\}$,
\end{center}
 o\`u $(\sing V)_{\infty}$ est le sous-ensemble ferm\'e de $V_{\infty}$ des arcs qui sont concentr\'es en $\sing V$.
 L'ensemble $N_{\E}$ est irr\'eductible et on a:
\begin{center} 
$V_{\infty}^{s}=\bigcup_{\E}N_{\E}$,

 \end{center}
voir \cite{Reg06}. On note $\alpha_{\E}$ le point g\'en\'erique de $N_{\E}$.\\

Un morphisme $\omega:\spec K[[s,t]]\rightarrow V$ est appel\'e $K$-{\it wedge  sur} $V$. D'apr\`es la propri\'et\'e fonctorielle de l'espace d'arcs $V_{\infty}$,  les $K$-wedges sont en correspondance bijective avec les   $K[[s]]$-points de $V_{\infty}$. L'image du point ferm\'e (resp. du point g\'en\'erique) de $\spec K[[s]]$ dans $V_{\infty}$ est appel\'e le centre (resp. l'arc g\'en\'erique)  du $K$-wedge $\omega$.\\
 
Un $K$-wedge $\omega$ est appel\'e  {\it $K$-wedge admissible centr\'e  en} $N_{\E}$  si le centre (resp. l'arc g\'en\'erique)  de $\omega$  est le point g\'en\'erique de $N_{\E}$ (resp. appartient \`a $V_{\infty}^{s}$). Dans ce cas, par d\'efinition, le corps  $K$ est forc\'ement un corps d'extension du corps r\'esiduel $\KK_{\alpha_{\E}}$ du point $\alpha_{\E}$.
Le $K$-wedge $\omega$ peut \^etre interpr\'et\'e comme une d\'eformation \`a un param\`etre des coefficients du comorphisme $\alpha_{\E}^{\star}$ de l'arc $\alpha_{\E}$.

 Dans  \cite{Reg06} (dans le cas de surfaces voir \cite{Lej80}) l'auteur  montre que  $\E$ appartient \`a l'image de l'application de Nash $\mathcal{N}_V$  si et seulement si  tout $K$-wedge admissible $\omega$ centr\'e en $N_{\E}$ se rel\`eve \`a $X$, o\`u $K$ est un corps d'extension quelconque du corps $K_{\alpha_{\E}}$,  c'est-\`a-dire  s'il existe un $K$-wedge $\widehat{\omega}$ sur $X$ tel que $\pi\circ \ \widehat{\omega}=\omega$.\\

On consid\`ere l'hypersurface  $\B{p}{q}$ de $\AF_{\KK}^3$ donn\'ee par l'\'equation $z^p+\hq{x}{y}=0$, o\`u $\fhq$ est un polyn\^ome homog\`ene de degr\'e $q$  sans facteurs multiples, $p\geq 2$, $q\geq 2$ deux entiers  premiers entre eux. Par exemple si $ 
 \hq{x}{y}=x^q+y^q$, $\B{p}{q}$ est une surface de {\it Pham-Brieskorn}. Les surfaces $\B{p}{q}$  sont toutes quasi-rationnelles  (voir \cite{FlZa03}) et  rationnelles si et seulement si $q=2$ ou $(p,q)=(2,3)$. On remarque que  $\B{p}{2}$ est une singularit\'e du type $\AF_{p-1}$.  Le r\'esultat principal de cet article  est le th\'eor\`eme suivant.

\begin{theoreme}
\label{th:nashBP} Pour tous les entiers $p\geq 2$, $q\geq 2$ premiers entre eux, l'application de Nash $\mathcal{N}_{\B{p}{q}}$ associ\'ee \`a $\B{p}{q}$ est bijective.
\end{theoreme}
Une note sur la  preuve de ce r\'esultat a \'et\'e publi\'ee dans \cite{Ley11}.\\   

On remarque que les crit\`eres de \cite{Mor08} et \cite{PlPo06} ne s'appliquent pas en g\'en\'eral aux familles $\B{p}{q}$.\\

En utilisant la m\^eme m\'ethode que dans la preuve du Th\'eor\`eme \ref{th:nashBP} on obtient le r\'esultat suivant:

\begin{theoreme} 

\label{th:nashDP} Si $S$ est une singularit\'e du type $\EE_{6}$ ou $\EE_{7}$,  l'application  de Nash $\mathcal{N}_{S}$ associ\'ee \`a $S$ est bijective.
\end{theoreme}
 On obtient aussi une preuve simple de la bijectivit\'e  de l'application de Nash pour plusieurs cas connus, par exemple $\DD_{n}$, o\`u $n$ est un entier, $n\geq 4$ (voir \cite{Ple08}).\\

R\'ecemment  sur ArXiv une preuve  de la bijectivit\'e de l'application de Nash pour la singularit\'e de type $\EE_6$ (resp. pour les surfaces quotients) a \'et\'e publi\'ee,  voir \cite{PlSp10} (resp. voir  \cite{Pe10})). Mais les m\'ethodes de d\'emonstration sont diff\'erentes de celle que nous utilisons.\\

Cet article est organis\'e de la fa\c con suivante: dans premi\`ere section on prouve  quelques r\'esultats qui vont \^etre utilis\'es dans toutes les sections de cet article et on d\'emontre le Th\'eor\`eme \ref{th:nashBP}; le Th\'eor\`eme \ref{th:nashDP} est le sujet de la deuxi\`eme section, on donne  la preuve du cas $\EE_6$ en d\'etails  et un r\'esum\'e pour le cas $\EE_7$;  dans la derni\`ere section on donne une preuve simple de la bijectivit\'e  de l'application de Nash  pour les singularit\'es du type $\DD_n$, $n\geq 4$  diff\'erente de celle de l'article  \cite{Ple08}.

\section*{Remerciements} Je tiens \`a remercier  G\'erard Gonzalez-Sprinberg pour les nombreuses discussions et son constant encouragement. Je remercie \'egalement Marcel Morales pour l'int\'er\^et qu'il a port\'e \`a mon travail. 

\section{Le probl\`eme Nash pour les hypersurfaces quasi-rationelles $\B{p}{q}$}

On rappelle que $\B{p}{q}$ est l'hypersurface de $\AF_{\KK}^3$ donn\'ee par l'\'equation $z^p+\hq{x}{y}=0$, o\`u $\fhq$ est un polyn\^ome homog\`ene de degr\'e $q$  sans facteurs multiples et $p\geq 2$, $q\geq 2$ sont deux entiers  premiers entre eux.\\

Nash a d\'emontr\'e que l'application  $\mathcal{N}_{\AF_{m}}$ associ\'ee \`a  la surface de type $\AF_{m}$, $m\geq 2$, est bijective (voir \cite{Nas95}). On suppose donc que $q\geq 3$, car l'hypersurface $\B{p}{2}$ est une singularit\'e du type  $\AF_{p-1}$.\\

\subsection{D\'esingularisation des hypersurfaces $\bm{\B{p}{q}}$}

Un syst\`eme de coordonn\'ees affines $\{x,y,z\}$ de $\AF_{\KK}^3$ \'etant fix\'e, notons $ O $ l'origine de $\AF_{\KK}^3$. On d\'ecrit la d\'esingularisation de $\B{p}{q}$ en utilisant les constellations toriques de points infiniment voisins de $O$ (voir \cite{CGL96}), les \'eventails de Newton (voir \cite{GoLe91}) et les $G$-{\it d\'esingularisations} (voir \cite{BoGo95}).\\
 
Soient $\N=\ZZ^3$ muni de sa base standard $\{\e_1, \e_2,\e_3\}$ et $\M=\Hom_{\ZZ}(N,\ZZ)$ le dual de $\N$. On note $\{\e_1^{\star}, \e_2^{\star},\e_3^{\star}\}$ la base duale de $\{\e_1, \e_2,\e_3\}$.
 On identifie la $\KK$-alg\`ebre $\KK[\M]$ avec la $\KK$-alg\`ebre $\KK[x,x^{-1},y,y^{-1},z,z^{-1}]$ par l'isomorphisme qui envoie  le caract\`ere  $\chi^{\e_1^{\star}}$ (resp. $\chi^{\e_2^{\star}}$, $\chi^{\e_3^{\star}}$) sur $x$ (resp. $y$, $z$).

Soit $\Sigma$ un \'eventail en $\N_{\RR}:=\N\otimes_{\ZZ}\RR$. On note $X(\Sigma)$ la vari\'et\'e torique associ\'ee  \`a l'\'eventail $\Sigma$ et munie de l'action du tore alg\'ebrique $(\KK^{\star})^3=\spec\KK[x,x^{-1},y,y^{-1},z,z^{-1}]$.

Soit  $X_0= \AF_{\KK}^3$.  Une constellation torique de points infiniment voisins de $O$ est un ensemble fini de points $\CO=\{Q_0=O,Q_1,...,Q_{m}\}$, o\`u chaque $Q_{i}$, $1\leq i\leq m$, est une orbite de dimension $0$ de la vari\'et\'e torique $X_i$
obtenue par l'\'eclatement $\varsigma_i:X_{i}\rightarrow X_{i-1}$ de centre $Q_{i-1}$. On note $X(\CO)=X_{m+1}$.

On dit que $Q_j$ se projette sur $Q_i$, not\'e $Q_j\geq Q_i$, si le point $Q_j\in X_j$  est obtenu \`a partir de $Q_i\in X_i$ par une suite d'\'eclatements de points. La relation $\geq$  est une relation d'ordre partiel sur les points de $\CO$. Si $\geq$ est un ordre total, on dit que $\CO$ est une constellation en cha\^ine.\\

Supposons que $\CO$ est une constellation torique en cha\^ine  de points infiniment voisins de $O$. Soient $\Bo:=\{\e_1,\e_2,\e_3\}$  la base ordonn\'ee de $\N$ et  $\Delta:=\langle \Bo\rangle$ le c\^one r\'egulier engendr\'e par $\Bo$. 
Notons $\Sigma_i$ l'\'eventail associ\'e \`a la vari\'et\'e torique $X_i$. L'\'eventail $\Sigma_1$ est obtenu par  la subdivision \'el\'ementaire de $\Delta$ centr\'ee en $u=e_1+e_2+e_3$. Pour chaque entier $j$, $1\leq j\leq 3$, soit $\Bo_j$ la base ordonn\'ee de $N$
obtenue en rempla\c cant $\e_j$ par $u$ en la base $\Bo$. Soit $\Delta_j:=\langle \Bo_j\rangle$ le c\^one r\'egulier engendr\'e par $\Bo_j$, pour $1\leq j \leq 3$. Le choix du point $Q_1>Q_0$  \'equivaut \`a choisir un entier $a_1$, $1\leq a_1\leq 3$, qui  d\'etermine 
un c\^one $\Delta_{a_1}$ de l'\'eventail $\Sigma_1$. La subdivision $\Sigma_2$ de $\Sigma_1$ est obtenue en rempla\c cant $\Delta_{a_1}$ en $\Sigma_1$ par les c\^ones   $\Delta_{a_1j}:=\langle \Bo_{a_1j}\rangle$, o\`u $\Bo_{a_1j}$ est la base ordonn\'ee de $\N$ obtenue en   rempla\c cant le $j$-i\`eme vecteur de $\Bo_{a_1}$ par $\sum_{u\in \Bo_{a_1}}u$. Le choix du point $Q_2>Q_1$  \'equivaut \`a choisir un entier $a_2$, $1\leq a_2\leq 3$, qui d\'etermine  un c\^one $\Delta_{a_1a_2}$ de l'\'eventail $\Sigma_2$. Par r\'ecurrence, on obtient une codification de la constellation en cha\^ine $\CO$. Cette codification est not\'ee $Q_j=Q_0(a_1a_2\cdots a_j)$ pour $1\leq j\leq m$.

Soit   $b_1,b_2,\cdots,b_k$  une suite d'entiers, o\`u  $1\leq b_i\leq 3$, $1\leq i \leq k \leq m$, et telle que le c\^one $\Delta_{b_1b_2\cdots b_k}$ appartient \`a l'\'eventail $\Sigma_{m+1}$. Notons $U_{ b_1b_2\cdots b_k}$ l'ouvert torique affine qui est en correspondance avec le c\^one $\Delta_{b_1b_2\cdots b_k}$.

Soit $\CO_{n}=\{Q_0=O,Q_1,...,Q_{n-1}\}$, $n\geq 1$, une  cha\^ine torique  de points infiniment voisins de $O$ donn\'ee par la codification $Q_j=Q_0(3^j)$ ($j$ fois l'entier $3$) pour $1\leq j\leq n-1$.  On note $\sigma_{n}:X(\CO_n)\rightarrow \AF_{\KK}^3$ le morphisme torique induit par $\CO_n$  et $S_{\mathcal C}$ le transform\'e strict de  $\B{p}{q}$.

\begin{proposition}
\label{pr:const}
 Soient $p>q$ et $p=nq+r$ la division enti\`ere, $1\leq r< q$. Si $r=1$, alors  $S_{\mathcal C}$ est la r\'esolution minimale de $\B{p}{q}$ et la fibre exceptionnelle  $\sigma_{n}^{-1}(O)\cap S_{\mathcal C}$  est la r\'eunion de $nq$ courbes rationnelles.
Si $r> 1$, alors  $S_{\mathcal C}$ a un unique point singulier $s$, et de plus,  $(S_{\mathcal C},s)$ est isomorphe au germe  $(\B{r}{q},O)$.  
 \end{proposition} 
\begin{proof}
La proposition  se d\'emontre par r\'ecurrence sur l'entier $n\geq 1$. Mais d'abord  d\'emontrons le lemme suivant.\\

 On consid\`ere l'\'eclatement de $\AF_{\KK}^3$ de centre $O$, $\varsigma_1:X_1\rightarrow \AF_{\KK}^3$. On note  $x_i$, $y_i$, $ z_i$ les coordonn\'ees canoniques de l'ouvert torique affine $U_i$, pour l'entier $1\leq i\leq 3$.

Soit $S$ le transform\'e strict de $\B{p}{q}$ induit par le morphisme $\varsigma_1:X_1\rightarrow \AF_{\KK}^3$.  Par abus de notation, on note $\varsigma_1:S\rightarrow \B{p}{q}$ la restriction du morphisme $\varsigma_1:X_1\rightarrow \AF_{\KK}^3$ \`a $S$.\\
\begin{lemme}
\label{le:prconst}
 Les ouverts de $S$, $U_1\cap S$ et $U_2\cap S$ sont lisses et $U_3\cap S=\{(x_3,y_3,z_3)\in U_3\mid z_3^{p-q}+\hq{x_3}{y_3}=0\}$.

La fibre exceptionnelle du morphisme  $\varsigma_1:S\rightarrow \B{p}{q}$ est la r\'eunion de $q$ courbes rationnelles. Les courbes ne s'intersectent qu'en le point ferm\'e de $U_3$ fix\'e par l'action du tore alg\'ebrique $(\KK^{\star})^3$.

Si $p-q=1$, alors $S$ est la r\'esolution minimale de  $\B{p}{q}$.
\end{lemme}
\begin{proof}
Les restrictions du morphisme $\varsigma_1:X_1\rightarrow \AF_{\KK}^3$ aux ouverts $U_i$ sont donn\'ees de la fa\c con suivante:

\begin{itemize}
 \item[] $U_1\rightarrow \AF_{\KK}^3$, $(x_1,y_1, z_1)\mapsto (x_1,x_1y_1,x_1z_1)$;
\item[] $U_2\rightarrow \AF_{\KK}^3$, $(x_2,y_2, z_2)\mapsto (y_2x_2,y_2,y_2z_2)$;
\item[] $U_3\rightarrow \AF_{\KK}^3$, $(x_3,y_3, z_3)\mapsto (z_3x_3,z_3y_3,z_3)$.
\end{itemize}
Ainsi, on obtient que 
\begin{itemize}
 \item[] $U_1\cap S=\{(x_1,y_1,z_1)\in U_1\mid z_1^{p}x_1^{p-q}+\hq{1}{y_1}=0\}$;
\item[] $U_2\cap S=\{(x_2,y_2,z_2)\in U_2\mid z_2^{p}y_2^{p-q}+\hq{x_2}{1}=0\}$;
\item[] $U_3\cap S=\{(x_3,y_3,z_3)\in U_3\mid z_3^{p-q}+\hq{x_3}{y_3}=0\}$.
\end{itemize}
Par un calcul direct, on montre que $U_1\cap S$ et $U_2\cap S$ sont lisses. Ceci ach\`eve la preuve de la premi\`ere partie du lemme.\\

Soit $F$ la fibre exceptionnelle du morphisme   $\varsigma_1:S\rightarrow \B{p}{q}$. Alors on a 
\begin{center}
$U_3\cap F=\{(x_3,y_3,z_3)\in U_3\mid z_3=0,\;\hq{x_3}{y_3}=0\}$.
\end{center}
Comme $\fhq$ est un polyn\^ome homog\`ene de degr\'e $q$ sans  facteurs multiples, l'ensemble  $U_3\cap F$ est la r\'eunion de $q$ courbes rationnelles qui s'intersectent en l'origine de $U_3$.

On remarque que les deux ensembles suivants sont de cardinalit\'e $q$.
\begin{itemize}
 \item[]$(X_1\backslash U_3)\cap F\cap U_1 =\{(x_1,y_1,z_1)\in U_1\mid z_1=0,\;x_1=0,\; \hq{1}{y_1}=0\}$;
\item[] $(X_1\backslash U_3)\cap F\cap U_2 =\{(x_2,y_2,z_2)\in U_2\mid z_2=0,\;y_2=0,\; \hq{x_2}{1}=0\}$.
\end{itemize}
Par cons\'equent, la fibre exceptionnelle est la r\'eunion d'exactement $q$ courbes rationnelles qui s'intersectent en l'origine de $U_3$. Ceci ach\`eve la preuve de la deuxi\`eme parti du lemme.\\

Si $p-q=1$,  $S$ est une surface lisse. On  consid\`ere la fonction r\'eguli\`ere de $\AF^{3}_{\KK}$, $g(x,y,z)=x$. Alors $g^{\star}=g\circ \varsigma_1$ est une fonction r\'eguli\`ere de $X_1$. 

Soit $C$ le diviseur principal  de $\B{p}{q}$  associ\'e \`a la restriction de $g$ \`a $\B{p}{q}$, c'est-\`a-dire $C:=\Div{g\mid_{\B{p}{q}}}$. On remarque que  $C$  est une courbe (l'intersection sch\'ematique de $\Div{g}$ et $\B{p}{q}$ est irr\'eductible et r\'eduit). Notons $C'$ le transform\'e strict de $C$  par  le morphisme  $\varsigma_1:S\rightarrow \B{p}{q}$. Par un calcul direct, on peut montrer que $C'$ intersecte $F$ en l'origine de $U_3$. 

On consid\`ere la  restriction de  $g^{\star}$ \`a l'ouvert $U_3$. Alors, on a

\begin{center}
 $g^{\star}=z_3x_3$.
\end{center}
 Par cons\'equent, le diviseur principal  de  $S$ associ\'e \`a  la restriction de  $g^{\star}$ \`a $S$ est le suivant:
\begin{center}
$\Div{g^{\star}\mid_{S}}=C'+\sum_{i=1}^{q}F_i$,
 \end{center}
o\`u les $F_i$ sont les composantes  irr\'eductibles de $F$. Comme $\Div{g^{\star}\mid_{S}}\cdot F_i=0$ pour tout $i\in\{1,2,...,q\}$, on obtient:
\begin{center} 
 $F_i\cdot F_i=-C'\cdot F_i-\sum_{j\neq i}F_j\cdot F_i=-q\leq -3$.
\end{center}
  Ceci ach\`eve la preuve du lemme.
\end{proof}
Raisonnons par r\'ecurrence sur l'entier $n$.\\

 Si $n=1$, on a $p=q+r$ et $\CO_1=\{Q_0=O\}$. La preuve de ce cas r\'esulte du Lemme \ref{le:prconst}.\\

Maintenant, on suppose que $p=nq+r$.\\

  D'apr\`es le Lemme \ref{le:prconst}, on a  $U_3\cap S=\B{(n-1)q+r}{q}$ et la fibre exceptionnelle de $\varsigma_1:S\rightarrow \B{p}{q}$ est la r\'eunion de  $q$ courbes rationnelles.  
En appliquant l'hypoth\`ese de r\'ecurrence sur $U_3\cap S$, on montre:
\begin{itemize}
 \item[]la fibre exceptionnelle de $\sigma_n:S_{\mathcal{C}}\rightarrow \B{p}{q}$ est la r\'eunion de $nq$ courbes rationnelles;
\item[] si $r>1$, alors  $S_{\mathcal{C}}$  a un unique point singulier de type $\B{r}{q}$;
\item[] si $r=1$,  alors  $S_{\mathcal{C}}$ est lisse.\\
\end{itemize}

Pour achever la preuve de la proposition, il faut montrer que si $r=1$, alors le morphisme $\sigma_n:S_{\mathcal{C}}\rightarrow \B{p}{q}$ est la r\'esolution minimale de $\B{p}{q}$.\\

Soient $F$ la fibre exceptionnelle du morphisme $\varsigma_1:S\rightarrow \B{p}{q}$ et $F_i$ une composante irr\'eductible de $F$, $1\leq i\leq q$.
On note $F'_i$ le transform\'e strict de $F_i$ dans $S_{\mathcal C}$ (on rappelle que le morphisme $\varsigma_1$ factorise $\sigma_n$).
En utilisant l'hypoth\`ese de r\'ecurrence,  pour montrer que   $S_{\mathcal{C}}$  est la r\'esolution minimale de $\B{p}{q}$, il suffit de montrer que $F'_i\cdot F'_i\leq -2$ pour tout $1\leq i\leq q$.\\

Pour chaque $Q_i\in \CO_n$, on note $\Bl_i$ le diviseur exceptionnel $\varsigma_i^{-1}(Q_{i})$ en $X_{i+1}$ et $\Dl_i$ (resp. $\Dl_i^{\star}$) le transform\'e strict (resp. total) de $\Bl_i$ dans $X(\CO_n)$.\\

 On rappelle que la constellation $\CO_n$ est donn\'ee par la codification $Q_j=Q_{0}(3^j)$ pour $1\leq j \leq n-1$. En vertu du lemme $1.3$ de \cite{CGL96}, on a 
\begin{center}
$\Dl_i=\Dl_i^{\star}- \Dl_{i+1}^{\star}$,
\end{center}
 d'o\`u $\Dl_i^{\star}=\sum_{j=i}^{n-1}\Dl_j$.\\

Avant d'achever la preuve de la Proposition \ref{pr:const}, on introduit les notions suivantes.\\

Soient $\mathcal{I}$, $\mathcal{J}$, $\mathcal{P}$ trois id\'eaux.  Le $\star$-produit  de $\mathcal{I}$ et  $\mathcal{J}$, not\'e  $\mathcal{I}\star\mathcal{J}$, est la cl\^oture int\'egrale  du produit $\mathcal{I}\mathcal{J}$. On suppose que $\mathcal{P}$ est un id\'eal complet  non-trivial.  L'id\'eal $\mathcal{P}$   est $\star$-simple si $\mathcal{P}$ n'a pas de $\star$-factorisation non-triviale.

On suppose que $\mathcal{I}$ est un id\'eal tel que $\mathcal{I}\mathcal{O}_{X(\CO_{n})}$ soit un faisceau d'id\'eaux inversible. 
On d\'efinit  par r\'ecurrence le vecteur $\m(\mathcal{I}) =(\m_0,...,\m_{n-1})$:  $\mathcal{I}_{0}=\mathcal{I}$, $\m_0=\ord_{Q_0}\mathcal{I}_0$ et $\m_{i}=\ord_{Q_i} \mathcal{I}_{i}$, $1\leq i\leq n-1$,   o\`u  $\mathcal{I}_{i}=x^{-\m_{i-1}}\mathcal{I}_{i-1}\mathcal{O}_{X_{i},Q_{i}}$ et $x=0$ est l'\'equation locale de $\Bl_{i-1}$ en $Q_i$. 
Dans \cite{Lip88} l'auteur montre qu'il existe un unique id\'eal $\star$-simple, not\'e $\mathcal{P}_{Q_{n-1}}$, tel que le vecteur 
$\m(\mathcal{P}_{Q_{n-1}})$ est minimal pour l'ordre lexicographique inverse et $\m_{n-1}=1$. L'id\'eal 
 $\mathcal{P}_{Q_{n-1}}$  est appel\'e {\it l'id\'eal  $\star$-simple sp\'ecial} associ\'e \`a $Q_{n-1}$.\\

Reprenons la d\'emonstration de la Proposition \ref{pr:const}.\\

 Soient $\mathcal{P}_{Q_{n-1}}$ {\it l'id\'eal  $\star$-simple sp\'ecial} associ\'e \`a $Q_{n-1}$ et  $\Dl'$  le diviseur de $X(\CO_n)$ tel que 

\begin{center}
 $\mathcal{P}_{Q_{n-1}}\mathcal{O}_{X(\CO_n)}=\mathcal{O}_{X(\CO_n)}(-\Dl')$.
\end{center}

D'apr\`es le Lemme $2.16$ de \cite{CGL96}, on a 
 \begin{center}
$\Dl'=\sum_{i=0}^{n-1}\Dl_i^{\star}$.
\end{center}
Par cons\'equent, on a 
 \begin{center}
$\Dl'=\sum_{i=0}^{n-1}(i+1)\Dl_i$.
\end{center}

Soit  $g$ un \'el\'ement g\'en\'eral de l'id\'eal $\mathcal{P}_{Q_{n-1}}$. Alors $g^{\star}=g\circ\sigma_n$ est une fonction r\'eguli\`ere de $X(\CO_n)$. On note  $C$ le diviseur associ\'e \`a $g$ et  $C'$  le transform\'e strict  de $C$ dans $X(\CO_n)$.
Alors, on a:

\begin{center}
$\Div{g^{\star}}=C'+\sum_{i=0}^{n-1}(i+1)\Dl_i$.
\end{center}

Pour un diviseur $Z$ de $X(\CO_n)$, tel que son support ne contienne pas  $S_{\mathcal{C}}$, notons  $Z\cdot S_{\mathcal{C}}$ le diviseur de $S_{\mathcal{C}}$ obtenu par la somme formelle des composantes irr\'eductibles  de $Z\cap S_{\mathcal{C}}$ pond\'er\'ees par leurs multiplicit\'es en $S_{\mathcal{C}}$.  Comme $g$ est un \'el\'ement g\'en\'eral de $\mathcal{P}_{Q_{n-1}}$ on peut supposer que le support de $C'$ ne contient pas $S_{\mathcal{C}}$. Ainsi, on obtient que 
\begin{center}
$\Div{g^{\star} \mid_{S_{\mathcal{C}}}}=C'\cdot S_{\mathcal{C}}+\sum_{i=0}^{n-1}(i+1)\Dl_i\cdot S_{\mathcal{C}}$.
\end{center}

On remarque que $\D_0\cdot S_{\mathcal{C}}= \sum_{i=1}^{q}F'_i$. En effet,  il suffit de consid\'erer l'ouvert affine $U_1$. On a donc
$U_1\cap S_{\mathcal{C}}=\{(x_1,y_1,z_1)\in U_1\mid z_1^{p}x_1^{p-q}+\hq{1}{y_1}=0\}$ et $\D_0\cap U_1= \{(x_1,y_1,z_1)\in U_1\mid x_1=0\}$. Par cons\'equent,  $(D_0\cdot S_{\mathcal{C}})\cap U_1= \sum_{i=1}^{q}(F'_i\cap U_1)$, d'o\`u $\D_0\cdot S_{\mathcal{C}}= \sum_{i=1}^{q}F'_i$.\\

On a donc 

\begin{center}
$\Div{g^{\star} \mid_{S_{\mathcal{C}}}}=C'\cdot S_{\mathcal{C}}+\sum_{i=1}^{n-1}(i+1)\Dl_i\cdot S_{\mathcal{C}} +\sum_{i=1}^{q}F'_i$. 

\end{center}

 On remarque que l'intersection $\Div{g^{\star} \mid_{S_{\mathcal{C}}}}\cdot F'_i$ est nulle  pour tout $1\leq i\leq q$, car $\Div{g^{\star} \mid_{S_{\mathcal{C}}}}$ est un diviseur principal. On remarque aussi que  $(\sum_{i=1}^{n-1}(i+1)\Dl_i\cdot S_{\mathcal{C}})\cdot F'_i\geq 2$, car il existe au moins un $ 1\leq i\leq n-1$ tel que 
$((i+1)\Dl_i\cdot S_{\mathcal{C}})\cdot F'_i\geq (i+1)\geq 2$. Par cons\'equent, on obtient que $F'_i\cdot F'_i\leq -2$ pour tout $1\leq i\leq q$, d'o\`u la proposition.
\end{proof}

Pour un polyn\^ome  $\g=\sum c_{e}x^{e_1}y^{e_2}z^{e_3}$ dans  $\KK[x,y,z]$, o\`u $e=(e_1,e_2,e_3)$ et  $c_{e} \in \KK$, on note  $\mathcal{E}(\g)$ l'ensemble des exposants $e\in \ZZ^{3}_{\geq 0}$  dont le  coefficient  $c_{e}$ est non nul, c'est-\`a-dire 
$\mathcal{E}(\g):=\{e\in\ZZ^{3}_{\geq 0}\mid c_e\neq 0\}$.
Soient  $\Gamma_{+}(\g)$ l'enveloppe convexe de l'ensemble $\{e+\RR^{3}_{\geq 0}\mid e\in \mathcal{E}(\g)\}$ et $\Gamma(\g)$  la r\'eunion des faces compactes de  $\Gamma_{+}(\g)$.
On note $\mathcal{I}(\g)$ l'id\'eal monomial  de $\KK[x,y,z]$ engendr\'e par les mon\^omes $x^{e_1}y^{e_2}z^{e_3}$ tels que $(e_1,e_2,e_3)\in \Gamma(\g)\cap \ZZ^3$, c'est-\`a-dire $\mathcal{I}(\g):=(\{x^{e_1}y^{e_2}z^{e_3}\mid (e_1,e_2,e_3)\in \Gamma(\g)\cap \ZZ^3\})$.
L'\'eventail de Newton  $\EN{\g}$ associ\'e \`a $\g$ est la subdivision de $\RR^{3}_{\geq 0}$ correspondant \`a l'\'eclatement normalis\'e de $\AF_{\KK}^3$ de centre l'id\'eal $\mathcal{I}(g)$. Pour plus de d\'etails voir \cite{GoLe91}  ou \cite{KKMS73}.\\
\begin{remarque}
\label{re:x-et-y-no-div-h}
 Dans toute la suite, \`a  automorphisme  lin\'eaire de $\AF^3_{\KK}$  pr\`es, $x$ et $y$ ne divisent pas $\hq{x}{y}$. 
\end{remarque}
 La Figure \ref{fi:PN} repr\'esente le poly\`edre de Newton $\Gamma(\f)$ associe \`a $\f:=z^p+\hq{x}{y}$.
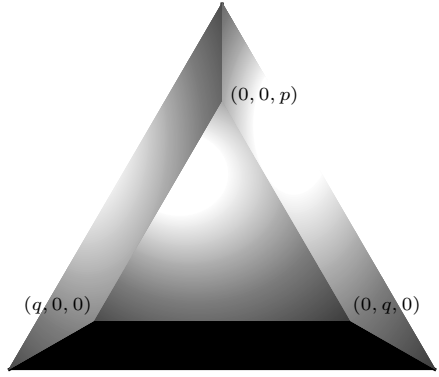
\begin{figure}[!h]
\begin{tikzpicture}[x=0.65cm,y=0.65cm ]

\tikzfading[name=fade inside,left color=transparent!10,right color=transparent!80, inner color=transparent!50,outer color=transparent!30]
\tikzfading[name=fade outside,outer color=transparent!80, inner color=transparent!30]
\tikzfading[name=fade bottom,
left color=transparent!50, 
right color=transparent!50,top color=transparent!50, bottom color=transparent!10]

\tikzfading[name=fade left,
left color=transparent!10, 
right color=transparent!100,top color=transparent!80, bottom color=transparent!30]

 \draw[very thick] (330:3)  -- (330:5);
 \draw[very thick] (90:3)  -- (90:5);
 \draw[very thick] (210:3)  -- (210:5) ;
   \draw[very thick] (330:3) --  (90:3)  -- (210:3) -- (330:3);
 
\shade[ball color=white, path fading=fade inside]  (90:3)  -- (330:3) -- (210:3); 
\fill[path fading=fade bottom] (210:3) --(330:3) -- (330:5)   -- (210:5) ;
\shade[ball color=white, path fading=fade left] (210:3) --  (90:3) -- (90:5) -- (210:5) ;
\shade[ball color=white, path fading=fade inside] (90:3) -- (330:3)  --  (330:5)  --  (90:5);

\draw (342:3.5) node[inner sep=-1pt,below=-1pt] {{\tiny $(0,q,0)$}};

\draw (198:3.5) node[inner sep=-1pt,below=-1pt] {{\tiny $(q,0,0)$}};

\draw (75:3.3) node[inner sep=-1pt,below=-1pt] {{\tiny $(0,0,p)$}};
 \end{tikzpicture}
\caption{Poly\`edre de Newton $\Gamma(\f)$}
\label{fi:PN}
\end{figure}

Soit $H$ un plan de $\RR^3$ qui ne contient pas l'origine de $\RR^3$ et tel que l'intersection de $H$ et $\RR^3_{\geq 0}$ soit un ensemble compact.  La Figure \ref{fi:EN}  repr\'esente l'intersection de $H$ avec la subdivision $\EN{\f}$ de $\RR^3_{\geq 0}$. Chaque sommet du diagramme est identifi\'e avec le {\it vecteur extr\'emal} (autrement dit, vecteur primitif d'un c\^one de dimension $1$ de l'\'eventail $\EN{\f}$) correspondant.  On note $\tau_1$ (resp. $\tau_2$, $\tau_3$)  le c\^one engendr\'e par les vecteurs $(1,0,0)$ (resp. $(0,1,0)$, $(0,0,1)$) et $(p,p,q)$. 

\begin{figure}[!h]

\begin{tikzpicture}[font=\small]

   \draw[thick] (0,0)  -- (330:3)  (0,0)  -- (90:3)  (0,0)  -- (210:3);
   \draw[thick] (330:3) --  (90:3)  -- (210:3) -- (330:3);
\draw (0,-0.4) node[inner sep=-1pt,below=-1pt,rectangle,fill=white] {{\tiny $(p,p,q)$}};
\draw (335:3.5) node[inner sep=-1pt,below=-1pt,rectangle,fill=white] {{\tiny $(0,1,0)$}};
\draw (340:1.5) node[inner sep=-1pt,below=-1pt,rectangle,fill=white] {$\tau_2$};
\draw (205:3.5) node[inner sep=-1pt,below=-1pt,rectangle,fill=white] {{\tiny $(1,0,0)$}};
\draw (200:1.5) node[inner sep=-1pt,below=-1pt,rectangle,fill=white] {$\tau_1$};
\draw (90:3.3) node[inner sep=-1pt,below=-1pt,rectangle,fill=white] {{\tiny $(0,0,1)$}};
\draw (80:1.5) node[inner sep=-1pt,below=-1pt,rectangle,fill=white] {$\tau_3$};
    \draw[fill=black]  (0,0) circle(0.7mm) (330:3) circle(0.7mm)  (90:3) circle(0.7mm)  (210:3) circle(0.7mm);

  \end{tikzpicture}

\caption{\'Eventail de Newton $\EN{f}$}
\label{fi:EN}
\end{figure}
La  proposition suivante r\'esulte d'un calcul direct.
\begin{proposition}
\label{pr:tau12}
 Le c\^ones $\tau_1$ et $\tau_2$ sont r\'eguliers.
\end{proposition}

 Soit  $\GD{\f}$ une $G$-{\it subdivision r\'eguli\`ere}  de  $\EN{\f}$, c'est-\`a-dire une subdivision r\'eguli\`ere de chaque c\^one $\tau\in\EN{\f}$  n'ayant comme ar\^etes que celles qui portent les vecteurs du syst\`eme g\'en\'erateur minimal du semi-groupe $\tau\cap \ZZ^3$. D'apr\`es \cite{BoGo95}, cette subdivision existe.\\

 On note  $\pi_{\mathcal{N}}:X(\EN{\f})\rightarrow \AF^{3}_{\KK}$  (resp. $\pi_{\mathcal{G}}:X(\GD{\f})\rightarrow X(\EN{\f})$)  le morphisme torique induit par la subdivision $\EN{\f}$ de $\RR^3_{\geq 0}$ (resp. $\GD{\f}$ de  $\EN{\f}$) et $S_{\mathcal{G}}$ le transform\'e strict de $\B{p}{q}$ associ\'e au morphisme $\pi:=\pi_{\mathcal{G}}\circ \pi_{\mathcal{N}}$
Par abus de notation, on note $\pi:S_{\mathcal{G}}\rightarrow \B{p}{q}$ la restriction du morphisme $\pi:X(\GD{\f})\rightarrow \AF_{\KK}^3$ \`a $S_{\mathcal{G}}$.\\

Soit $k$, un entier $k\geq 1$.  Pour un ensemble d'entiers $m_i\geq 2$,  $1\leq i\leq k$, on note 
$[m_1;m_2;\cdots ;m_k]$ la fraction continue d\'efinie de la fa\c con suivante:

 \begin{center}
$[m_k]:=m_k$, $[m_{k-1};m_{k}]:= \displaystyle m_{k-1}-\frac{1}{m_k}$ et $[m_1;m_2;\cdots ;m_k]:  =\displaystyle m_1 - \frac{1}{\displaystyle [m_2;\cdots ;m_k]}$.
\end{center}
La proposition suivante est une application directe du Th\'eor\`eme $6.1$ de \cite{Oka87}.
\begin{proposition}
\label{pr:Gdes1}
 $S_{\mathcal{G}}$ est une bonne r\'esolution de  $\B{p}{q}$ et son graphe dual pond\'er\'e  est une \'etoile \`a $q$ branches identiques.
Le diagramme de chaque branche est le suivant: 
\begin{center}
\begin{tikzpicture}[x=0.37cm, y=0.37cm]
 \draw (0:6.1) node[inner sep=0pt,below=2pt,rectangle,fill=white] {$ \tiny{ -m_{2}}$} 
(0:3) node[inner sep=0pt,below=2pt,rectangle,fill=white] {$ {\small -m_{k}}$};

\draw (0:9.1) node[inner sep=0pt,below=2pt,rectangle,fill=white] {$ \tiny{ -m_{1}}$};

 \draw (180:1) node[inner sep=0pt,below=-1pt,rectangle,fill=white] {$ \tiny{ E}^{2}_{0}$};

 \draw[very thick] (0, 0) -- (0:4) (0:4.35) -- (0:4.4) (0:4.75) -- (0:4.8) (0:5.15) -- (0:5.2)  (0:5.6) --(0:9.1);

\draw[fill=black]  (0,0) circle(0.7mm) (0:6.1) circle(0.7mm) (0:3) circle(0.7mm)  (0:9.1) circle(0.7mm);
   \end{tikzpicture},
\end{center}
 o\`u $\E_0$ le diviseur associ\'e au sommet  central du graphe et les entiers $m_i\geq 2$ (resp. l'entier $k$) sont d\'efinis (resp. est d\'efini) de la fa\c con suivante:
\begin{itemize}
 \item[] si $q>p$ et $q=np+r$ la division enti\`ere, $1\leq r<p$,  alors on a 
\begin{center}
 $\displaystyle\frac{p}{p-r}=[m_1;m_2;\cdots ;m_k]$;
\end{center}

\item[] si $p>q$ et  $p=nq+r$, la division enti\`ere $1\leq r<q$ (resp. $n\geq 1$), alors on a 
\begin{center}
 $\displaystyle\frac{p}{p-q}=[m_1;m_2;\cdots ;m_k]$.\end{center}

\end{itemize}
  
De plus, cette r\'esolution est minimale si et seulement si $p\not \equiv 1 \mod{q}$.
\end{proposition}

\begin{corollaire}
\label{co:Gdes.r=1}
 Si $p\equiv 1\mod{q}$, seul le diviseur $\E_0$ n'est pas un diviseur essentiel sur $\B{p}{q}$.
\end{corollaire} \begin{proof} Soit $n\geq 1$ tel que $p=nq+1$.
En vertu de la Proposition \ref{pr:Gdes1}, on a 
 \begin{center}
$\displaystyle \frac{p}{p-q} = \displaystyle \frac{nq +1 }{(n-1)q+1}  =\displaystyle 2 - \frac{1}{\displaystyle 2 -\frac{\ddots}{\displaystyle 2-\frac{1}{\displaystyle  q+1}}}.$

\end{center}

En particulier l'entier $k$ est \'egal \`a $n$, d'o\`u le graphe dual de la r\'esolution  $\pi:\SD \rightarrow \B{p}{q}$ a $nq +1$ sommets. 
D'apr\`es la Proposition  \ref{pr:const}, il n'y a que un diviseur de la fibre exceptionnelle de $\pi$ qui n'est pas un diviseur essentiel. D'apr\`es {\it le crit\`ere de contraction de Castelnuovo}, ce diviseur a une auto-intersection \'egale \`a $-1$. 

Comme les branches du  graphe dual de la r\'esolution $\pi:\SD \rightarrow \B{p}{q}$  sont identiques, forcement le diviseur 
$\E_0$   a une  auto-intersection \'egal \`a $-1$, d'o\`u le corollaire.  
 \end{proof}

Un polyn\^ome $\g=\sum c_{e}x^{e_1}y^{e_2}z^{e_3}$, o\`u $e=(e_1,e_2,e_3)$ et  $c_{e} \in \KK$, est appel\'e {\it non}-{\it d\'eg\'en\'er\'e par rapport \`a la fronti\`ere de Newton} si pour toute face compacte $\gamma$ de $\EN{\g}$ le polyn\^ome $\g_{\gamma}:=\sum_{e\in \gamma}c_ex^{e_1}y^{e_2}z^{e_3}$ est non singulier sur le tore $T:=\N\oplus_{\ZZ}\KK$, c'est-\`a-dire les polyn\^omes $\g_{\gamma}$, $\partial_x\g_{\gamma}$, $\partial_y\g_{\gamma}$, $\partial_z\g_{\gamma}$ n'ont pas de z\'ero commun en dehors  de l'ensemble $xyz=0$.\\

Dans la proposition suivante on suppose que $S$ est une hypersurface normale de $\AF^{3}_{\KK}$, donn\'ee par l'\'equation $\g=0$, o\`u $\g$ est un polyn\^ome irr\'eductible 
non-d\'eg\'en\'er\'e par rapport \`a la fronti\`ere de Newton $\Gamma(\g)$. De plus, on suppose  que $O$ est l'unique point singulier de $S$.\\

On consid\`ere une $G$-subdivision r\'eguli\`ere $\GD{\g}$ de l'\'eventail de Newton $\EN{g}$. On note $\pi':X(\GD{\g})\rightarrow \AF^{3}_{\KK}$ le morphisme torique induit par la subdivision $\GD{\g}$ du c\^one $\RR^{3}_{\geq 0}$ et  $X$ le transform\'e strict de $S$ dans $X(\GD{\g})$. Par abus de notation, on note $\pi':X\rightarrow S$ la restriction du morphisme $\pi':X(\GD{\g})\rightarrow \AF_{\KK}^3$ \`a $X$.\\  

 La proposition suivante    r\'esulte  des  Lemmes $10.2$ et $10.3$ de \cite{Var76} (pour avoir plus de d\'etails, voir \cite{Mer80})  ou du Th\'eor\`eme principal et de la Remarque {\it a}) de la Section $4$  de \cite{GoLe91}.\\

\begin{proposition}
\label{pr:cr_nor-g}  Le morphisme $\pi':X\rightarrow S$ est une d\'esingularisation de $S$. 

 Si l'hypersurface $S$ ne contient pas de $T$-orbite de dimension $1$, alors le morphisme $\pi':X(\GD{\g})\rightarrow \AF^3_{\KK}$ est une r\'esolution plong\'ee de $S$, c'est-\`a-dire $\pi':X(\GD{\g})\rightarrow \AF_{\KK}^3$ est un morphisme propre et birationnel, $\pi':X(\GD{\g})\backslash(\pi')^{-1}(O)\rightarrow \AF_{\KK}^3\backslash \{O\}$  est un isomorphisme et $(\pi')^{-1}(S)$ est un diviseur \`a croisements normaux.
\end{proposition}
On remarque que le polyn\^ome $\f:=z^p+\hq{x}{y}$ est non-d\'eg\'en\'er\'e par rapport \`a la fronti\`ere de Newton et que  l'hypersurface $\B{p}{q}$ ne contient pas de $T$-orbite de dimension $1$ (voir la Remarque \ref{re:x-et-y-no-div-h}).

\begin{corollaire}
\label{co:cr_nor}
le morphisme $\pi:X(\GD{\f})\rightarrow \AF^3_{\KK}$ est une r\'esolution plong\'ee de $\B{p}{q}$

\end{corollaire}

La proposition suivante \'etablit une relation entre le morphisme $\pi:\SD\rightarrow \B{p}{q}$ et le morphisme  $\sigma_n:S_{\mathcal C}\rightarrow \B{p}{q}$  (voir les Propositions \ref{pr:const}  et \ref{pr:Gdes1}).

\begin{proposition}
\label{pr:Gdes2}
Si $p>q$, le morphisme   $\sigma_n:S_{\mathcal C}\rightarrow \B{p}{q}$  factorise le morphisme  $\pi:\SD\rightarrow \B{p}{q}$, c'est-\`a-dire il existe un morphisme $\pi_0:\SD\rightarrow S_{\mathcal C}$ tel que $\pi=\sigma_n \circ \pi_0$.
\end{proposition}
 \begin{proof} 
Dans la preuve de cette proposition, on peut appliquer le Lemme  \ref{le:prconst} car ses hypoth\`eses sont v\'erifi\'ees.

On remarque que 
\begin{center}
$(1,1,1)=\dfrac{(p,p,q)+(p-q)(0,0,1)}{p}$
\end{center} 
et que le c\^one engendr\'e par les vecteurs $(0,0,1)$ et $(1,1,1)$ est r\'egulier. Par cons\'equent, le vecteur $(1,1,1)$ appartient au syst\`eme g\'en\'erateur minimal du semi-groupe $\tau_3\cap\ZZ^3$, o\`u $\tau_3$ est le c\^one engendr\'e par les vecteurs $(0,0,1)$ et $(p,p,q)$ (voir la Figure \ref{fi:EN}).\\

On remarque aussi que l'\'eventail  obtenu par l'\'eclatement de Newton de $\AF_{\KK}^3$ associ\'e \`a $\f=z^p +\hq{x}{y}$ suivi de la subdivision \'el\'ementaire centr\'ee en $(1,1,1)$, co\"incide avec celui obtenu par l'\'eclatement de $\AF_{\KK}^3$  de centre le point $O$  suivi de l'\'eclatement de Newton de $U_3$ associ\'e \`a $z_3^{p-q} +\hq{x_3}{y_3}$ (voir le Lemme \ref{le:prconst}).\\

 Soit  $p=nq+r$ la division enti\`ere, $1\leq r<q$. En utilisant la remarque ci-dessus et  le Lemme \ref{le:prconst}, la proposition r\'esulte d'une r\'ecurrence sur l'entier $n\geq 1$. \end{proof}

Dans la proposition suivante, on suppose que $S$ est une hypersurface  quasi-homog\`ene de $\AF^{3}_{\KK}$ donn\'ee par l'\'equation $\g=0$, o\`u $\g$ est un polyn\^ome quasi-homog\`ene et irr\'eductible,
que  $O$ est l'unique point singulier de l'hypersurface $S$ et que $S$ ne contient pas de $T$-orbite de dimension $1$.\\

D'apr\`es la proposition \ref{pr:cr_nor-g}, le morphisme $\pi':X(\GD{\g}) \rightarrow \AF^{3}_{\KK}$ est une r\'esolution plong\'ee de  $S$. On note  $X$ le transform\'e strict de $S$ dans $X(\GD{\g})$. Par abus de notation, on note $\pi':X\rightarrow S$ la restriction du morphisme $\pi':X(\GD{\g})\rightarrow \AF_{\KK}^3$ \`a $X$.\\

Pour $\rho$ un {\it vecteur extr\'emal} de  $\GD{\g}$ (c'est-\`a-dire $\rho$ est un vecteur primitif d'un c\^one de dimension  $1$ qui appartient \`a l'\'eventail $\GD{\g}$), on note $\D_{\rho}$  l'orbite ferm\'ee associ\'ee \`a $\rho$. Notons  $E\GD{\f}$ l'ensemble des vecteurs extr\'emaux de $\GD{\g}$, et $S_{2}\EN{g}$ le $2$-squelette de $\EN{g}$ (c'est-\`a-dire $S_{2}\EN{g}$ est la r\'eunion des c\^ones de dimension $2$ qui appartiennent \`a l'\'eventail $S_{2}\EN{g}$).\\

La proposition suivante donne des \'equations locales pour les composantes irr\'eductibles de la fibre exceptionnelle de la d\'esingularisation  $\pi':X\rightarrow S$. 

Dans le cas $S=\B{p}{q}$,  on rappelle que  $\E_0$ est le diviseur associ\'e au sommet central du graphe induit par $\pi$.

\begin{proposition}\label{pr:sygemi} Soit $\rho$ un vecteur extr\'emal de $\GD{\g}$ ($\rho \in E\GD{\g}$). Alors, on a:
\begin{itemize}
\item[i)]l'intersection 
 $\D{\rho}\cap X$ n'est pas vide si et seulement si  $\rho$ appartient au 2-squelette de l'\'eventail de Newton $\EN{g}$, c'est-\`a-dire  $\rho\in S^{2}\EN{\g}\cap E\GD{\g}$;
\item[ii)]les composantes irr\'eductibles de $\D{\rho}\cap X$ sont diviseurs exceptionnels du morphisme $\pi':X\rightarrow S$  
 si et seulement si  de plus $\rho\in \ZZ_{>0}^3$;
\item[iii)] une  composante irr\'eductible $\E$ de la fibre exceptionnelle de $\pi'$ \'etant donn\'ee, il existe un unique $\rho\in E\GD{\g}$ tel que $\E\subset \D{\rho}$;
\item[iv)] si $S=\B{p}{q}$, alors $\E_0=\D{(p,p,q)}\cap\SD$;   
 \item[v)]  si $S=\B{p}{q}$, alors l'ensemble form\'e par le vecteur $(0,0,1)$ et les vecteurs  $\rho \in E\GD{\f}$  tels que les composantes irr\'eductibles de $\D{\rho}\cap \SD$ sont diviseurs exceptionnels de $\pi$ est le syst\`eme g\'en\'erateur minimal du semi-groupe $\tau\cap \ZZ^3$, o\`u $\tau$ est le c\^one engendr\'e par les vecteurs  $(0,0,1)$ et $(p,p,q)$.
\end{itemize}
\end{proposition}
\begin{remarque} 
Les r\'esultats de l'article \cite{Oka87} reposent sur la construction d'une subdivision r\'eguli\`ere $\Sigma$ de l'\'eventail $\EN{\g}$. Cette construction est longue \`a d\'efinir, or quand le polyn\^ome $\g$ est quasi-homog\`ene, on peut supposer que $\Sigma$ est une $G$-subdivision r\'eguli\`ere de l'\'eventail $\EN{\g}$.
\end{remarque}
 \begin{proof} Les points {\it i),ii),iii)} et {\it iv)} de la proposition r\'esultent de la Remarque $4.3$ et du Lemme $4.7$  de \cite{Oka87}. Le point {\it v)} r\'esulte des points {\it i)}, {\it ii)} et de la Proposition \ref{pr:tau12}.
 \end{proof}

\subsection{Preuve du Th\'eor\`eme 1.1.}
Dans cette section, on montre la bijectivit\'e de l'application de Nash pour les hypersurfaces quasi-rationnelles $\B{p}{q}$, ce qui \'equivaut \`a montrer que tous les wedges admissibles se rel\`event \`a la r\'esolution minimale de $\B{p}{q}$ (voir \cite{Reg06}).  
Notre but, dans toute la suite de cette section, est de montrer que pour chaque diviseur essentiel $\E$  ($\E\in \Ess(\B{p}{q})$) tous les $K$-wedges admissibles centr\'es en $N_{\E}$  se rel\`event \`a la r\'esolution minimale de $\B{p}{q}$. De plus, on profite de d\'emontrer quelques r\'esultats qu'on utilise dans toutes les sections de cet article.\\

 Avec le th\'eor\`eme suivant on r\'eduit le nombre de cas \`a \'etudier. Une preuve de ce r\'esultat, dans le cas des singularit\'es de surfaces  rationnelles, se trouve dans \cite{Ple05}.

\begin{theoreme}
\label{th:faremi}
 Soient $V$ une surface alg\'ebrique normale sur  $\KK$ et $\pi:Y\rightarrow V$ la r\'esolution minimale de $V$. Supposons  qu'il existe un morphisme propre et birationnel $\pi':V'\rightarrow V$, o\`u $V'$ est une surface alg\'ebrique normale sur $\KK$, tel que $\pi'$ factorise $\pi$. Alors, si l'application de Nash $\NC_{V}$ associ\'ee \`a  $V$ est bijective,  l'application de Nash $\NC_{V'}$ associ\'ee \`a $V'$  l'est aussi.  
\end{theoreme}
\begin{proof} On remarque que $Y$ est la r\'esolution minimale de $V'$.
Si $\omega$ est un $K$-wedge sur $V'$, alors  $\pi'\circ \omega$ est un $K$-wedge admissible sur $V$. Par cons\'equent,  $\pi'\circ \omega$ se rel\`eve \`a $Y$, d'o\`u le Th\'eor\`eme.
\end{proof}

En vertu des r\'esultats \ref{pr:const}, \ref{pr:Gdes2} et \ref{th:faremi},  on a le corollaire suivant.
\begin{corollaire}
 Si l'application de Nash $\NC_{\B{p}{q}}$ associ\'ee \`a l'hypersurface $\B{p}{q}$ est bijective pour tous les entiers $p>q\geq 3$, premiers entre eux,  alors l'application de Nash  $\NC_{\B{p}{q}}$ est bijective pour tous les entiers  $p\geq 2$ $q\geq 2$, premiers entre eux.
\end{corollaire}

\begin{remarque}
\label{re:p>q}
Dans toute la suite, on suppose que  $p>q\geq 3$.
\end{remarque}

Dans la proposition suivante, $S$ d\'esigne l'hypersurface de la Proposition \ref{pr:sygemi}, $\E$ d\'esigne un diviseur essentiel sur $S$   et  $\alpha_{\E}$ d\'esigne le point g\'en\'erique de $N_{\E}$. On pose 
\begin{center} $(\mu_x,\mu_y,\mu_z):=(\ordt{\com{\alpha_{\E}}}(x),\ordt{\com{\alpha_{\E}}}(y), \ordt{\com{\alpha_{\E}}}(z))\in \ZZ^3_{>0}$.\end{center}
\begin{proposition}
\label{pr:musygemi-g}
 Le vecteur $(\mu_x,\mu_y,\mu_z)$  appartient \`a l'intersection de l'ensemble  $E\GD{g}$ des extr\'emaux de $\GD{g}$ avec le $2$-squelette $S_2\EN{g}$ de l'\'eventail de Newton $\EN{g}$, c'est-\`a-dire  $(\mu_x,\mu_y,\mu_z)\in E\GD{g}\cap S_2\EN{\g}$.
\end{proposition}
\begin{proof} Dans la d\'emonstration, on utilise les notations usuelles  de vari\'et\'es toriques (voir \cite{KKMS73}).

Soit  $\E\in \Ess(S)$.  En vertu de la Proposition  \ref{pr:sygemi}, il existe un unique vecteur extr\'emal $\rho_1$  appartenant  \`a
$S_{2}\EN{\g}\cap E\GD{\g}$  tel que $\E\subset \D_{\rho_1}$. Soient
$\rho_2$ et $\rho_3$ deux vecteurs extr\'emaux de $\GD{\g}$  adjacents \`a $\rho_1$, c'est-\`a-dire il existe un c\^one $\sigma\in \GD{\g}$ de dimension $3$ tel que les vecteurs $\rho_1$, $\rho_2$ et $\rho_3$ sont vecteurs extr\'emaux de $\sigma$. On remarque que le point g\'en\'erique de $\E$ n'est pas contenu dans $\D_{\rho_2}$ ou  $\D_{\rho_3}$.\\ 

 Pour un vecteur  $m=(a,b,c)\in \ZZ^3$, on note $\chi^{m}(t_1,t_2,t_3)=t_1^{a}t_2^{b}t_3^{c}$ le caract\`ere associ\'e \`a  $m$. Soient  $U_{\sigma}$ l'ouvert torique  de $X(\GD{\g})$ associ\'e \`a $\sigma$  et  $\chi^{m_i}$ le caract\`ere qui d\'efinit une \'equation de $\D_{\rho_i}\cap U_{\sigma}$,  pour $i\in\{1,2,3\}$. Alors, on a $m_i\cdot \rho_j=\delta_{ij}$, o\`u $\delta_{ij}$ est le symbole de Kronecker. Quitte \`a remplacer les vecteurs $\rho_2$, $\rho_3$, on peut  supposer que  $U_{\sigma}\cap E\neq \emptyset$.\\

On consid\`ere l'unique  rel\`evement  $\widehat{\alpha}_{\E}$ \`a $X$ du point g\'en\'erique $\alpha_{\E}$ de $N_{\E}$ (on rappelle que $X$ est le transform\'e strict de $S$ dans $X(\GD{\g})$).  On remarque que $\widehat{\alpha}_{\E}(0)$ est le point g\'en\'erique de $\E$ et que  le $\KK_{\alpha}$-arc $\widehat{\alpha}_{\E}$ est transverse \`a $\E$, c'est-\`a-dire $\ordt f\circ \widehat{\alpha}_{\E}=1$, o\`u $f$ est une \'equation locale de $\E$.\\

 D'apr\`es la Proposition \ref{pr:cr_nor-g}, le morphisme $\pi':X(\GD{\g})\rightarrow \AF^3_{\KK}$ est une r\'esolution plong\'ee de $S$.  En particulier $X$ est  transverse au diviseur $\D_{\rho_1}$.Comme le $\KK_{\alpha}$-arc $\widehat{\alpha}_{\E}$ est transverse \`a $\E$, le $\KK_{\alpha}$-arc $\widehat{\alpha}_{\E}$ est transverse \`a $\D_{\rho_1}$.  Par cons\'equent, on obtient que  $ m_i\cdot (\mu_x,\mu_y,\mu_z)=\delta_{i 1}$, car   $\widehat{\alpha}_{\E}(0)$ n'est pas contenu dans $\D_{\rho_2}$ ou $\D_{\rho_3}$. Ceci implique que $(\mu_x,\mu_y,\mu_z)=\rho_{1}$, d'o\`u la proposition. \end{proof}

\begin{corollaire}
\label{co:musygemi} On conserve les hypoth\`eses et notations de la Proposition \ref{pr:musygemi-g} et on suppose que $S=\B{p}{q}$. Alors, le vecteur $(\mu_x,\mu_y,\mu_z)$  appartient au syst\`eme g\'en\'erateur minimal du semi-groupe $\tau\cap \ZZ^{3}_{\geq 0}$,
o\`u $\tau$ est le c\^one engendr\'e par les vecteurs  $(0,0,1)$ et $(p,p,q)$. En particulier, on a $\mu_x=\mu_y\leq p$, $\mu_z\leq q$ et  $p\mu_z-q\mu_x\geq 0$.
\end{corollaire}
\begin{proof}
 Le corollaire r\'esulte des Propositions  \ref{pr:sygemi} et \ref{pr:musygemi-g}.
\end{proof}

Dans la proposition suivante on suppose que $S$ est une hypersurface normale de $\AF^{3}_{\KK}$, donn\'ee par l'\'equation $\g=0$, o\`u $\g$ est un polyn\^ome irr\'eductible 
non-d\'eg\'en\'er\'e par rapport \`a la fronti\`ere de Newton $\Gamma(\g)$. De plus, on suppose que $O$ est l'unique point singulier de $S$.\\

On consid\`ere un diviseur essentiel $\E$ sur $S$ et un $K$-wedge admissible, $\omega:\spec K[[s,t]]\rightarrow S$, centr\'e en $N_{\E}$. On pose
\begin{center}
$(\eta_x,\eta_y,\eta_z):=(\ordt{\com{\omega}}(x),\ordt{\com{\omega}}(y), \ordt{\com{\omega}}(z)).$ 
\end{center}

On peut \'ecrire le comorphisme de $\omega$ de la fa\c con suivante:
\begin{center}
$\com{\omega}(x)=t^{\eta_x}\chi$; $\com{\omega}(y)=t^{\eta_y}\varphi$; $\com{\omega}(z)=t^{\eta_z}\psi$,
\end{center}
o\`u $\chi, \varphi$ et $\psi$ sont des s\'eries formelles  dans $K[[s,t]]$ qui ne sont pas divisibles par $t$. \\



Maintenant, on donne la proposition cl\'e pour la preuve du th\'eor\`eme \ref{th:nashBP}.
\begin{proposition}
\label{pr:invserBP}
Si les s\'eries formelles $\chi$,$\varphi$ et $\psi$ sont inversibles, alors le  $K$-wedge admissible $\omega$ centr\'e en $N_{\E}$ se rel\`eve \`a la r\'esolution minimale de $S$. 
\end{proposition}
 \begin{proof} On consid\`ere une $G$-subdivision r\'eguli\`ere $\GD{\g}$ de l'\'eventail de Newton $\EN{g}$ et on note $\pi':X(\GD{\g})\rightarrow \AF^{3}_{\KK}$ le morphisme torique induit par la subdivision $\GD{\g}$ du c\^one $\RR^{3}_{\geq 0}$.
On note $X$ le transform\'e strict de $S$ dans $X(\GD{\g})$. D'apr\`es la proposition \ref{pr:cr_nor-g}, le morphisme  $\pi':X\rightarrow S$ est une d\'esingularisation de $X$.\\

 En vertu de la {\it version torique du Lemme de Chow} (voir \cite{Sum74}), il existe une subdivision $\Sigma$ d'\'eventail $\GD{\g}$
tel que le morphisme torique $\pi'':X(\Sigma)\rightarrow \AF^3_{\KK}$, (le morphisme $\pi''$ est induit par la subdivision $\Sigma$  du c\^one $\RR^{3}_{\geq 0}$) est un morphisme projectif, la vari\'et\'e $X(\Sigma)$ est quasi-projective et $\pi'$ factorise $\pi''$. Par cons\'equent, il existe un id\'eal monomial  $\mathcal{I}\subset \KK[x,y,z]$ tel que $X(\Sigma)$ est l'\'eclatement de $\AF^3_{\KK}$ de centre l'id\'eal $\mathcal{I}$.\\

Comme les s\'eries formelles $\chi$,$\varphi$ et $\psi$ sont inversibles et  l'id\'eal $\mathcal{I}$ est monomial, l'id\'eal $\omega^{-1}\mathcal{I}\cdot K[[s,t]]$ est inversible. En vertu de la {\it propri\'et\'e universelle de l'\'eclatement}, le morphisme $\omega$ se rel\`eve \`a $X(\Sigma)$. Par cons\'equent $\omega$ se rel\`eve \`a $X$.  Ceci implique que le $K$-wedge $\omega$ se rel\`eve \`a r\'esolution minimale de $S$. \end{proof}

Maintenant, on donne quelques notions et r\'esultats techniques qui nous permettent de montrer, dans le cas $S=\B{p}{q}$, que les s\'eries  formelles $\chi$,$\varphi$ et $\psi$ sont inversibles.\\

Pour  une s\'erie non nulle $\phi:=\sum c_{(e_1,e_2)}s^{e_1}t^{e_2}$, o\`u $c_{(e_1,e_2)} \in K$, on d\'efinit les applications suivantes:\\

\begin{itemize}
 \item[] $\nu:\RR^{2}_{>0}\rightarrow \RR_{\geq 0},v\mapsto \nu_v\phi:=\min\{v\cdot e\mid e \in\mathcal{E}(\phi)\}$, o\`u $ \mathcal{E}(\phi)= \{(e_1,e_2)\mid c_{(e_1,e_2)}\neq 0\}$;\\

\item[] $\pp:\RR^{2}_{>0}\rightarrow K[s,t],v\mapsto\phi_v:=\sum\limits_{{\tiny \mbox{$e\cdot v= \nu_v\phi$}}} c_{(e_1,e_2)}s^{e_1}t^{e_2}$;\\

\item[] $\FI:K[[s,t]]\backslash\{0\}\rightarrow \ZZ_{\geq 0}$, o\`u $\FI(\phi)$ est le nombre de facteurs irr\'eductibles de $\phi$ compt\'es avec multiplicit\'e.
\end{itemize}
\vspde

Un vecteur $v\in \RR^2_{>0}$ d\'efinit une graduation positive sur l'anneau $K[[s,t]]$. Cette graduation  est appel\'ee  $v$-{\it graduation}. Pour une s\'erie formelle $\phi$, le polyn\^ome $\phi_v$  est la partie principale de $\phi$  pour la $v$-graduation.
Le polyn\^ome $\phi_v$ est appel\'e la $v$-{\it partie principale} de $\phi$.\\

Soient $\phi$, $\phi'\in K[[s,t]]$ deux s\'eries formelles non nulles. Les s\'eries formelles $\phi$ et $\phi'$ sont associ\'ees (resp. non associ\'ees)  s'il existe (resp. s'il n'existe pas) une s\'erie formelle $I\in K[[s,t]]$ inversible tel que $\phi=I\phi'$.\\

\begin{proposition}
\label{pr:FI-g} On conserve les hypoth\`eses et notations de la Proposition \ref{pr:invserBP}. Alors,  il existe un vecteur $v\in \QQ^{2}_{>0} $  tel que:

\begin{center}
 
  $\FI(\chi)\leq \degt{\chi_v}= \nu_v\chi =\mu_x-\eta_x$;

 $\FI(\varphi)\leq \degt{\varphi_v}= \nu_v\varphi= \mu_y-\eta_y$;

 $\FI(\psi)\leq \degt{\psi_v}= \nu_v\psi= \mu_z-\eta_z$.
\end{center}
De plus, $\chi$ (resp. $\varphi$, $\psi$) est inversible si et seulement si $\mu_x-\eta_x=0$ (resp.  $\mu_y-\eta_y=0$,  $\mu_z-\eta_z=0$).

\end{proposition}

\begin{proof} 

Soit $\phi\in K[[s,t]]$ une s\'erie formelle  non nulle et on suppose que 

\begin{center}
 $\phi=I \phi_1^{m_1}\cdots\phi_{n}^{m_n}$, $n\geq 1$,
\end{center}
o\`u  les entiers $m_i$ sont strictement positifs et les $\phi_i$ sont des s\'eries  formelles irr\'eductibles deux \`a deux non associ\'ees. Alors, on a:

\begin{center}
 $\phi_v=((\phi_1)_v)^{m_1}\cdots((\phi_{n})_v)^{m_n}$, pour tout $v\in\RR^{2}_{>0}$.
\end{center}

Par cons\'equent $\FI(\phi)\leq \FI(\phi_v)$ pour tout $v\in \RR^{2}_{>0}$. On remarque que les $(\phi_i)_v$, $1\leq i\leq n$, ne sont pas n\'ecessairement irr\'eductibles.\\

Dans la suite on cherche  un vecteur $v\in \QQ^{2}_{>0} $ tel que  

\begin{center}
 
  $\FI(\chi_v)=\degt{\chi_v}= \nu_v\chi =\mu_x-\eta_x$;

 $\FI(\varphi_v)= \degt{\varphi_v}= \nu_v\varphi= \mu_y-\eta_y$;

 $\FI(\psi_v)=\degt{\psi_v}= \nu_v\psi= \mu_z-\eta_z$.
\end{center}

On rappelle les notations suivantes:\\

\begin{itemize}
 \item[-] $\omega:\spec K[[s,t]]\rightarrow S$ est un $K$-wedge admissible centr\'e en $N_{\E}$;

\item[-] $\alpha_{\E}$ est le point g\'en\'erique de $N_{\E}$ et  $(\mu_x,\mu_y,\mu_z):=(\ordt{\com{\alpha_{\E}}}(x),\ordt{\com{\alpha_{\E}}}(y), \ordt{\com{\alpha_{\E}}}(z))$;

\item[-]  $\com{\omega}(x)=t^{\eta_x}\chi$; $\com{\omega}(y)=t^{\eta_y}\varphi$; $\com{\omega}(z)=t^{\eta_z}\psi$,
o\`u $\chi, \varphi$ et $\psi$ sont des s\'eries formelles  dans $K[[s,t]]$ qui ne sont pas divisibles par $t$.\\
\end{itemize}

On remarque qu'on peut \'ecrire le comorphisme du $K$-wedge $\omega$ de la fa\c con suivante:
\begin{center}
$ (\star) \left \{\begin{array}{ccc}
\omega^{\star}(x) & =&\displaystyle \sum \limits_{\eta_x\leq i<\mu_x} a_i s^{l_i}t^i+\displaystyle \sum \limits_{\mu_x\leq i} a_it^i;\\
 
\omega^{\star}(y)&=&\displaystyle \sum \limits_{\eta_y\leq j<\mu_y} b_js^{m_j}t^j+\displaystyle \sum \limits_{\mu_y\leq j} b_jt^j;\\
 
\omega^{\star}(z)&=&\displaystyle \sum \limits_{\eta_z\leq k<\mu_z} c_ks^{n_k}t^k+\displaystyle \sum \limits_{\mu_x\leq k} c_kt^k,
\end{array} \right . $
\end{center}
o\`u les exposants  $l_i$ (resp. $ m_j,n_k $)  sont strictement positifs et les s\'eries formelles  $a_{i}, b_{j}, c_{k} \in K[[s]]$ sont inversibles pour $(i,j,k)\in \{(\eta_x,\eta_y,\eta_z),(\mu_x,\mu_y,\mu_z)\}$ et inversibles ou nulles pour $\eta_x< i< \mu_{x}$, $\eta_y<j < \mu_{y}$, $\eta_z<k< \mu_{z}$.
En effet, soit $\lambda_0:\spec K[[t]]\rightarrow \spec K[[s,t]]$ le morphisme induit par l'homomorphisme canonique $K[[s,t]]\rightarrow K[[s,t]]/(s)=K[[t]]$.  On pose $\alpha:=\omega\circ \lambda_0$. 
Comme  $\omega$ est un $K$-wedge admissible centr\'e en $N_{\E}$ et d'apr\`es la propri\'et\'e fonctorielle de l'espace d'arcs 
$S_{\infty}$, on a  $\alpha=\alpha_{\E}\circ \lambda_1$, o\`u $\lambda_1:\spec K[[s,t]]\rightarrow \KK_{\alpha_{\E}}[[s,t]]$ est un morphisme induit par une inclusion  $\KK_{\alpha_{\E}}\hookrightarrow K$.
Comme $\alpha=\alpha_{\E}\circ \lambda_1$, on a $(\ordt{\com{\alpha}}(x),\ordt{\com{\alpha}}(y), \ordt{\com{\alpha}}(z))=(\mu_x,\mu_y,\mu_z)$, d'o\`u les s\'eries formelles ($\star$).\\

 Soit $v=(u,1)\in \QQ^{2}_{>0}$. Si $u$ est ``assez grand'', alors   $\chi_v=at^{\mu_x-\eta_x}$,   $\varphi_v=b t^{\mu_y-\eta_y}$  et $\psi_v=ct^{\mu_z-\eta_z}$, o\`u $a$, (resp. $b$ , $c$) est le terme constant de la s\'erie formelle inversible  $a_{\mu_x}$ (resp.   $b_{\mu_y}$,   $c_{\mu_z}$). Ceci ach\`eve la preuve de la Proposition \ref{pr:FI-g}.\\

On remarque que les s\'erie formelle $\chi$ (resp. $\varphi$, $\psi$) est inversible si et seulement si $\chi_v\in K\backslash \{0\}$ (resp. $\varphi_v\in K\backslash \{0\}$, $\phi_v\in K\backslash \{0\}$).  Par cons\'equent $\chi$ (resp. $\varphi$, $\psi$)
est inversible si et seulement si  $\mu_x-\eta_x=0$ (resp.  $\mu_y-\eta_y=0$,  $\mu_z-\eta_z=0$).
\end{proof}

Dans la proposition suivante, on consid\`ere l'hypersurface $\B{p}{q}$ et on majore, en termes des entiers $p$ et $q$, le nombre des facteurs irr\'eductibles compt\'es avec multiplicit\'e des s\'eries formelles $\chi$, $\varphi$ et $\psi$ qui sont  associ\'ees au  $K$-wedge admissible  $\omega:\spec K[[s,t]]\rightarrow \B{p}{q}$  centr\'e en $N_{\E}$.

\begin{proposition}
\label{pr:mu-eta_BP}
 On conserve les hypoth\`eses et notations des Propositions \ref{pr:invserBP} et \ref{pr:FI-g}. De plus, on suppose que $S=\B{p}{q}$.
Alors, on a:
\begin{center}
 $\mu_x-\eta_x\leq p-1$,  $\mu_y-\eta_y\leq p-1$ et  $\mu_z-\eta_z\leq q-1$.
\end{center}
En particulier, on a  $\FI(\chi)\leq p-1$, $\FI(\varphi)\leq p-1$ et $\FI(\psi)\leq q-1$.
\end{proposition}
\begin{proof}
D'apr\`es la Proposition \ref{pr:FI-g}, si $\mu_x-\eta_x\leq p-1$,  $\mu_y-\eta_y\leq p-1$ et  $\mu_z-\eta_z\leq q-1$, alors
$\FI(\chi)\leq p-1$, $\FI(\varphi)\leq p-1$ et $\FI(\psi)\leq q-1$.\\

En vertu du Corollaire \ref{co:musygemi}, le vecteur $(\mu_x,\mu_y,\mu_z)$  appartient au syst\`eme g\'en\'erateur minimal du semi-groupe $\tau\cap \ZZ^{3}_{\geq 0}$, o\`u $\tau$ est le c\^one engendr\'e par les vecteurs  $(0,0,1)$ et $(p,p,q)$. Par cons\'equent, on a 
$\mu_x\leq p$,  $\mu_y\leq p$ et $\mu_z\leq q$.

Comme  $\omega$ est un $K$-wedge admissible centr\'e en $N_{\E}$, l'arc g\'en\'erique du $K$-wedge $\omega$ appartient \`a $\B{p}{q}^{s}_{\infty}$. Par cons\'equent, on obtient que $\eta_x\geq 1$, $\eta_y\geq 1$
et $\eta_z\geq 1$, d'o\`u   $\mu_x-\eta_x\leq p-1$,  $\mu_y-\eta_y\leq p-1$ et  $\mu_z-\eta_z\leq q-1$. 
\end{proof}

Dans la proposition suivante, $S$ est l'hypersurface de la Proposition \ref{pr:invserBP}, c'est-\`a-dire $S$ est une hypersurface normale de $\AF^{3}_{\KK}$  ayant $O$ comme unique point singulier et qui est donn\'ee par l'\'equation $\g=0$, o\`u  $\g$ est un polyn\^ome irr\'eductible  non-d\'eg\'en\'er\'e par rapport \`a la fronti\`ere de Newton $\Gamma(\g)$.\\

On consid\`ere un diviseur essentiel $\E$ sur $S$ et un $K$-wedge  $\omega:\spec K[[s,t]]\rightarrow S$ admissible centr\'e en $N_{\E}$. On note
\begin{center}
$(\eta_x,\eta_y,\eta_z):=(\ordt{\com{\omega}}(x),\ordt{\com{\omega}}(y), \ordt{\com{\omega}}(z)).$ 
\end{center}

\begin{proposition}
 \label{pr:etacone-g}
Le vecteur $(\eta_x,\eta_y,\eta_z)$ appartient \`a l'intersection de $\ZZ_{>0}^3$ avec le  $2$-squelette $S_2\EN{g}$ de l'\'eventail $\EN{\g}$.
\end{proposition}
\begin{proof}
Comme $\omega$ est un $K$-wedge admissible centr\'e en $N_{\E}$, l'arc g\'en\'erique de $\omega$ appartient \`a $S_{\infty}^s$. Par cons\'equent,  $(\eta_x,\eta_y,\eta_z)$ appartient \`a $\ZZ_{>0}^3$. Il suffit donc de montrer que $(\eta_x,\eta_y,\eta_z)\in S_2\EN{\g}$.\\

 Pour un r\'eel  $u\in \RR$, on d\'efinit le vecteur suivant:
\begin{center}
$\nu_{(u,1)}\omega =(\nu_{(u,1)}\omega^{\star}(x), \nu_{(u,1)}\omega^{\star}(y), \nu_{(u,1)}\omega^{\star}(y))\in \RR^3_{>0}$.
\end{center}
Ce vecteur d\'efinit une graduation sur l'anneau $\KK[x,y,z]$. Soit $\g_u\in \KK[x,y,z]$ la partie principale de  $\g$, par rapport \`a cette graduation.

Le $K$-wedge $\omega$ doit satisfaire l'\'equation $\g=0$, c'est-\`a-dire on a \begin{center}  $\g(\omega^{\star}(x),\omega^{\star}(y),\omega^{\star}(z))=0$, \end{center} ce qui implique que \begin{center}$\g_u((\omega^{\star}(x))_{(u,1)},(\omega^{\star}(x))_{(u,1)},(\omega^{\star}(x))_{(u,1)})=0$. \end{center}

Par cons\'equent, $\g_{u}$ n'est pas un mon\^ome. Ceci implique que le vecteur  $\nu_{(u,1)}\omega$ appartient au $2$-squelette $S^{2}\EN{g}$. 

On remarque que  $\displaystyle\lim_{u\rightarrow 0} \nu_{(u,1)}\omega=(\eta_x,\eta_y,\eta_z)$. Soit $n> >0$ un entier  ``assez  grand'' tel que  $\eta_x < n,\;\; \eta_y<n,\;\;\eta_z<n$. Il existe alors un r\'eel  $u_0>0$ tel que \begin{center} $\nu_{(u,1)}\omega \in K_n:=S_{2}\EN{\g}\cap \{(\lambda_1,\lambda_2,\lambda_3)\in \RR_{\geq 0}^3\mid\;  \lambda_j\leq n\;\mbox{pour}\; j\in \{1,2,3\}\}$, \end{center} pour tout $u\leq u_0$. Comme $K_n$ est compact, on a $ (\eta_x,\eta_y,\eta_z)\in S_{2}\EN{f}$. 
\end{proof}

\begin{corollaire} \label{co:etacone}
 On conserve les hypoth\`eses et notations des Propositions \ref{pr:FI-g} et \ref{pr:etacone-g}. De plus, on suppose que $S=\B{p}{q}$.
Alors, le vecteur $(\eta_x,\eta_y,\eta_z)$  appartient au  semi-groupe $\tau\cap \ZZ^{3}_{\geq 0}$,
o\`u $\tau$ est le c\^one engendr\'e par les vecteurs  $(0,0,1)$ et $(p,p,q)$. En particulier $\eta_x=\eta_y$ et $p\eta_z-q\eta_x\geq 0$.
De plus, si les s\'eries formelles $\chi$,$\varphi$ et $\psi$ ne sont pas simultan\'ement inversibles, alors on a $p\eta_z-q\eta_x> 0$.
\end{corollaire}
\begin{proof} Soit $S_2\RR^{3}_{\geq 0}$ le $2$-squelette du  c\^one $\RR^{3}_{\geq 0}$. On rappelle que $\tau_1\in \EN{\f}$ (resp. $\tau_2\in \EN{\f}$, $\tau_3\in \EN{\f}$) est le c\^one engendr\'e par les vecteurs $(1,0,0)$ (rep. $(0,1,0)$, $(0,0,1)$) et $(p,p,q)$ (voir la Figure \ref{fi:EN}). Remarquons que $S_2\EN{\f}=\bigcup_{i=1}^{3}\tau_i\cup S_2\RR^{3}_{\geq 0}$.
En vertu de la Proposition \ref{pr:etacone-g}, on a $(\eta_x,\eta_y,\eta_z)\in S_2\EN{\f}\cap \ZZ_{>0}^{3}$. Ce qui implique que le vecteur $(\eta_x,\eta_y,\eta_z)$ appartient \`a l'ensemble $\bigcup_{i=1}^{3}\tau_i$.\\

D'apr\`es la Proposition \ref{pr:tau12}, le c\^one $\tau_1$ (resp. $\tau_{3}$) est r\'egulier. Par cons\'equent, le semi-groupe
 $\tau_1\cap \ZZ^3$ (resp.  $\tau_2\cap \ZZ^3$) est engendr\'e  par les vecteurs $(1,0,0)$ (resp. $(0,1,0)$) et $(p,p,q)$. En particulier, si $(a,b,c)\in \tau_1\cap \ZZ_{>0}^3$ (resp. $(a,b,c)\in \tau_2\cap \ZZ_{>0}^3$), alors $p\leq a$, $p\leq b$ et $q\leq c$.\\

On rappelle que le vecteur $(\mu_x,\mu_y,\mu_z)$ appartient au syst\`eme g\'en\'erateur minimal du semi-groupe $\tau_3\cap \ZZ^3$ (voir le Corollaire \ref{co:musygemi}). Par cons\'equent,  on a $\mu_x \leq p$, $\mu_y\leq p$ et $\mu_z\leq q$. Comme 
$\eta_x\leq \mu_x\leq p$, $\eta_y\leq \mu_y\leq p$ et $\eta_z\leq \mu_z\leq p$, on obtient que  le vecteur $(\eta_x,\eta_y,\eta_z)$ appartient au c\^one  $\tau=\tau_3$.\\

Si les s\'eries formelles $\chi$,$\varphi$ et $\psi$ ne sont pas simultan\'ement inversibles, alors $p\eta_z-q\eta_x\neq 0$,
car si $p\eta_z-q\eta_x=0$, alors $(\eta_x,\eta_y,\eta_z)=(\mu_x,\mu_y,\mu_z)=(p,p,q)$, d'o\`u les s\'eries formelles $\chi$, $\varphi$ et $\psi$ sont inversibles (voir la Proposition \ref{pr:FI-g}).\end{proof}
Dans toute la suite de cette section, on se restreint au cas des hypersurfaces $\B{p}{q}$, o\`u les entiers $p>q\geq 3$ (voir la Remarque \ref{re:p>q}) sont premiers entre eux, c'est-\`a-dire, dans toute la suite, on a:\\

\begin{itemize}
\item[-] $\E$ est un diviseur essentiel de $\B{p}{q}$, $p>q\geq 3$;
 \item[-] $\omega:\spec K[[s,t]]\rightarrow \B{p}{q}$ est un $K$-wedge admissible centr\'e en $N_{\E}$;

\item[-] $\alpha_{\E}$ est le point g\'en\'erique de $N_{\E}$ et  $(\mu_x,\mu_y,\mu_z):=(\ordt{\com{\alpha_{\E}}}(x),\ordt{\com{\alpha_{\E}}}(y), \ordt{\com{\alpha_{\E}}}(z))$;

\item[-]  $\com{\omega}(x)=t^{\eta_x}\chi$; $\com{\omega}(y)=t^{\eta_y}\varphi$; $\com{\omega}(z)=t^{\eta_z}\psi$,
o\`u $\chi, \varphi$ et $\psi$ sont des s\'eries formelles  dans $K[[s,t]]$ qui ne sont pas divisibles par $t$.\\
\end{itemize}

Soit $\Gamma_{(\mu_x,\mu_z)}$ (resp. $\Gamma_{(p,q)}$) l'enveloppe convexe de  $\tau'\cap\ZZ_{> 0}^2$, o\`u $\tau'$ est le c\^one engendr\'e par $(0,1)$ et $(\mu_x,\mu_z)$ (resp. $(0,1)$ et $(p,q)$). La figure suivante donne une id\'ee intuitive de la forme du poly\`edre $\Gamma_{(\mu_x,\mu_z)}$.

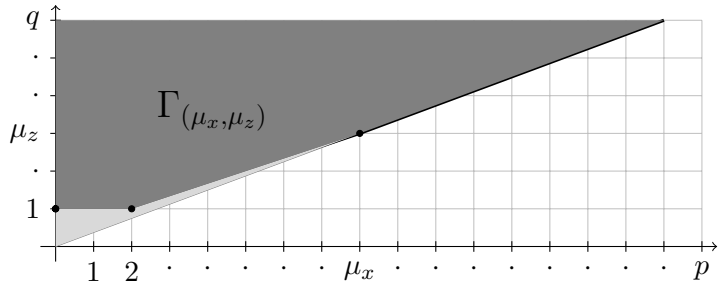
\begin{figure}[!h]

  \begin{tikzpicture}[scale=1]

\tikzfading[name=fade gamma,
left color=transparent!35, 
right color=transparent!35,top color=transparent!35, bottom color=transparent!5]
\tikzfading[name=fade cone,
left color=transparent!60, 
right color=transparent!60,top color=transparent!80, bottom color=transparent!90]

  \draw[style=help lines,step=0.5cm,color=gray!50] (0,0) grid (8.5,3);

  \draw[->] (-0.2,0) -- (8.7,0) node[right]{}; 
  \draw[->] (0,-0.2) -- (0,3.2) node[above]{};
 \
\draw[shift={(0.5,0)}] (0pt,2pt) -- (0pt,-2pt) node[below] {$1$};
 \draw[shift={(1,0)}] (0pt,2pt) -- (0pt,-2pt) node[below] {$2$}; 

\draw[shift={(4,0)}] (0pt,2pt) -- (0pt,-2pt) node[below] {$\mu_x$};

 \foreach \x in {1.5, 2,...,3.5} \draw[shift={(\x,0)}] (0pt,2pt) -- (0pt,-2pt) node[below] {$\cdot$};
\foreach \x in {4.5, 5,...,8} \draw[shift={(\x,0)}] (0pt,2pt) -- (0pt,-2pt) node[below] {$\cdot$};
\draw[shift={(8.5,0)}] (0pt,2pt) -- (0pt,-2pt) node[below] {$p$}; 

\draw[shift={(0,0.5)}] (2pt,0pt) -- (-2pt,0pt) node[left] {$1$};
 \draw[shift={(0,1)}] (2pt,0pt) -- (-2pt,0pt) node[left] {$\cdot$};
 \draw[shift={(0,2)}] (2pt,0pt) -- (-2pt,0pt) node[left] {$\cdot$};
 \draw[shift={(0,1.5)}] (2pt,0pt) -- (-2pt,0pt) node[left] {$\mu_z$};
  
    \draw[shift={(0,2.5)}] (2pt,0pt) -- (-2pt,0pt) node[left] {$\cdot$};

\draw[shift={(0,3)}] (2pt,0pt) -- (-2pt,0pt) node[left] {$q$}; 
 
\draw[gray] (0,0) -- (4,1.5);
\draw[very thick] (4,1.5) -- (8,3);
\draw[very thick] (0,0.5) -- (1,0.5);
\draw[very thick] (1,0.5) -- (4,1.5);
\fill[color=gray!100,path fading=fade gamma]  (4,1.5) -- (8,3)--(0,3) --(0,0.5) -- (1,0.5);
\fill[color=gray!30,path fading=fade cone] (0,0)--(0,0.5)--(1,0.5)-- (4,1.5);
 \draw[fill=black] (4,1.5) circle(0.4mm) (0,0.5) circle(0.4mm) (1,0.5) circle(0.4mm)  (0,0.5) circle(0.4mm);
\draw (2,2) node[inner sep=-1pt,below=-1pt] {{ \Large $\Gamma_{(\mu_x,\mu_z)}$}};
\end{tikzpicture}
\caption{Poly\`edre $\Gamma_{(\mu_x,\mu_z)}$}
\label{fi:Gamma}
\end{figure}
\begin{proposition}
\label{pr:Gammapentpos} 
 Soient $(a_1,b_1)$ et $(a_2,b_2)$, $a_1\leq a_2$, les coordonn\'es des sommets d'une face compacte du poly\`edre  $\Gamma_{(\mu_x,\mu_z)}$  et $L$ la droite qui joint les points  $(a_1,b_1)$ et $(a_2,b_2)$. Alors la pente de la droite $L$ est positive et strictement plus petite que $\frac{q}{p}$.
De plus  $\Gamma_{(\mu_x,\mu_z)}=\tau'\cap \Gamma_{(p,q)}$, o\`u  $\tau'$ est le c\^one engendr\'e par les vecteurs $(0,1)$ et $(\mu_x,\mu_z)$.
\end{proposition}

\begin{proof} 
D'abord on suppose que $(\mu_x,\mu_z)=(p,q)$.
On note $\Gamma=\Gamma_{(p,q)}$.
Comme $p>q$, le vecteur $(1,1)$ appartient au syst\`eme g\'en\'erateur minimal du semi-groupe  $\tau''\cap \ZZ^{2}_{\geq 0}$, o\`u $\tau''$ est le c\^one engendr\'e par $(0,1)$ et $(p,q)$ (voir la Figure \ref{fi:Gamma}). En particulier le vecteur  $(1,1)$ appartient \`a une face compacte du poly\`edre $\Gamma$. Comme la pente de la droite qui joint les points $(0,1)$ et $(1,1)$ est nulle et $1< q$, on obtient que la pente de la droite $L$ est un nombre r\'eel positif, car $\Gamma$ est convexe.

En raisonnant par l'absurde, on suppose que \begin{center} $\dfrac{b_2-b_1}{a_2-a_1}\geq \dfrac{q}{p}$, \end{center}
d'o\`u $p(b_2-b_1)-q(a_2-a_1)\geq 0$. Ceci implique que le vecteur $(a_2-a_1,b_2-b_1)$ appartient au c\^one  $\tau''$. Or $(a_2,b_2)= (a_2-a_1,b_2-b_1)+(a_1,b_1)$, d'o\`u une contradiction, car  le vecteur $(a_2,b_2)$ appartient  au syst\`eme g\'en\'erateur minimal du semi-groupe  $\tau''\cap \ZZ^{2}_{\geq 0}$.\\

On a $\Gamma_{(\mu_x,\mu_z)}=\tau'\cap \Gamma_{(p,q)}$, car la pente de la droite qui joint les points $(0,0)$ et $(\mu_x,\mu_z)$ est plus grande que $\frac{q}{p}$ et les pentes des droites  qui d\'efinissent les faces compactes de $\Gamma_{(p,q)}$ sont  strictement plus petites que  $\frac{q}{p}$.  \end{proof}

On d\'efinit l'application suivante:
\begin{center}
 $\m:\Gamma_{(\mu_x,\mu_z)} \rightarrow \RR, (u,v)\mapsto pv-qu$.
\end{center}

\begin{proposition}\label{pr:minGamma} Il existe un vecteur  $(u_0,v_0)\in \Gamma_{(\mu_x,\mu_z)}$ tel que
\begin{center}
 $\m(u_0,v_0)=\inf\{\m(u,v)\mid (u,v)\in \Gamma_{(\mu_x,\mu_z)}\}$.
\end{center}
 De plus on a:
\begin{itemize}
 \item[i)] si $(\mu_x,\mu_z)=(p,q)$, alors  $(u_0,v_0)$ appartient au rayon engendr\'e par le  vecteur $(p,q)$;
\item[ii)] si $(\mu_x,\mu_z)\neq(p,q)$, alors  $(u_0,v_0)=(\mu_x,\mu_z)$.
\end{itemize}

\end{proposition}

\begin{proof}Le vecteur $(\mu_x,\mu_z)$ \'etant fix\'e, on note $\Gamma=\Gamma_{(\mu_x,\mu_z)}$.

   G\'eom\'etriquement, lorsque $c$ cro\^it depuis $-\infty$, les droites $L_c:=\{(u,v)\in \RR^2\mid\;(-q,p)\cdot (u,v)=c\}$  finissent par toucher le bord du poly\`edre $\Gamma$ en un point $(u_0,v_0)$. On remarque que $\m(u_0,v_0)=\inf\{\m(u,v)\mid (u,v)\in \Gamma\}$ et que l'ensemble $M:=\{(u,v)\in \Gamma \mid \m(u,v)=\m(u_0,v_0)\}$ est une face  ou un sommet de $\Gamma$.\\

On rappelle que pour tout vecteur $(u,v)\in \Gamma$ on a $pv-qu\geq 0$.
Si  $(\mu_x,\mu_z)=(p,q)$, alors $L_0\cap \Gamma$  est la face non compacte engendr\'ee par le vecteur 
$(p,q)$, d'o\`u le point $i)$ de la proposition.
Maintenant, on suppose que  $(\mu_x,\mu_z)\neq(p,q)$. La pente de la droite $L'$ qui joint les points $(0,0)$ et $(\mu_x,\mu_z)$  est  strictement  plus grande que $\frac{q}{p}$, donc toute droite $L_c$ pour $c\in \RR$ intersecte en exactement  un point la droite $L'$. Par cons\'equent, le couple $(u_0,v_0)$ appartient \`a une face compacte de $\Gamma$.  D'apr\`es la Proposition \ref{pr:Gammapentpos}, si $L$ est une droite engendr\'ee par une face de $\Gamma$, alors la pente de $L$ est strictement plus petite
que $\frac{q}{p}$. Comme la pente des droites $L_c$, $c\in \RR$, est $\frac{q}{p}$, on obtient que l'ensemble $M$  est un sommet de $\Gamma$.

 Soit $n_{\Gamma}$ le nombre de faces compactes de $\Gamma$. On rappelle que $\Gamma$ est l'enveloppe convexe de  $\tau'\cap\ZZ^{2}_{>0}$, o\`u $\tau'$ est le c\^one engendr\'e par les vecteurs $(0,1)$ et $(\mu_x,\mu_z)$.
Alors, il existe deux suites d'entiers
\begin{center}
$ 0=a_0<a_1\cdots <a_{n_{\Gamma}}=\mu_x$ et 
$ 1=b_0<b_1\cdots <b_{n_{\Gamma}}=\mu_z$,
\end{center}
tels que les couples $(a_i,b_i)$, $0\leq i\leq n_{\Gamma}$, sont les coordonn\'ees des sommets cons\'ecutifs de $\Gamma$.

Pour  $1\leq i\leq n_{\Gamma}$, on note $L_i$ la droite qui joint  les points  $(a_{i-1},b_{i-1})$ et $(a_i,b_i)$ et posons $c_i$ (resp $c_0$) le r\'eel tel que la droite  $L_{c_i}$ (resp. $L_{c_0}$) intersecte la droite $L_i$ (resp. $L_1$) en le point $(a_i,b_i)$ (resp. $(0,1)$).

  En vertu de la Proposition \ref{pr:Gammapentpos},  on a  $c_{i}<c_{i-1}$, pour tout $1\leq i\leq n_{\Gamma}$ (la pente de la droite $L_{i}$, $0\leq i\leq n_{\Gamma}$,  est positive et strictement plus petite que $\frac{q}{p}$).  Comme $\m(a_i,b_i)=c_i$, pour tout $0\leq i \leq n_{\Gamma}$,  on obtient que $\m(\mu_x,\mu_z)<\m(a_i,b_i)$, pour tout $0\leq i \leq n_{\Gamma}-1$. Ceci ach\`eve la preuve de la proposition. 
\end{proof}

\begin{proposition}\label{pr:etaGamma}
Le vecteur $(\eta_x,\eta_z)$  appartient au poly\`edre  $\Gamma_{(\mu_x,\mu_z)}$.
\end{proposition}
\begin{proof} 
En vertu du Corollaire \ref{co:etacone} le vecteur $(\eta_x,\eta_z)$ appartient au c\^one $\tau''$ engendr\'e par les vecteurs $(0,1)$ et $(p,q)$. Par cons\'equent, le vecteur  $(\eta_x,\eta_z)$ appartient  au poly\`edre $\Gamma_{(p,q)}$.

En raisonnant par l'absurde si le vecteur $(\eta_x,\eta_z)$  n'appartient pas au poly\`edre $\Gamma_{(\mu_x,\mu_z)}$, alors ce vecteur  appartient \`a l'intersection $\Omega$ du c\^one $\tau''$ et l'int\'erieur du triangle d\'efinit par les vecteurs $(0,0)$, $(\mu_x,\mu_z)$ et $(\mu_x,0)$ car $\eta_x\leq \mu_x$, $\eta_z\leq \mu_z$ et $\Gamma_{(\mu_x,\mu_z)}$ est l'enveloppe convexe de  $\tau'\cap\ZZ_{> 0}^2$, o\`u $\tau'$ est le c\^one engendr\'e par les vecteurs $(0,1)$ et $(\mu_x,\mu_z)$. Mais d'apr\`es la Proposition \ref{pr:Gammapentpos}, l'intersection du  poly\`edre $\Gamma_{(p,q)}$ et l'ensemble $\Omega$ est vide,  d'o\`u une contradiction. \end{proof}
On rappelle qu'on veut montrer que les s\'eries formelles $\chi$, $\varphi$ et $\psi$ sont inversibles (Proposition \ref{pr:invserBP}). Le r\'esultat suivant est une r\'eduction du probl\`eme.
\begin{proposition}
\label{pr:etaGammapsi}
 Si $\chi$ ou $\varphi$ est une s\'erie formelle inversible, alors les s\'eries formelles   $\chi$, $\varphi$ et $\psi$ sont inversibles 
\end{proposition}
\begin{proof}
 D'apr\`es les r\'esultats \ref{co:musygemi},  \ref{pr:FI-g} et \ref{co:etacone}, on obtient que la s\'erie $\chi$ est inversible si et seulement si  $\varphi$ est inversible. Supposons que $\chi$ soit inversible, on a donc  $\mu_x=\eta_x$ (Proposition \ref{pr:FI-g}).  Comme on a $\eta_z \leq \mu_z$ et $\eta_x=\mu_x$, on a  $p\eta_z-q\eta_x\leq p\mu_z-q\mu_x$. En vertu des  Propositions \ref{pr:minGamma} et  \ref{pr:etaGamma}, on obtient que $\mu_z=\eta_z$. La proposition r\'esulte de la Proposition \ref{pr:FI-g}. 
\end{proof}

Soit $\hq{x}{y}=\prod _{i=1}^{q}(a_ix+b_iy)$ la d\'ecomposition en facteurs irr\'eductibles de $\fhq$. Le $K$-wedge $\omega$ doit satisfaire l'\'equation $z^p=-\hq{x}{y}$ donc:  
\begin{center}
$t^{p\eta_{z}-q\eta_x}\psi^p=-\hq{\chi}{\varphi}=-\prod _{i=1}^{q}\gamma_{i},$ o\`u $\gamma_{i}:=a_i\chi+b_i\varphi$.
\end{center}

  Les combinaisons lin\'eaires de $\chi$ et $\varphi$ donn\'ees par les $\gamma_i$ plus l'hypoth\`ese sur les facteurs irr\'eductibles de $\fhq$ permettent de montrer le lemme suivant.

\begin{lemme}
\label{le:pgcdninv} Soit $\lambda:=\pgcd{\gamma_{1}}{\gamma_{2}}$. Alors    $\pgcd{\gamma_{i}}{\gamma_{j}}=\lambda I_{ij}$, o\`u $I_{ij}$ est  inversible pour tous les entiers   $1\leq i<j\leq q$. De  plus, si $\lambda$ est inversible, alors  $\chi$, $\varphi$ et $\psi$ sont inversibles.
\end{lemme}

 \begin{proof} 
 Soit  $\lambda_0:=\pgcd{\gamma_{i_0}}{\gamma_{j_0}}$, o\`u $i_0$ et $j_0$ sont  deux entiers tels que $1\leq i_0<j_0\leq q$. Si $\lambda_0$ n'est pas inversible, alors $\lambda_0$ divise $\chi$ et $\varphi$. Par cons\'equent,  $\lambda_0$ divise $\pgcd{\gamma_{i}}{\gamma_{j}}$ pour tous les entiers $i$, $j$ tels que $1\leq i<j\leq q$. Ce qui ach\`eve la premi\`ere partie du lemme.\\

Pour la deuxi\`eme partie de la proposition on suppose que la s\'erie formelle $\lambda$ est inversible. 

Raisonnons par l'absurde.  En vertu de la Proposition \ref{pr:etaGammapsi}, les s\'eries $\chi$ et $\varphi$ ne sont pas inversibles. Par cons\'equent, la s\'erie $\gamma_i$, $1\leq i\leq q$, n'est pas inversible. 
On rappelle que les s\'eries formelles  $\chi$, $\varphi$ et $\psi$ satisfont la relation suivante:  

\begin{center}
$t^{p\eta_{z}-q\eta_x}\psi^p=-\hq{\chi}{\varphi}=-\prod _{i=1}^{q}\gamma_{i},$ o\`u $\gamma_{i}:=a_i\chi+b_i\varphi$.
\end{center}

 Comme $\chi$ et $\varphi$ ne sont pas divisibles par $t$, $t$ divise $\gamma_i$ si et seulement si $t$ ne divise pas  $\gamma_j$ pour tout $j\neq i$. Quitte \`a re-num\'eroter  les $\gamma_i$, on peut supposer que $t^{p\eta_z-q\eta_x}$ divise $\gamma_1$. Soit $\gamma_{1}=-t^{p\eta_{z}-q\eta_x}\gamma'_{1}$.

Comme  $\pgcd{\gamma_{i}}{\gamma_{j}}$ est inversible pour tout $2\leq i<j\leq q$, la s\'erie formelle $\gamma'_1\prod _{i=2}^{q}\gamma_{i}$ a au moins $q-1$ facteurs irr\'eductibles deux \`a deux non associ\'es.

Comme $\psi^p= \gamma'_1\prod _{i=2}^{q}\gamma_{i}$, la s\'erie formelle $\gamma'_1\prod _{i=2}^{q}\gamma_{i}$ est le produit de $q-1$  puissances de s\'eries formelles irr\'eductibles deux \`a deux non-associ\'ees, car $\FI(\psi)\leq q-1$ (Proposition \ref{pr:mu-eta_BP}). Par cons\'equent,  on obtient que  $\gamma'_1$ est inversible, que la s\'erie $\gamma_i$, $2\leq i\leq q$, est une puissance d'une s\'erie formelle irr\'eductible et  que
 $\psi=\prod_{i=1}^{q-1}\psi_i$, o\`u les  $\psi_{i}$ sont des s\'eries formelles irr\'eductibles deux \`a deux non associ\'ees.
 Par cons\'equent, on peut supposer que $\gamma_{i+1}=\psi_i^pI_i$ pour $1\leq i\leq q-1$, o\`u les s\'eries formelles $I_i$ sont inversibles. Comme  $\gamma_i=a_i\chi+b_i\varphi$, $1\leq i\leq q$, et $q\geq 3$,  il existe deux constantes $a, b\in K$ telles que 
\begin{center} $t^{p\eta_{z}-q\eta_x}\gamma'_{1}=aI_{1}\psi_{1}^{p}+bI_{2}\psi_{2}^{p}$. \end{center}
                                                                                         
On rappelle qu'une s\'erie formelle dans $K[[s,t]]$ est inversible si et seulement si elle l'est  dans $\overline{K}[[s,t]]$, o\`u $\overline{K}$ est la cl\^oture alg\'ebrique de $K$. Soient $J_1$ et $J_2$ deux s\'eries formelles inversibles dans   $\overline{K}[[s,t]]$ telles que $J_1^p=aI_1$ et  $J_2^p=bI_2$. Ainsi, on obtient que:
  \begin{center} $t^{p\eta_{z}-q\eta_x}\gamma'_{1}=\prod_{i=1}^{p}(J_1\psi_1+w_iJ_2\psi_2)$, \end{center}
 o\`u les $w_i$ sont les racines $p$-i\`emes de l'unit\'e.                                                                                       
Mais $\gamma'_1$ est inversible et   $p\eta_{z}-q\eta_x>0$ (Proposition \ref{co:etacone}), donc $\psi_{1}$ et $\psi_{2}$ sont inversibles ou divisibles par $t$, ce qui est absurde. \end{proof}

La proposition suivante  ach\`eve la preuve du Th\'eor\`eme \ref{th:nashBP} (voir la Proposition \ref{pr:invserBP}). 

\begin{proposition}
\label{pr:NashBP}
 Les s\'eries formelles $\chi$, $\varphi$, $\psi$  sont inversibles.
\end{proposition} \begin{proof} 
 En raisonnant par l'absurde, on suppose les s\'eries formelles $\chi$, $\varphi$ et $\psi$ non toutes inversibles.
D'apr\`es le corollaire \ref{co:etacone}, on a  $p\eta_z-q\eta_x>0$.\\

On rappelle que les s\'eries formelles  $\chi$, $\varphi$ et $\psi$ satisfont la relation suivante:  

\begin{center}
$t^{p\eta_{z}-q\eta_x}\psi^p=-\hq{\chi}{\varphi}=-\prod _{i=1}^{q}\gamma_{i},$ o\`u $\gamma_{i}:=a_i\chi+b_i\varphi$.
\end{center}

 Comme $\chi$ et $\varphi$ ne sont pas divisibles par $t$, $t$ divise $\gamma_i$ si et seulement si  
$t$ ne divise pas  $\gamma_j$ pour tout $j\neq i$. Quitte \`a re-num\'eroter  les $\gamma_i$, on peut supposer que $t^{p\eta_z-q\eta_x}$ divise $\gamma_1$.\\

 D'apr\`es le Lemme \ref{le:pgcdninv}, il existe des s\'eries formelles $\gamma'_i$ telles que: 
\begin{center}
$\gamma_{1}=-t^{p\eta_{z}-q\eta_x}\gamma'_{1}\lambda$ et  $\gamma_{i}=\gamma'_{i}\lambda$ pour $2\leq i\leq q$,
\end{center}
 o\`u $\lambda$ n'est pas inversible et  $\pgcd{\gamma'_{i}}{\gamma'_{j}}$ est inversible pour $1\leq i<j\leq q$.\\

Le lemme suivant est le r\'esultat  cl\'e pour la preuve de la Proposition \ref{pr:NashBP}.\\

Dans toute la suite $v$, d\'esigne le vecteur de la Proposition \ref{pr:FI-g}.
\begin{lemme}
\label{le:lambda}
Le $v$-ordre de  $\lambda$ est $\mu_{x}-\eta_{x}$, c'est-\`a-dire $\nu_v{\lambda}=\mu_{x}-\eta_{x}$.
\end{lemme} 
D'abord finissons la preuve de la Proposition \ref{pr:NashBP}. 
 On rappelle que  $\eta_y=\eta_x\leq \mu_x=\mu_y\leq p$,  $\eta_z\leq \mu_z\leq q $ (voir le Corollaire \ref{co:musygemi}) et     que $(\eta_x,\eta_z)$  appartient \`a l'enveloppe convexe $\Gamma_{(\mu_x,\mu_z)}$ (voir la Proposition \ref{pr:etaGamma}) de l'ensemble $\tau'\cap\ZZ_{>0}^2$, o\`u $\tau'$ est le c\^one engendr\'e par $(0,1)$ et $(\mu_x,\mu_z)$. Soit $\gamma'=\prod_{i=1}^{q} \gamma'_{i}$, d'o\`u $\psi^p=\gamma'\lambda^q$. D'apr\`es le Lemme \ref{le:lambda}, on a  $\nu_v\gamma'=p\mu_{z}-q\mu_{x} -(p\eta_{z}-q\eta_{x})$. Par d\'efinition $\nu_v\gamma'\geq 0.$ 
Or  $\nu_v\gamma'\geq 0$ si et seulement si  $(\eta_x,\eta_z)=(\mu_x,\mu_z)$ (Proposition \ref{pr:minGamma}), d'o\`u une contradiction. En effet,  si  $(\eta_x,\eta_z)=(\mu_x,\mu_z)$ , alors les s\'eries formelles $\chi$, $\varphi$ et $\psi$ sont  inversibles  (voir les Propositions \ref{pr:FI-g} et \ref{pr:etaGammapsi}).\\

{\it D\'emonstration du Lemme \ref{le:lambda}} 
On rappelle que $\gamma_{i}=a_i\chi+b_i\varphi$ pour $1\leq i\leq q$. Alors, il existe au plus un $1\leq i_0\leq q$ tel que les $v$-parties principales $\chi_v$, $\varphi_v$ satisfont la relation  $a_{i_0}\chi_v+b_{i_0}\varphi_v=0$. En particulier on a $\nu_{v}\gamma_{i}=\mu_x-\eta_x$ pour tout $1\leq i\leq q$ tel que $i\neq i_{0}$. Par cons\'equent, $\nu_v \lambda\leq \mu_x-\eta_x$.\\

Si  $\nu_v \lambda< \mu_x-\eta_x$, alors $\nu_v\gamma'_{i}>0$, pour tout $2\leq i\leq q$. Ceci implique que les $\gamma'_{i}$, pour  $2\leq i\leq q$, ne sont pas inversibles. 

Comme  $\pgcd{\gamma'_{i}}{\gamma'_{j}}$ est inversible pour tous les entiers $i$, $j$ tels que  $2\leq i<j\leq q$, la s\'erie formelle 
 $\gamma'=\prod_{i=1}^{q} \gamma'_{i}$ 
a au moins $q-1$ facteurs irr\'eductibles deux \`a deux non associ\'es.   Comme on a 
$\psi^p=\gamma'\lambda^q$, la s\'erie formelle $\gamma'$ est le produit de $q-1$ puissances de s\'eries formelles  irr\'eductibles non-associ\'ees, car $\FI(\psi)\leq q-1$ (Proposition \ref{pr:mu-eta_BP}). Par cons\'equent,  on obtient que  $\gamma'_1$ est inversible, que la s\'erie $\gamma'_i$, $2\leq i\leq q$, est une puissance d'une s\'erie formelle irr\'eductible,  que
 $\psi=\prod_{i=1}^{q-1}\psi_i$, o\`u les  $\psi_{i}$ sont des s\'eries formelles irr\'eductibles deux \`a deux non associ\'ees, et que  $\mu_z-\eta_z=q-1$. De plus, on a  $(\mu_x,\mu_y,\mu_z)=(p,p,q)$. Ceci  implique que $p\not\equiv 1\mod{q}$,   car si $p\equiv 1\mod{q}$, alors  le diviseur $\E_{0}=\D_{(p,p,q)}\cap\SD$ n'est pas un diviseur essentiel (voir la Proposition \ref{pr:sygemi} et le Corollaire  \ref{co:Gdes.r=1}).\\

On rappelle qu'une s\'erie formelle dans $K[[s,t]]$ est inversible si et seulement si elle est inversible dans $\overline{K}[[s,t]]$, o\`u $\overline{K}$ est la cl\^oture alg\'ebrique de $K$. Dans la suite on suppose que le corps  $K$ est alg\'ebriquement clos.\\

On fixe un entier $1\leq i\leq q-1$ quelconque. On peut donc supposer que $\psi^p_{i}=\gamma'_{i+1}\lambda_i^{q}$, o\`u $\xi\lambda=\prod_{j=1}^{q-1}\lambda_j$, $\xi^q=\gamma'_1$ et  $\pgcd{\lambda_j}{\lambda_{j'}}$ est inversible pour tous les entiers $j$, $j'$ tels que  $1\leq j<j'\leq q-1$.
Comme $\psi_i$ est irr\'eductible et $\gamma'_{i+1}$ n'est pas inversible, il existe deux entiers $l_i\geq 1$ et $m_i\geq 0$ tels  que $\gamma'_{i+1}=I_i\psi^{l_i}$ et $\lambda_i=I_i^{-1}\psi_i^{m_i}$, o\`u $I_i$ est une s\'erie formelle inversible. Comme  $\nu_v\psi_i=1$, on a $\nu_v\lambda_i=m_i$,  
$\nu_v\gamma'_{i+1}=l_i$ et $l_i+qm_i=p$.\\

 Comme  $(\mu_x,\mu_y,\mu_z)=(p,p,q)$, on a $t^{p\eta_z-q\eta_x}\psi_v=-\hq{\chi_{ {\it v}}}{\varphi_{{\it v}}}$ (Proposition \ref{pr:FI-g}), d'o\`u $\nu_v\gamma_{j}=\mu_x-\eta_x$ pour tout $1\leq j\leq q$. Par cons\'equent, pour tout $1\leq j\leq q-1$, on a  $l_j=\nu_v\gamma'_{j+1}=\nu_v \gamma'_2=l_1$, car $\gamma_{j+1}=\gamma'_{j+1}\lambda$.  En particulier, on a  $m_j=m_1$, pour tout $1\leq j\leq q-1$, car $l_1+qm_j=p$. Ainsi, on obtient que $\gamma_{j+1}=I_{j}\psi_{j}^{p-m_1q}\lambda$, pour $1\leq j\leq q-1$, o\`u les $I_{j}$ sont  inversibles. Comme $q\geq 3$ et  $\gamma_{j}:=a_j\chi+b_j\varphi$, il existe deux constantes  $a, b\in K$ telles que \begin{center} $t^{p\eta_{z}-q\eta_x}\gamma'_{1}=aI_{1}\psi_{1}^{p-m_1q}+bI_{2}\psi_{2}^{p-m_1q}$. \end{center}

 Soient $J_1$ et $J_2$ deux s\'eries formelles inversibles telles que $J_1^{p-m_1q}=aI_1$ et  $J_2^{p-m_1q}=bI_2$. Ainsi, on obtient que:
  \begin{center} $t^{p\eta_{z}-q\eta_x}\gamma'_{1}=\prod_{i=1}^{p-m_1q}(J_1\psi_1+w_iJ_2\psi_2)$, \end{center}
 o\`u les $w_i$ sont les racines $(p-m_1q)$-i\`emes de l'unit\'e.      
Mais $\gamma'_1$ est inversible,  $p\eta_{z}-q\eta_x>0$ et $p\not\equiv 1 \mod q$, donc $\psi_{1}$ et $\psi_{2}$ sont inversibles ou divisibles par $t$, ce qui est absurde. \end{proof}

\section[Les singularit\'es de type $\EE_6$ et $\EE_7$]{Preuve de la bijectivit\'e de l'application de Nash pour les singularit\'es de type $\EE_6$ et $\EE_7$ }
\subsection{La singularit\'e de type $\EE_6$}
      Dans cette section, on d\'emontre la bijectivit\'e de l'application de Nash pour la singularit\'e de type $\EE_6$, ce qui \'equivaut \`a montrer que tous les wedges admissibles se rel\`event \`a la r\'esolution minimale de $\EE_6$ (voir \cite{Reg06}).\\

Soit $S$ l'hypersurface normale de $ \AF_{\KK}^3$  donn\'ee par l'\'equation $x^2+y^3+z^4=0$.  L'hypersurface $S$ a un unique point singulier de type $\EE_6$ \`a l'origine de $\AF_{\KK}^3$.\\

Notons $\f=x^2+y^3+z^4$; on consid\`ere l'\'eventail de Newton  $\EN{\f}$ associ\'e \`a $\f$. Soit $H$ un plan de $\RR^3$ qui ne contient pas l'origine de $\RR^3$ et tel que l'intersection de $H$ et $\RR^3_{\geq 0}$ soit un ensemble compact.  La Figure \ref{fi:ENE6} repr\'esente l'intersection de $H$ avec la subdivision $\EN{\f}$ de $\RR^3_{\geq 0}$. Chaque sommet du diagramme est identifi\'e avec le {\it vecteur extr\'emal}  correspondant.  On note $\tau_1$ (resp. $\tau_2$, $\tau_3$)  le c\^one engendr\'e par les vecteurs $(1,0,0)$ (resp. $(0,1,0)$, $(0,0,1)$) et $(6,4,3)$. 

\begin{figure}[!h]

 \begin{tikzpicture}[font=\small]

   \draw[thick] (0,0)  -- (330:3)  (0,0)  -- (90:3)  (0,0)  -- (210:3);
   \draw[thick] (330:3) --  (90:3)  -- (210:3) -- (330:3);
\draw (0,-0.4) node[inner sep=-1pt,below=-1pt,rectangle,fill=white] {{\tiny $(6,4,3)$}};
\draw (335:3.5) node[inner sep=-1pt,below=-1pt,rectangle,fill=white] {{\tiny $(0,1,0)$}};
\draw (340:1.5) node[inner sep=-1pt,below=-1pt,rectangle,fill=white] {$\tau_2$};
\draw (205:3.5) node[inner sep=-1pt,below=-1pt,rectangle,fill=white] {{\tiny $(1,0,0)$}};
\draw (200:1.5) node[inner sep=-1pt,below=-1pt,rectangle,fill=white] {$\tau_1$};
\draw (90:3.3) node[inner sep=-1pt,below=-1pt,rectangle,fill=white] {{\tiny $(0,0,1)$}};
\draw (80:1.5) node[inner sep=-1pt,below=-1pt,rectangle,fill=white] {$\tau_3$};
    \draw[fill=black]  (0,0) circle(0.7mm) (330:3) circle(0.7mm)  (90:3) circle(0.7mm)  (210:3) circle(0.7mm);

  \end{tikzpicture}

\caption{\'Eventail de Newton $\EN{f}$}
\label{fi:ENE6}
\end{figure}

 Soit  $\GD{\f}$ une $G$-{\it subdivision r\'eguli\`ere}  de  $\EN{\f}$. On note  $\pi_{\mathcal{N}}:X(\EN{\f})\rightarrow \AF^{3}_{\KK}$  (resp. $\pi_{\mathcal{G}}:X(\GD{\f})\rightarrow X(\EN{\f})$)  le morphisme torique induit par la subdivision $\EN{\f}$ de $\RR^3_{\geq 0}$ (resp. $\GD{\f}$ de  $\EN{\f}$) et $S_{\mathcal{G}}$ le transform\'e strict de $S$ associ\'e au morphisme $\pi:=\pi_{\mathcal{G}}\circ \pi_{\mathcal{N}}$.\\

La proposition suivante est un analogue de la Proposition \ref{pr:Gdes1}.

\begin{proposition}
 \label{pr:desingE6} $S_{\mathcal{G}}$ est la r\'esolution minimale de $S$.
\end{proposition}

Soient $\E$ un diviseur essentiel sur $S$ ($\E\in  \Ess(S))$  et  $\alpha_{\E}$ le point g\'en\'erique de $N_{\E}$. On note 
\begin{center} $(\mu_x,\mu_y,\mu_z):=(\ordt{\com{\alpha_{\E}}}(x),\ordt{\com{\alpha_{\E}}}(y), \ordt{\com{\alpha_{\E}}}(z))$,\end{center}
o\`u $\com{\alpha_{\E}}$ est le comorphisme de $\alpha_{\E}$.\\

Pour un c\^one  $\tau$ dans $\EN{\f}$ on note $G\tau$ le syst\`eme g\'en\'erateur minimal du semi-groupe $\tau\cap \ZZ^3$.\\

La proposition suivante est un analogue du Corollaire \ref{co:musygemi}.
\begin{proposition}
\label{pr:musygemiE6}
 Le vecteur $(\mu_x,\mu_y,\mu_z)$ appartient \`a l'union de  $G\tau_2$ et $G\tau_3$, o\`u $\tau_2$  (resp. $\tau_3$) est  le c\^one engendr\'e par les vecteurs $(0,1,0)$ (resp.  $(0,0,1)$) et $(6,4,3)$ (voir la figure \ref{fi:ENE6}).  Autrement dit,  $(\mu_x,\mu_y,\mu_z)\in \{(2,2,1),(3,2,2),(4,3,2),(6,4,3)\}$.
\end{proposition}

Dans la suite, on  montre que pour chaque diviseur essentiel $\E$ tous les $K$-wedges admissibles centr\'es en $N_{\E}$  se rel\`event \`a la r\'esolution minimale de $S$.\\

Soit $\E\in \Ess(S)$ et on consid\`ere un $K$-wedge $\omega:\spec K[[s,t]]\rightarrow S$ admissible centr\'e en $N_{\E}$. On pose: \begin{center} $(\eta_x,\eta_y,\eta_z): =(\ordt \omega^{\star }(x),\ordt \omega^{\star }(y),\ordt \omega^{\star }(z))$. \end{center}

  On peut \'ecrire le comorphisme  de  $\omega$  de la fa\c con  suivante:

\begin{center} 
$\omega^{\star }(x)=t^{\eta_x}\chi,\;\omega^{\star }(y)=t^{\eta_y}\varphi,\; \omega^{\star }(z)=t^{\eta_z}\psi,$
\end{center}
o\`u les s\'eries formelles $\chi$, $\varphi$, $\psi$ ne sont pas divisibles par $t$. On rappelle que $(\eta_x,\eta_y,\eta_z)\in\ZZ^{3}_{>0}$, car $\omega$ est un $K$-wedge admissible.\\

\begin{remarque}
\label{re:221eta643}
 En vertu de la Proposition  \ref{pr:musygemiE6}, on a $\mu_x\leq 6$, $\mu_y\leq 4$ et $\mu_z\leq 3$. En particulier, on a 
$\eta_x\leq 6$, $\eta_y\leq 4$ et $\eta_z\leq 3$, car $\eta_x\leq \mu_x$, $\eta_y\leq \mu_y$ et $\eta_z\leq \mu_z$.

Le $K$-wedge $\omega$ doit satisfaire l'\'equation $x^2+y^3+z^{4}=0$, d'o\`u la relation suivante:
             
\begin{equation*} t^{2\eta_x}\chi^2+t^{3\eta_y}\varphi^3+t^{4\eta_z}\psi^{4}=0. 
\end{equation*}
Ce qui implique que $2\leq \eta_x\leq 6$, $2\leq \eta_y\leq 4$ et $1\leq \eta_z\leq 3$.
\end{remarque}

\begin{remarque} 
\label{re:plan-dem-E6}
D'apr\`es la  Proposition \ref{pr:invserBP}, si les s\'eries formelles $\chi$, $\varphi$ et $\psi$ sont inversibles, alors le $K$-wedge admissible $\omega$ centr\'e en $N_{\E}$ se rel\`eve \`a la r\'esolution minimale de $S$.
 En raisonnant par l'absurde, si au moins l'une d'elles n'est pas inversible on obtient deux cas pour le vecteur $(\eta_x,\eta_y,\eta_z)$ (voir la Proposition \ref{pr:etaconeE6}), ensuite on consid\`ere chaque cas s\'epar\'ement pour obtenir des contradictions (voir les Propositions \ref{pr:2e6} et \ref{pr:1e6}), d'o\`u le cas $\EE_6$ du  Th\'eor\`eme \ref{th:nashDP}.\\
\end{remarque}
On rappelle les d\'efinitions suivantes:\\

Pour  une s\'erie non nulle $\phi:=\sum c_{(e_1,e_2)}s^{e_1}t^{e_2}$, o\`u $c_{(e_1,e_2)} \in K$, on d\'efinit les applications suivantes:\\

\begin{itemize}
 \item[] $\nu:\RR^{2}_{>0}\rightarrow \RR_{\geq 0},v\mapsto \nu_v\phi:=\min\{v\cdot e\mid e \in\mathcal{E}(\phi)\}$, o\`u $ \mathcal{E}(\phi)= \{(e_1,e_2)\mid c_{(e_1,e_2)}\neq 0\}$;\\

\item[] $\pp:\RR^{2}_{>0}\rightarrow K[s,t],v\mapsto\phi_v:=\sum\limits_{{\tiny \mbox{$e\cdot v= \nu_v\phi$}}} c_{(e_1,e_2)}s^{e_1}t^{e_2}$;\\

\item[] $\FI:K[[s,t]]\backslash\{0\}\rightarrow \ZZ_{\geq 0}$, o\`u $\FI(\phi)$ est le nombre de facteurs irr\'eductibles de $\phi$ compt\'es avec multiplicit\'e.\\
\end{itemize}

Le r\'eel  $\nu_v\phi$ (resp. Le polyn\^ome $\phi_v$) est appel\'e le $v$-{\it ordre}  (resp. la $v$-{\it partie principale}) de $\phi$.\\

 D'apr\`es la Proposition \ref{pr:musygemiE6} et la Remarque \ref{re:221eta643}, on a:\\

\begin{center} 
 $\mu_x-\eta_x\leq 4$, $\mu_y-\eta_y\leq 2$ et $\mu_z-\eta_z\leq 2$.\\
\end{center}

En vertu de la Proposition \ref{pr:FI-g}, on peut majorer le nombre de facteurs irr\'eductibles compt\'es avec multiplicit\'e des s\'eries formelles $\chi$, $\varphi$ et $\psi$ \`a l'aide des $v$-ordres. Plus  pr\'ecis\'ement,  il existe un vecteur $v\in \QQ^{2}_{>0} $  tel que:\\

\begin{itemize}
   
  \item[]$\FI(\chi)\leq \degt{\chi_v}= \nu_v\chi =\mu_x-\eta_x\leq 4$;
\item[] $\FI(\varphi)\leq \degt{\varphi_v}= \nu_v\varphi= \mu_y-\eta_y\leq 2$;
 \item[]$\FI(\psi)\leq \degt{\psi_v}= \nu_v\psi= \mu_z-\eta_z\leq 2$.\\
\end{itemize}

Sauf mention du contraire, dans toute la suite le vecteur  $v\in \QQ_{>0}^2$ satisfait la propri\'et\'e ci-dessus.\\

  On remarque qu'une s\'erie formelle $\phi \in K[[s,t]]$ est inversible dans $K[[s,t]]$ si et seulement si elle est inversible dans $\overline{K}[[s,t]]$, o\`u $\overline{K}$ est la cl\^oture alg\'ebrique de $K$. Dans toute la suite on suppose que le corps $K$ est alg\'ebriquement clos.\\

 Maintenant, on prouve la  premi\`ere des trois propositions qu'on a anticip\'e dans la Remarque \ref{re:plan-dem-E6}.

\begin{proposition}
\label{pr:etaconeE6}
 S'il existe au moins  une s\'erie formelle parmi les s\'eries  $\chi$, $\varphi$, $\psi$ qui n'est pas  inversible, alors le vecteur $(\eta_x,\eta_y,\eta_z)\in\ZZ^{3}_{>0}$ satisfait une des relations suivantes:

 \begin{itemize}
\item[i)] $2\eta_x=3\eta_y$ et $\eta_x<2\eta_z$;

  \item[ii)] $\eta_x=2\eta_z$ et $2\eta_x<3\eta_y$.
 \end{itemize}

\end{proposition}
\begin{proof}

Soit $S_2\RR^3_{\geq 0}$ le $2$-squelette du c\^one $\RR^3_{\geq 0}$.
On rappelle que $\tau_1\in \EN{\f}$ (resp. $\tau_2\in \EN{\f}$, $\tau_3\in \EN{\f}$) est   le c\^one engendr\'e par les vecteurs $(1,0,0)$ (resp. $(0,1,0)$, $(0,0,1)$) et $(6,4,3)$ (voir la Figure \ref{fi:ENE6}). Remarquons que $S_2\EN{\f}=\bigcup_{i=1}^{3}\tau_i\cup S_2\RR^3_{\geq 0}$.\\

En vertu de la Proposition \ref{pr:etacone-g}, on a $(\eta_x,\eta_y,\eta_z)\in S_{2}\EN{\f}\cap \ZZ^{3}_{>0}$. Ce qui implique que  le vecteur $(\eta_x,\eta_y,\eta_z)$ appartient \`a l'ensemble $\bigcup_{i=1}^{3}\tau_i$.\\

On remarque que  $(\eta_x,\eta_y,\eta_z)\neq (6,4,3)$. En effet, si $(\eta_x,\eta_y,\eta_z)=(6,4,3)$ alors 
$(\eta_x,\eta_y,\eta_z)=(\mu_x,\mu_y,\mu_z)$ (voir la Remarque \ref{re:221eta643}), d'o\`u les s\'eries formelles  $\chi$, $\varphi$ et  $\psi$ sont inversibles (Proposition \ref{pr:FI-g}). Ceci  rentre en contradiction avec les hypoth\`eses de la Proposition \ref{pr:etaconeE6}. Par cons\'equent, le vecteur $(\eta_x,\eta_y,\eta_z)$ appartient \`a l'ensemble $\bigcup_{i=1}^{3}\tau_i^0$, o\`u $\tau_i^0$ est l'int\'erieur relatif du c\^one $\tau_i$, $1\leq i\leq 3$.\\

Par un calcul direct, on montre que le c\^one $\tau_1$ est r\'egulier. Par cons\'equent,
le semi-groupe $\tau_1\cap \ZZ^{3}$  est engendr\'e par les vecteurs  $(1,0,0)$ et $(6,4,3)$.  En particulier, si $(a,b,c)\in \tau_1\cap \ZZ^{3}_{>0}$, alors $6\leq a$, $4\leq b$ et $3\leq c$. Ce qui implique que le vecteur $(\eta_x,\eta_y,\eta_z)$ appartient \`a l'ensemble $\tau^0_2\cup \tau^0_3$.\\ 

Si $(\eta_x,\eta_y,\eta_z)$ appartient \`a l'ensemble $\tau^0_2$ (resp. $\tau^0_3$), alors  $\eta_x=2\eta_z$ et $2\eta_x<3\eta_y$ (resp.  $2\eta_x=3\eta_y$ et $\eta_x<2\eta_z$), d'o\`u la proposition.
 \end{proof}

 Maintenant, on d\'emontre quelques r\'esultats techniques.\\

On rappelle que le $K$-wedge $\omega$ doit satisfaire l'\'equation $x^2+y^3+z^{4}=0$, d'o\`u la relation suivante:
             
\begin{equation} t^{2\eta_x}\chi^2+t^{3\eta_y}\varphi^3+t^{4\eta_z}\psi^{4}=0. \label{eq:1}\\
\end{equation}

Cette relation est utilis\'ee dans plusieurs endroits de la d\'emonstration du Th\'eor\`eme \ref{th:nashDP}.

\begin{lemme}
\label{le:643}
Si le vecteur  $(\mu_x,\mu_y,\mu_z)$ est \'egal au vecteur (6,4,3), alors les $v$-ordres des s\'eries formelles $t^{\eta_x}\chi+it^{2\eta_z}\psi^2$  et  $t^{\eta_x}\chi-it^{2\eta_z}\psi^2$ sont \'egaux \`a $6$, c'est-\`a-dire on a: 

 $$\nu_v(t^{\eta_x}\chi+it^{2\eta_z}\psi^2)=\nu_v(t^{\eta_x}\chi-it^{2\eta_z}\psi^2)=6.$$
\end{lemme}

\begin{proof}
Dans ce cas, on a  $\nu_v(t^{2\eta_x}\chi^2)=\nu_v(t^{3\eta_y}\varphi^3)=\nu_v(t^{4\eta_z}\psi^4)=12$, donc les $v$-parties principales   $\chi_v$, $\varphi_v$ et $\psi_v$ satisfont la relation suivante (voir la Relation (\ref{eq:1})):

$$ t^{3\eta_y}\varphi_v^3+(t^{\eta_x}\chi_v+it^{2\eta_z}\psi_v^2)(t^{\eta_x}\chi_v-it^{2\eta_z}\psi_v^2)=0.$$

 La relation ci-dessus  implique que $(t^{\eta_x}\chi_v+it^{2\eta_z}\psi_v^2)(t^{\eta_x}\chi_v-it^{2\eta_z}\psi_v^2)\neq0$ et par cons\'equent, on obtient que les $v$-ordres des s\'eries formelles $t^{\eta_x}\chi+it^{2\eta_z}\psi^2$  et  $t^{\eta_x}\chi-it^{2\eta_z}\psi^2$ sont \'egaux \`a $6$, car  les $v$-parties principales $t^{\eta_x}\chi_v+it^{2\eta_z}\psi_v^2$ et $ t^{\eta_x}\chi_v-it^{2\eta_z}\psi_v^2$  sont diff\'erentes de z\'ero et  les $v$-ordres des s\'eries formelles $t^{\eta_x}\chi$ et $t^{2\eta_z}\psi^{2}$ sont \'egaux \`a $6$. \end{proof}

\begin{lemme}
\label{le:pre0E6}
S'il existe une s\'erie formelle   irr\'eductible $\lambda$ qui divise $\chi$,$\varphi$, et $\psi$, alors $\lambda^2$ divise $\varphi$.
\end{lemme}
\begin{proof}
On suppose que  $\lambda$ divise $\chi$, $\varphi$, et $\psi$.  Soient $\chi=\lambda\chi_1$,  $\varphi=\lambda\varphi_1$ et $\psi=\lambda\psi_1$. Alors au moyen de la Relation (\ref{eq:1}) on obtient la relation suivante:

$$t^{2\eta_x}\chi_1^2+t^{3\eta_y}\lambda\varphi_1^3 +t^{4\eta_z}\lambda^2\psi_1^{4}=0,$$
ce qui implique que $\lambda$ divise $\chi_1$. En particulier $\lambda$ divise $\varphi_1$.\end{proof}

 \begin{lemme}
 \label{le:varinv}
 Si la s\'erie formelle $\varphi$ est inversible, alors les s\'eries formelles $\chi$ et $\psi$ sont inversibles.
 \end{lemme}
  \begin{proof}

 Raisonnons par l'absurde. On suppose que $\varphi$ est inversible et  que $\chi$ ou $\psi$ n'est pas inversible.\\
 
 D'apr\`es la Proposition \ref{pr:etaconeE6},  on a  deux possibilit\'es pour le vecteur $(\eta_x,\eta_y,\eta_z) $:

 \begin{itemize}
 \item[1)] $\eta_x=2\eta_z$ et $2\eta_x<3\eta_y$;
 
 \item[2)] $2\eta_x=3\eta_y$ et $\eta_x<2\eta_z$.\\
 \end{itemize}
 
  Cas 1).  On suppose que: $\eta_x=2\eta_z$  et  $2\eta_x<3\eta_y$.\\

   D'apr\`es  la Relation (\ref{eq:1}), on a:

$$ t^{3\eta_y-2\eta_x}\varphi^3+(\chi+i\psi^2)(\chi-i\psi^2)=0.$$

On remarque que $t$ divise $\chi+i\psi^{2}$ si et seulement si $t$ ne divise pas $\chi-i\psi^{2}$ car $t$ ne divise pas  $\chi$  et $ \varphi$. On peut donc sans perte de g\'en\'eralit\'e supposer que $t^{3\eta_y-2\eta_x}$ divise $\chi+i\psi^2$.\\

Si la s\'erie formelle $\varphi$ est inversible, alors la relation  $ t^{3\eta_y-2\eta_x}\varphi^3+(\chi+i\psi^2)(\chi-i\psi^2)=0$ \'equivaut au syst\`eme de relations suivant:
  
  $$t^{3\eta_y-2\eta_x}I_1=\chi+i\psi^2,\; I_2=-\chi+i\psi^2,$$
  o\`u $I_1,$ et $I_2$  sont deux s\'eries formelles inversibles de  $K[[s,t]]$ telles que $I_1I_2=\varphi^3$. Alors, les s\'eries formelles $\chi$ et $\psi$ sont inversibles, d'o\`u la contradiction.\\

  Cas 2).  On suppose que: $2\eta_x=3\eta_y$  et $\eta_x<2\eta_z$.\\
 
Comme  $2\eta_x=3\eta_y$, on obtient que  $\eta_x$ (resp. $\eta_y$) est divisible par $3$ (resp. $2$). Par cons\'equent, on a  $(\eta_x,\eta_y)=(3,2)$ et $2\leq \eta_z\leq 3$, car  $2\leq \eta_x\leq 6$, $2\leq \eta_y\leq 4$, $1\leq \eta_z\leq 3$ (voir la Remarque \ref{re:221eta643}) et $\eta_x<2\eta_z$.\\
 
 On rappelle que   $(\mu_x,\mu_y,\mu_z)\in \{(2,2,1),(3,2,2),(4,3,2),(6,4,3)\}$.
  Comme  la s\'erie formelle   $\chi$ ou  la s\'erie formelle  $\psi$  n'est pas  inversible, le  $v$-ordre  $\nu_v\chi$ ou le $v$-ordre $\nu_v\psi$  n'est pas nul (voir Proposition \ref{pr:FI-g}). Alors, on a $(\mu_x,\mu_y,\mu_z)\in \{(4,3,2),(6,4,3)\},$ d'o\`u    $1\leq \mu_y-\eta_y\leq2$. Par cons\'equent, la s\'erie formelle $\varphi$ n'est pas inversible, ce qui est une contradiction.
\end{proof}

 Maintenant, on consid\`ere le cas $i)$ de la Proposition \ref{pr:etaconeE6}.

 \begin{proposition}
\label{pr:2e6}
On suppose que $2\eta_x=3\eta_y$ et $\eta_x<2\eta_z$. Alors, les s\'eries formelles $\chi$, $\varphi$ et $\psi$ sont inversibles.
\end{proposition}
 \begin{proof}
Raisonnons par l'absurde. En vertu du Lemme \ref{le:varinv} on suppose que $\varphi$ n'est pas inversible. D'apr\`es la Proposition \ref{pr:FI-g}, on a $\mu_y-\eta_y\geq 1$.\\
 
Comme  $2\eta_x=3\eta_y$, on obtient que  $\eta_x$ (resp. $\eta_y$) est divisible par $3$ (resp. $2$). Par cons\'equent, on a  $(\eta_x,\eta_y)=(3,2)$ et $2\leq \eta_z\leq 3$, car  $2\leq \eta_x\leq 6$, $2\leq \eta_y\leq 4$, $1\leq \eta_z\leq 3$ (voir la Remarque \ref{re:221eta643}) et $\eta_x<2\eta_z$.\\

On rappelle que $(\mu_x,\mu_y,\mu_z)\in \{(2,2,1),(3,2,2),(4,3,2),(6,4,3)\}$. Comme $\eta_y=2$ et $\mu_y-\eta_y\geq 1$, on obtient que $(\mu_x,\mu_y,\mu_z)\in \{(4,3,2),(6,4,3)\}$. En particulier, $1 \leq \mu_y-\eta_y\leq 2$.

 D'apr\`es la Proposition \ref{pr:FI-g},  le nombre de facteurs irr\'eductibles compt\'es avec multiplicit\'e de $\varphi$  est   inf\'erieur ou \'egal \`a $2$.\\

Au moyen de  la Relation (\ref{eq:1})  on obtient la relation suivante:

$$ \varphi^3+(\chi+it^{2\eta_z-\eta_x}\psi^2)(\chi-it^{2\eta_z-\eta_x}\psi^2)=0.$$

On remarque que la s\'erie formelle $\chi+it^{2\eta_z-\eta_x}\psi^2$ est inversible  si et seulement si la s\'erie formelle $\chi-it^{2\eta_z-\eta_x}\psi^2$ l'est. Par cons\'equent, si $\chi+it^{2\eta_z-\eta_x}\psi^2$ est inversible, alors $\varphi$ l'est.
 On obtient donc  que les s\'eries formelles  $\chi+it^{2\eta_z-\eta_x}\psi^2$ et  $\chi-it^{2\eta_z-\eta_x}\psi^2$ ne sont pas inversibles.

La s\'erie formelle  $\varphi$ n'est pas irr\'eductible, car si $\varphi$ est irr\'eductible, alors  $\varphi$ divise $\chi$ et $\psi$, ce qui rentre en contradiction avec le Lemme  \ref{le:pre0E6}. On a donc  $\mu_y-\eta_y=2$ (voir la Proposition \ref{pr:FI-g}), ce qui implique que $(\mu_x,\mu_y,\mu_z)=(6,4,3)$, car $\eta_y=2$. 
On rappelle que  $2\leq \eta_z\leq 3$ et $\mu_z=3$, d'o\`u  la s\'erie formelle $\psi$  est irr\'eductible ou inversible, parce que $\mu_z-\eta_z\leq 1$  (voir la Proposition \ref{pr:FI-g}).\\

Comme $\varphi$ n'est pas irr\'eductible, le nombre de facteurs irr\'eductibles de $\varphi$ est \'egal \`a $2$. On a donc deux cas:\\

\begin{itemize}
\item[Cas 1).] La s\'erie formelle $\varphi$ est le produit de deux s\'eries formelles irr\'eductibles  associ\'ees.

\item[Cas 2).] La s\'erie formelle $\varphi$ est le produit de deux s\'eries formelles irr\'eductibles  non associ\'ees.\\

\end{itemize}

 Maintenant, on va montrer que dans ces deux cas  on arrive \`a  des contradictions, ce qui d\'emontre que les s\'eries formelles $\chi$, $\varphi$ et $\psi$ sont inversibles.\\

Cas 1).  On suppose que $\varphi=\varphi_1\varphi_2$,  o\`u  $\varphi_1$ et $\varphi_2$  sont deux s\'eries formelles irr\'eductibles associ\'ees.\\

Comme le corps $K$ est alg\'ebriquement clos,   on peut  supposer que $\varphi_1=\varphi_2$.\\

D'apr\`es la Relation (\ref{eq:1}), on a:

$$ \varphi_1^6+(\chi+it^{2\eta_z-\eta_x}\psi^2)(\chi-it^{2\eta_z-\eta_x}\psi^2)=0,$$
ce qui \'equivaut au syst\`eme de relations  suivant:

$$\varphi_1^3 I=-\chi - it^{2\eta_z-\eta_x}\psi^2,\;\; \varphi_1^{3}I^{-1}=\chi - it^{2\eta_z-\eta_x}\psi^2,$$
o\`u  $I$ est une s\'erie formelle inversible de $K[[s,t]]$. En effet, le vecteur  $(\mu_x,\mu_y,\mu_z)$ est \'egal \`a $(6,4,3)$, donc   le $v$-ordre $\nu_{v}(\chi + it^{2\eta_z-\eta_x}\psi^2)$ est \'egal au $v$-ordre $\nu_{v}(\chi - it^{2\eta_z-\eta_x}\psi^2)$ (voir le Lemme \ref{le:643}), ce qui implique le syst\`eme de relations ci-dessus. Par cons\'equent, on a la relation suivante:
  
 \begin{center} $2it^{2\eta_z-\eta_z}\psi^2=-(I+ I^{-1})\varphi_1^3.$\end{center}

On rappelle que la s\'erie formelle $\psi$  est irr\'eductible ou inversible,  parce que $\mu_z-\eta_z\leq 1$. Alors, la s\'erie formelle irr\'eductible $\varphi_1$ divise $t^{2\eta_y-\eta_z}$, d'o\`u une contradiction.  \\  

Cas 2). On suppose que $\varphi=\varphi_1\varphi_2$  o\`u $\varphi_1$  et  $\varphi_2$  sont s\'eries formelles irr\'eductibles non associ\'ees.\\

D'apr\`es la Relation (\ref{eq:1}), on a :

$$ \varphi_1^3\varphi_2^3+(\chi+it^{2\eta_z-\eta_x}\psi^2)(\chi-it^{2\eta_z-\eta_x}\psi^2)=0.$$

Au moyen des Lemmes \ref{le:643} et \ref{le:pre0E6} on obtient que cette relation \'equivaut au syst\`eme de relations  suivant:

\begin{center}$\varphi_1^3 I_1=-\chi - it^{2\eta_z-\eta_x}\psi^2,\;\; \varphi_2^{3}I_2=\chi - it^{2\eta_z-\eta_x}\psi^2,$\end{center}
o\`u  $I_1$, $I_2$  sont deux s\'eries formelles inversibles de $K[[s,t]]$ telles que $I_1I_2=1$, d'o\`u\\

\begin{center}
$-2it^{2\eta_z-\eta_x}\psi^2=\varphi_1^3I_1+\varphi_2^3I_2=\prod_{i=1}^3(\varphi_1J_1+w_i\varphi_2J_2),$
\end{center}
o\`u les $w_i$ sont les racines i-i\`emes de l'unit\'e et $J_1$, $J_2$ sont deux s\'eries formelles inversibles telles que   $J_1^3=I_1$ et $J_2^3=I_2$.

 On rappelle que la s\'erie formelle $\psi$  est irr\'eductible ou inversible.
Si $\psi$ est inversible,  alors les s\'eries formelles $\varphi_1$ et $\varphi_2$ sont inversibles ou divisibles par $t$. On a donc une contradiction dans les deux cas. Ainsi, on obtient que la s\'erie formelle $\psi$ est irr\'eductible. 

Comme $t$ ne divise pas $\varphi_1$ et $\varphi_2$, la s\'erie formelle $\psi$ divise $\varphi_1$ et $\varphi_2$. Par cons\'equent, la s\'erie $\psi$ divise $t$, d'o\`u une contradiction.  Ceci ach\`eve la d\'emonstration de la proposition.\end{proof}

Les deux lemmes suivants sont tr\`es importants dans la d\'emonstration.

\begin{lemme}
\label{le:1.1}
On suppose que  $(\mu_x,\mu_y,\mu_z)=(4,3,2)$. Alors, les s\'eries formelles $\chi$, $\varphi$ et $\psi$ sont inversibles.
\end{lemme}

\begin{proof} Raisonnons par l'absurde. En vertu du Lemme \ref{le:varinv} on suppose que $\varphi$ n'est pas inversible.
 Par cons\'equent, $\mu_y-\eta_y\geq 1$ (voir la Proposition \ref{pr:FI-g}).\\

En vertu de la Proposition \ref{pr:etaconeE6}, le vecteur $(\eta_x,\eta_y,\eta_z)$ satisfait une des relations suivantes:

 \begin{center}
  $\eta_x=2\eta_z$ et $2\eta_x<3\eta_y$; $\;\;\;\;$  $2\eta_x=3\eta_y$ et $\eta_x<2\eta_z$.\\
 \end{center} 
 
 D'apr\`es la proposition \ref{pr:2e6}, si on a $2\eta_x=3\eta_y$ et $\eta_x<2\eta_z$, alors  les s\'eries formelles $\chi$, $\varphi$ et $\psi$ sont inversibles. Par cons\'equent, le vecteur  $(\eta_x,\eta_y,\eta_z)$ satisfait la relation $\eta_x=2\eta_z\;\mbox{et}\; 2\eta_x<3\eta_y$. En particulier, $2$ divise $\eta_x$.  Ainsi, on obtient que $\eta_x\in \{2,4\}$, car $2\leq \eta_x\leq \mu_x=4$ (voir la Remarque \ref{re:221eta643}).
On rappelle que  $\mu_y-\eta_y\geq 1$. Comme on a $ \eta_y\geq 2$ (Remarque \ref{re:221eta643}) et $\mu_y=3$,   on obtient que $\eta_y=2$. On a donc $\eta_x<3$,  car  $2\eta_x<3\eta_y$. Ainsi, on obtient que $(\eta_x,\eta_y,\eta_z)=(2,2,1)$. Par cons\'equent, on a les $v$-ordres suivants:

$$  \nu_v\chi=2,\; \nu_v\varphi=1,\;\mbox{et}\; \nu_v\psi=1.$$

Comme les $v$-ordres $\nu_v\chi$, $\nu_v\varphi$ et $\nu_v\psi$ sont strictement positifs, les s\'eries formelles $\chi$, $\varphi$ et $\psi$ ne sont pas inversibles. De plus, les s\'eries formelles $\varphi$ et $\psi$ sont irr\'eductibles  (voir la Proposition \ref{pr:FI-g}).\\

 D'apr\`es la Relation (\ref{eq:1}), on a:

$$ t^{2}\varphi^3+(\chi+i\psi^2)(\chi-i\psi^2)=0.$$

On remarque que $t$ divise $\chi+i\psi^{2}$ si et seulement si $t$ ne divise pas $\chi-i\psi^{2}$, car $t$ ne divise pas  $\chi$  et $ \varphi$. On peut donc supposer, sans perte de g\'en\'eralit\'e,  que  $t^{2}$ divise $\chi+i\psi^2$.\\

Si la s\'erie formelle $\varphi$ divise  $\chi+i\psi^2$ et $\chi-i\psi^2$, alors $\varphi$ divise $\chi$ et $\psi$, ce qui rentre en contradiction avec le Lemme \ref{le:pre0E6}.
Comme les s\'eries formelles $\chi$ et $\psi$ ne sont pas inversibles et  par hypoth\`ese $t^2$ divise $\chi+i\psi^2$, la relation  $ t^{2}\varphi^3+(\chi+i\psi^2)(\chi-i\psi^2)=0$ \'equivaut au syst\`eme de relations suivant:
  
  $$t^2I=\chi+i\psi^2,\;\varphi^3 I^{-1}=-\chi+i\psi^2,$$
o\`u $I\in K[[s,t]]$ est un s\'erie formelle inversible.  Ainsi, on obtient la relation suivante:

$$2i\psi^2=t^2I+\varphi^3 I^{-1}.$$

Le corps $K$ est alg\'ebriquement clos, donc il existe $I_1\in K[[s,t]]^{\star}$ tel que $I_1^2=I$.
On a donc:

$$(\kappa \psi -tI_1t)(\kappa \psi+tI_1)=\varphi^3I^{-1},$$
o\`u $\kappa^2=2i$.  Comme les s\'eries formelles $\kappa \psi -tI_1$ et $\kappa \psi +tI_1$ ne sont pas inversibles et la s\'erie formelle $\varphi$ est irr\'eductible, $\varphi$ divise  $\kappa \psi -tI_1$ et $\kappa \psi +tI_1$ ce qui implique que $\varphi$ divise  la s\'erie formelle   $tI_1$, d'o\`u une contradiction car $\varphi$ n'est pas divisible par $t$.
\end{proof}

\begin{lemme}
\label{le:1.2}
On suppose que  $(\mu_x,\mu_y,\mu_z)=(6,4,3)$,  alors les s\'eries formelles $\chi$, $\varphi$ et $\psi$ sont inversibles.
\end{lemme}

\begin{proof} Raisonnons par l'absurde. En vertu du Lemme \ref{le:varinv} on suppose que $\varphi$ n'est pas inversible.\\

En vertu de la Proposition \ref{pr:etaconeE6}, le vecteur $(\eta_x,\eta_y,\eta_z)$ satisfait une des relations suivantes:

 \begin{center}
  $\eta_x=2\eta_z$ et $2\eta_x<3\eta_y$; $\;\;\;\;$  $2\eta_x=3\eta_y$ et $\eta_x<2\eta_z$.\\
 \end{center} 
 
 D'apr\`es la Proposition \ref{pr:2e6}, si on a $2\eta_x=3\eta_y$ et $\eta_x<2\eta_z$, alors  les s\'eries formelles $\chi$, $\varphi$ et $\psi$ sont inversibles. Par cons\'equent, le vecteur  $(\eta_x,\eta_y,\eta_z)$ satisfait la relation $\eta_x=2\eta_z\;\mbox{et}\; 2\eta_x<3\eta_y$.\\

On rappelle que  $2\leq \eta_x\leq 6$, $2 \leq \eta_y\leq 4$ et $1\leq \eta_z\leq 3$ (voir la Remarque \ref{re:221eta643}). Comme la s\'erie formelle $\varphi$ n'est pas inversible, on a $2\leq \eta_y\leq 3$, car $\eta_y<\mu_y=4$ (voir la Proposition \ref{pr:FI-g}).\\

Si on a $\eta_y=3$, alors on a $\mu_y-\eta_y=1$. Par cons\'equent, la s\'erie formelle $\varphi$ est irr\'eductible. \\

D'apr\`es la Relation (\ref{eq:1}), on a:

$$ t^{9-2\eta_x}\varphi^3+(\chi+i\psi^2)(\chi-i\psi^2)=0.$$

On remarque que $t$ divise $\chi+i\psi^{2}$ si et seulement si $t$ ne divise pas $\chi-i\psi^{2}$, car $t$ ne divise pas  $\chi$  et $ \psi$. On peut donc supposer, sans perte de g\'en\'eralit\'e,  que $t^{9-2\eta_x}$ divise $\chi+i\psi^2$. On a donc  le syst\`eme de relations  suivant:

$$t^{9-2\eta_x}\varphi^j I= \chi + i\psi^2,\;\; \varphi^{3-j}I^{-1}=\chi - i\psi^2,$$
o\`u $0\leq j \leq 3$ et $I$ est un \'el\'ement inversible de $K[[s,t]]$. \\

Si on a $j=3$, alors la s\'erie formelle $\chi - i\psi^2$ est inversible, d'o\`u on a $\nu_v(\chi - i\psi^2)=0$. 
Le Lemme \ref{le:643} et  la propri\'et\'e multiplicative du $v$-ordre montrent que  le $v$-ordre $\nu_v(\chi + i\psi^2)$ est \'egal \`a z\'ero, ce qui implique que le $v$-ordre de $t^{9-2\eta_x}\varphi^3$ est \'egal \`a z\'ero, d'o\`u une contradiction.\\

 Si on a $1\leq j\leq 2$, alors $\varphi$ divise les s\'eries formelles $\chi$ et $\psi$ ce qui rentre en contradiction avec le Lemme  \ref{le:pre0E6}. On a donc  le syst\`eme de relations  suivant:

$$t^{9-2\eta_x} I= \chi + i\psi^2,\;\; \varphi^{3}I^{-1}=\chi - i\psi^2.$$

Par cons\'equent,  le  $v$-ordre  $\nu_v(\chi + i\psi^2)$ (resp. le $v$-ordre  $\nu_v(\chi- i\psi^2)$) est \'egal \`a $9-2\eta_x$ (resp. est \'egal \`a 3). Le Lemme \ref{le:643} et la propri\'et\'e multiplicative du $v$-ordre montrent qu'on a  $\nu_v(\chi + i\psi^2)= \nu_v(\chi - i\psi^2)$, ce qui implique que $\eta_x=3$, d'o\`u la contradiction car $\eta_x=2\eta_z$.
Forc\'ement,  on a donc $\eta_y=2$.\\

Comme on a $\eta_x=2\eta_z\;\mbox{et}\; 2\eta_x<3\eta_y$  et  $\eta_y=2$, on obtient que  $(\eta_x,\eta_y,\eta_z)=(2,2,1)$, d'o\`u 

$$ t^{2}\varphi^3+(\chi+i\psi^2)(\chi-i\psi^2)=0.$$

Comme  on a $\nu_v\varphi=\mu_y-\eta_y=2$, la s\'erie formelle  $\varphi$ est au plus le produit de deux s\'eries formelles irr\'eductibles compt\'es avec multiplicit\'e, voir la Proposition  \ref{pr:FI-g}.\\

Alors,   pour la s\'erie formelle $\varphi$ on obtient le trois cas  suivants:\\

\begin{itemize}
\item[Cas 1).] la s\'erie formelle $\varphi$ est irr\'eductible;

\item[Cas 2).] la s\'erie formelle $\varphi$ est le produit de deux s\'eries formelles non associ\'ees;

\item[Cas 3).] la s\'erie formelle $\varphi$ est le produit de deux s\'eries formelles  associ\'ees.\\
\end{itemize}

Maintenant, on va montrer  que dans chaque cas ci-dessus  on obtient une contradiction,  ce qui implique que $\varphi$ est inversible, d'o\`u le lemme.\\

On rappelle que dans le trois cas  on a la relation suivante:

$$ t^{2}\varphi^3+(\chi+i\psi^2)(\chi-i\psi^2)=0.$$

On remarque que $t$ divise $\chi+i\psi^{2}$ si et seulement si $t$ ne divise pas $\chi-i\psi^{2}$ car $t$ ne divise pas  $\chi$  et $ \varphi$. On peut donc supposer, sans perte de g\'en\'eralit\'e,  que $t^{2}$ divise $\chi+i\psi^2$.\\

 Cas 1). On suppose que $\varphi$  est une s\'erie formelle irr\'eductible.\\

 Si la s\'erie formelle $\varphi$ est irr\'eductible, alors la relation  $ t^{2}\varphi^3+(\chi+i\psi^2)(\chi-i\psi^2)=0$ \'equivaut au  syst\`eme de relations suivant:

$$-t^2\varphi^j I=\chi + i\psi^2,\;\; \varphi^{3-j}I^{-1}=\chi - i\psi^2\mbox{,}$$
o\`u $0\leq j \leq 3$ et $I$ est une s\'erie formelle inversible de $K[[s,t]]$.\\ 

Le Lemme \ref{le:643} et la propri\'et\'e multiplicative du $v$-ordre montrent  qu'on  a  $\nu_v(\chi + i\psi^2)= \nu_v(\chi - i\psi^2)$, ce qui implique que $2(1+j)=2(3-j)$ (on rappelle que le $v$-ordre $\nu_v\varphi$ est \'egal \`a $2$), d'o\`u $j=1$. Alors, la s\'erie formelle $\varphi$ divise les s\'eries formelles  $\chi$ et $\psi$, ce qui rentre en contradiction avec le Lemme \ref{le:pre0E6}.\\

 Cas 2).  On suppose que $\varphi=\varphi_1\varphi_2$ o\`u $\varphi_1$ et  $\varphi_2$  sont deux s\'eries formelles irr\'eductibles non associ\'ees.\\

Dans ce cas, les $v$-ordres  $\nu_v\varphi_1$ et $\nu_v\varphi_2$ sont \'egaux \`a $1$, car $\nu_v\varphi=2$.\\

De la m\^eme fa\c con que dans le  cas pr\'ec\'edent,  on a le syst\`eme de relations  suivant:

$$t^2\varphi_1^j \varphi_2^kI=-(\chi + i\psi^2),\;\; \varphi_1^{3-j}\varphi_2^{3-k}I^{-1}=(\chi - i\psi^2),$$
o\`u $0\leq j\leq 3,\;0\leq k\leq 3$ et $I$ est un \'el\'ement inversible de $K[[s,t]]$.\\ 

Le Lemme \ref{le:643} et la propri\'et\'e multiplicative du $v$-ordre montrent qu'on  a  $\nu_v(\chi + i\psi^2)= \nu_v(\chi - i\psi^2)$, ce qui implique que $2+j+k=6-j-k$, d'o\`u $j+k=2$.  Pour tous les  $j$ et $k$ tels que $j+k=2$, on rentre  en contradiction avec le Lemme \ref{le:pre0E6}.\\

 Cas 3).  On suppose que $\varphi=\varphi_1\varphi_2$,  o\`u  $\varphi_1$   et  $\varphi_2$  sont des s\'eries formelles irr\'eductibles associ\'ees.\\

Comme  le corps $K$ est alg\'ebriquement clos,   on peut   supposer que $\varphi_1=\varphi_2$. Dans ce cas, le $v$-ordre  $\nu_v\varphi_1$ est \'egal \`a $1$. \\

La relation $ t^{2}\varphi^3+(\chi+i\psi^2)(\chi-i\psi^2)=0$ \'equivaut \`a la relation:

$$ t^2\varphi_1^6+(\chi+i\psi^2)(\chi-i\psi^2)=0,$$
et par cons\'equent, on obtient le syst\`eme de relations suivant:

$$-t^2\varphi_1^j I=\chi + i\psi^2,\;\; \varphi_1^{6-j}I^{-1}=\chi - i\psi^2,$$
o\`u $0\leq j\leq 6$ et $I$ est un \'el\'ement inversible de $K[[s,t]]$. Le Lemme \ref{le:643} et la propri\'et\'e multiplicative du $v$-ordre montrent  qu'on  a  $\nu_v(\chi + i\psi^2)= \nu_v(\chi - i\psi^2)$, ce qui implique que $2+j=6-j$, d'o\`u $j=2$. On a donc: 

 $$2i\psi^2=-(t^2I+\varphi_1^2 I^{-1})\varphi_1^2 .$$

Comme $K$ est alg\'ebriquement clos, il existe une s\'erie formelle inversible $I_1$ appartenant \`a $K[[s,t]]$ telle que  $I_1^2=I$. Alors, on a:

$$2i\psi^2=-(tI_1+i\varphi_1 I_1^{-1})(tI_1-i\varphi_1 I_1^{-1})\varphi_1^2 .$$

Soient $r:=(tI_1+i\varphi_1 I_1^{-1})$ et $q:=(tI_1-i\varphi_1 I_1^{-1})$. Les s\'eries formelles 
$r$ et $q$ ne sont pas inversibles car  les s\'eries formelles $tI_1$ et $\varphi_1$ ne le sont pas.\\

Comme on a $(\eta_x,\eta_y,\eta_z)=(2,2,1)$ et $(\mu_x,\mu_y,\mu_z)=(6,4,3)$,  le $v$-ordre $\nu_v\psi$ est \'egal \`a $2$ (on remarque que $\mu_z-\eta_z=2$),  ce qui implique que $\psi$ est au plus le produit de deux s\'eries formelles irr\'eductibles compt\'ees avec multiplicit\'e (voir la Proposition \ref{pr:FI-g}). Par cons\'equent, la s\'erie formelle $\psi^2$ est au plus le produit de quatre s\'eries formelles irr\'eductibles compt\'ees avec multiplicit\'e,  ce qui implique que les s\'eries formelles $r$ et $q$ sont irr\'eductibles.\\

Les s\'eries formelles $r$ et $q$ ne sont pas associ\'ees  \`a la s\'erie formelle $\varphi_1$ parce que $t$ ne divise pas $\varphi_1$.\\

 On obtient  donc:
$$(tI_1+i\varphi_1 I_1^{-1})=J(tI_1-i\varphi_1 I_1^{-1}),$$
o\`u $J$ est une s\'erie formelle inversible.  En effet, la s\'erie formelle $r$ divise $\psi^2$, alors  $r$ divise $\psi$, d'o\`u $r^2$ divise  $\psi^2$. Comme $r$ ne divise pas $\varphi_1$, $r$ divise $q$. Par cons\'equent $r$ et $q$ sont deux  s\'eries irr\'eductibles associ\'ees.\\

 Ainsi,  on obtient la relation suivante:\\
$$t(J-1)I_1=i\varphi_1(J+1)I_1^{-1}.$$

On remarque que $J-1$ ou $J+1$ est inversible.
Si  $J-1$ est  inversible, alors la s\'erie $\varphi_1$ est inversible ou $t$ divise $\varphi_1$; ce qui est absurde. Si  $J+1$ est inversible,  alors $t$ divise $\varphi_1$, ce qui est aussi une contradiction. Ceci ach\`eve la d\'emonstration du lemme. \end{proof}

Maintenant, on consid\`ere le cas {\it ii}) de la proposition \ref{pr:etaconeE6}.

\begin{proposition}
\label{pr:1e6}
On suppose que $\eta_x=2\eta_z$  et $2\eta_x<3\eta_y$. Alors, les  s\'eries formelles $\chi$, $\varphi$ et $\psi$ sont inversibles.

\end{proposition}

\begin{proof}
Raisonnons par l'absurde. En vertu du Lemme \ref{le:varinv}, on suppose que $\varphi$ n'est pas inversible.\\

On rappelle que  $2\leq \eta_x\leq 6$, $2 \leq \eta_y\leq 4$ et $1\leq \eta_z\leq 3$ (voir la Remarque \ref{re:221eta643}). Comme on a  $\eta_x=2\eta_z$  et $2\eta_x<3\eta_y$,  on obtient que  $(\eta_x,\eta_z)\in\{(2,1),(4,2)\}$ et  que $2\leq \eta_y \leq 4$.\\

On rappelle que $(\mu_x,\mu_y,\mu_z)\in \{(2,2,1),(3,2,2),(4,3,2),(6,4,3)\}$. Comme $\varphi$ n'est pas inversible et  $2\leq \eta_y \leq 4$, le vecteur  $(\mu_x,\mu_y,\mu_z)\in \{(4,3,2),(6,4,3) \}$. En vertu des Lemmes \ref{le:1.1} et \ref{le:1.2}, si le vecteur  $(\mu_x,\mu_y,\mu_z)$ appartient \`a l'ensemble $ \{(4,3,2),(6,4,3) \}$, alors les s\'eries formelles $\chi$, $\varphi$ et $\psi$ sont inversibles, d'o\`u une  contradiction. \end{proof}

La d\'emonstration de la bijectivit\'e de l'application de Nash pour la singularit\'e de type $\EE_6$  r\'esulte  des propositions \ref{pr:invserBP}, \ref{pr:etaconeE6}, \ref{pr:2e6} et \ref{pr:1e6}.
\subsection{La singularit\'e de type $\EE_7$}

La preuve de la bijectivit\'e de l'application de Nash pour le cas de la singularit\'e de type  $\EE_7$ (Th\'eor\`eme \ref{th:nashDP}) repose sur la construction d'une $G$-d\'esingularisation de l'\'eventail de Newton et sur des propri\'et\'es, pour  $\EE_7$, des s\'eries formelles $\chi$, $\varphi$ et $\psi$ analogues aux s\'eries d\'efinies pour $\B{p}{q}$ ou $\EE_6$. Dans la suite, on donne un r\'esum\'e de la preuve.\\

Soit $S$ l'hypersurface normale de $\AF_{\KK}^3$ donn\'ee par l'\'equation $x^2+y(y^2+z^3)=0$. L'hypersurface $S$ a un unique point singulier de type $\EE_7$ \`a l'origine de $\AF^3$. 

On consid\`ere l'\'eventail de Newton $\EN{f}$ associ\'e \`a $\f:=x^2+y(y^2+z^3)$ et on note $\tau_1$ (resp. $\tau_2$, $\tau_3$)  le c\^one engendr\'e par les vecteurs $(1,0,0)$ (resp. $(1,2,0)$, $(0,0,1)$) et $(9,6,4)$ (voir la Figure \ref{fi:ENE7}).

\begin{figure}[!h]
\begin{tikzpicture}[font=\small]
   \draw[thick] (0,0)  -- (90:3) (0,0)  -- (210:3) (0,0) -- (310.89:1.984);
   \draw[thick] (330:3) --  (90:3)  -- (210:3) -- (330:3);
\draw (310.89:2.2) node[inner sep=-1pt,below=-1pt,rectangle,fill=white] {{\tiny $(1,2,0)$}};
\draw (0.5,0.2) node[inner sep=-1pt,below=-1pt,rectangle,fill=white] {{\tiny (9,6,4)}};
\draw (335:3.5) node[inner sep=-1pt,below=-1pt,rectangle,fill=white] {{\tiny $(0,1,0)$}};
\draw (325:1.2) node[inner sep=-1pt,below=-1pt,rectangle,fill=white] {$\tau_2$};
\draw (205:3.5) node[inner sep=-1pt,below=-1pt,rectangle,fill=white] {{\tiny $(1,0,0)$}};
\draw (200:1.5) node[inner sep=-1pt,below=-1pt,rectangle,fill=white] {$\tau_1$};
\draw (90:3.3) node[inner sep=-1pt,below=-1pt,rectangle,fill=white] {{\tiny $(0,0,1)$}};
\draw (100:1.5) node[inner sep=-1pt,below=-1pt,rectangle,fill=white] {$\tau_3$};
    \draw[fill=black]  (0,0) circle(0.7mm) (330:3) circle(0.7mm)  (90:3) circle(0.7mm)  (210:3) circle(0.7mm) (90:3) circle(0.7mm) 
(310.89:1.984) circle(0.7mm);

  \end{tikzpicture}

  \caption{\'Eventail de Newton $\EN{f}$}
\label{fi:ENE7}
\end{figure}

 Soit  $\GD{\f}$ une $G$-{\it subdivision r\'eguli\`ere}  de  $\EN{\f}$. On note  $\pi_{\mathcal{N}}:X(\EN{\f})\rightarrow \AF^{3}_{\KK}$  (resp. $\pi_{\mathcal{G}}:X(\GD{\f})\rightarrow X(\EN{\f})$)  le morphisme torique induit par la subdivision $\EN{\f}$ de $\RR^3_{\geq 0}$ (resp. $\GD{\f}$ de  $\EN{\f}$) et $S_{\mathcal{G}}$ le transform\'e strict de $S$ associ\'e au morphisme $\pi:=\pi_{\mathcal{G}}\circ \pi_{\mathcal{N}}$. 

De mani\`ere analogue \`a la Proposition \ref{pr:desingE6}, le morphisme  $\pi:S_{\mathcal{G}}\rightarrow S$ est la r\'esolution minimale de $S$.\\

Montrer la bijectivit\'e  de l'application de Nash $\mathcal{N}_{S}$ \'equivaut \`a
    montrer que, pour chaque diviseur essentiel $\E$, tous les $K$-wedges admissibles centr\'es en $N_{\E}$  se rel\`event \`a la r\'esolution minimale de $S$. Dans la suite, on donne une id\'ee de pourquoi un $K$-wedge admissible centr\'e en $N_{\E}$ se rel\`eve \`a $\SD$.\\

Soient $\E\in \Ess(S)$ et $\omega:\spec K[[s,t]]\rightarrow S$ un $K$-wedge admissible centr\'e en $N_{\E}$. On rappelle que $\alpha_{\E}$ est le point g\'en\'erique de $N_{\E}$. On consid\`ere les vecteurs suivants:

\begin{center} $(\mu_x,\mu_y,\mu_z):=(\ordt{\com{\alpha_{\E}}}(x),\ordt{\com{\alpha_{\E}}}(y), \ordt{\com{\alpha_{\E}}}(z))$;\end{center}

\begin{center} $(\eta_x,\eta_y,\eta_z): =(\ordt \omega^{\star }(x),\ordt \omega^{\star }(y),\ordt \omega^{\star }(z))$, \end{center}
o\`u $\com{\alpha_{\E}}$ (resp. $\com{\omega}$)  est le comorphisme de $\alpha_{\E}$ (resp. de $\omega$). On peut donc \'ecrire le comorphisme  de  $\omega$  de la fa\c con  suivante:
 
\begin{center} 
$\omega^{\star }(x)=t^{\eta_x}\chi,\;\omega^{\star }(y)=t^{\eta_y}\varphi,\; \omega^{\star }(z)=t^{\eta_z}\psi,$
\end{center}
o\`u les s\'eries formelles $\chi$, $\varphi$, $\psi$ ne sont pas divisibles par $t$. Le $K$-wedge $\omega$ doit satisfaire l'\'equation $x^2+y(y^2+z^{3})=0$, d'o\`u la relation suivante:
             
\begin{equation} t^{2\eta_x}\chi^2+t^{3\eta_y}\varphi^3+t^{\eta_y+3\eta_z}\varphi\psi^{3}=0. \label{eq:E7}
\end{equation}

En vertu des Propositions \ref{pr:invserBP}, si les s\'eries formelles $\chi$, $\varphi$ et $\psi$ sont inversibles, alors le $K$-wedge admissible $\omega$ centr\'e en $N_{\E}$ se rel\`eve \`a $\SD$.\\

D'apr\`es la Proposition \ref{pr:FI-g}, on peut majorer le nombre de facteurs irr\'eductibles compt\'es avec multiplicit\'e des s\'eries formelles $\chi$, $\varphi$ et $\psi$ \`a l'aide des $v$-ordres de la fa\c con suivant:

\begin{center}
$\FI(\chi)\leq \degt{\chi_v}= \nu_v\chi =\mu_x-\eta_x\leq 6$;

 $\FI(\varphi)\leq \degt{\varphi_v}= \nu_v\varphi= \mu_y-\eta_y\leq 4$;

$\FI(\psi)\leq \degt{\psi_v}= \nu_v\psi= \mu_z-\eta_z\leq 3$.
\end{center}
 
La proposition suivante est une propri\'et\'e importante du vecteur $(\mu_x,\mu_y,\mu_z)$.
  Pour un c\^one  $\tau$ dans $\EN{\f}$, on note $G\tau$ le syst\`eme g\'en\'erateur minimal du semi-groupe $\tau\cap \ZZ^3$.

\begin{proposition}
\label{pr:musygemiE7}
 Le vecteur $(\mu_x,\mu_y,\mu_z)$ appartient \`a l'union   $G\tau_1\cup G\tau_2\cup G\tau_3$, o\`u $\tau_1$  (resp. $\tau_2$ $\tau_3$)  est le c\^one engendr\'e par les vecteurs $(1,0,0)$ (resp. $(1,2,0)$,  $(0,0,1)$) et $(9,6,4)$ (voir la figure \ref{fi:ENE7}).  Autrement dit,  $(\mu_x,\mu_y,\mu_z)\in \{(3, 2, 2), (6, 4, 3), (9, 6, 4), (7, 5, 3),$ $(5, 4, 2), (3, 3, 1), (5, 3, 2) \}$.
\end{proposition}

Si on suppose les s\'eries $\chi$, $\varphi$ et $\psi$ non simultan\'ement inversibles,  on obtient la proposition suivante: 

\begin{proposition}
\label{pr:etaE7}
 S'il existe au moins  une s\'erie formelle parmi les s\'eries  $\chi$, $\varphi$, $\psi$ qui n'est pas  inversible, alors le vecteur $(\eta_x,\eta_y,\eta_z)$ satisfait une des relations suivantes:

 \begin{itemize}
 \item[i)] $2\eta_x=\eta_y+3\eta_z$ et $2\eta_x<3\eta_y$;

 \item[ii)] $2\eta_x=3\eta_y$ et $2\eta_y<3\eta_z$;

 \item[iii)] $2\eta_y=3\eta_z$ et $3\eta_y<2\eta_x$.
 \end{itemize}
\end{proposition}
La Proposition \ref{pr:etaE7} donne trois possibilit\'es pour le vecteur  $(\eta_x,\eta_y,\eta_z)$. Dans la suite, on consid\`ere chaque cas s\'epar\'ement pour obtenir des contradictions.\\

  Dans chaque cas du vecteur $(\eta_x,\eta_y,\eta_z)$, on peut simplifier la Relation (\ref{eq:E7}) et on peut la r\'ecrire  comme la somme de deux termes. Ceci permet  d'\'etablir des relations entre les  facteurs irr\'eductibles  des d\'ecompositions possibles des s\'eries  $\chi$, $\varphi$ et $\psi$. L'id\'ee est de trouver des contradictions  dans les diff\'erents cas de d\'ecomposition obtenus.\\

D'abord, \'enon\c cons le lemme suivant qui est obtenu en utilisant la Relation (\ref{eq:E7}).
 
\begin{lemme}
\label{le:varphi-no-irrE7}
 Si on suppose que  $\varphi=\varphi_1\varphi_2$,  o\`u  $\varphi_1$ est une s\'erie formelle  irr\'eductible,  alors  la s\'erie formelle $\varphi_1$ divise la s\'erie formelle $\varphi_2$. En particulier, la s\'erie formelle $\varphi$ n'est pas irr\'eductible. 
\end{lemme}
  Maintenant, on consid\`ere le premier cas de la Proposition \ref{pr:etaE7}.
\begin{proposition}
\label{pr:cas1etaE7}
   On suppose que $2\eta_x=\eta_y+3\eta_z$ et $2\eta_x<3\eta_y$. Alors, les s\'eries  $\chi$, $\varphi$, $\psi$ sont inversibles.
 \end{proposition}
\begin{proof}[Id\'ee de la d\'emonstration]  On remarque que d'apr\`es la relation (\ref{eq:E7}), on obtient la relation suivante:
\begin{equation}
\label{eq:C1E7}
 \chi^2+t^{3\eta_y-2\eta_x}\varphi^3+\varphi\psi^{3}=0.
\end{equation}

D'abord, on montre que la s\'erie $\varphi$ est inversible, ce qui \'equivaut \`a montrer que $\mu_y-\eta_y=0$.  En utilisant les relations  $2\eta_x=\eta_y+3\eta_z$ et $2\eta_x<3\eta_y$, on obtient que  $\eta_y\geq 3$.  En vertu du Lemme \ref{le:varphi-no-irrE7} et la Proposition \ref{pr:musygemiE7}, on obtient que $\mu_y-\eta_y\in \{0,2,3\}$.\\

Si on suppose que $\mu_y-\eta_y=2$, on obtient que $3\eta_y-2\eta_x\in \{2, 3\}$. D'apr\`es le lemme \ref{le:varphi-no-irrE7}, on a $\varphi=\varphi_1^2$, o\`u $\varphi_1$ est irr\'eductible.\\

Dans le cas $3\eta_y-2\eta_x=2$, on peut \'ecrire la Relation (\ref{eq:C1E7}) comme la somme de deux termes qui nous permet d'\'etablir des relations entre les facteurs  irr\'eductibles des s\'eries $\chi$, $\varphi$ et $\psi$ et  de trouver une contradiction.

 Dans le cas $3\eta_y-2\eta_x=3$, on ne peut pas \'ecrire de fa\c con \'evidente la Relation (\ref{eq:C1E7}) comme la somme de deux termes.  Cependant, on peut consid\'erer le changement de variable  $t=u^2$ et \'etudier la Relation (\ref{eq:C1E7}) dans l'anneau $K[[s,u]]$. Dans cet anneau, on peut \'ecrire cette relation comme la somme de deux termes et on traite ce cas comme le  pr\'ec\'edent.\\

Si on suppose que $\mu_y-\eta_y=3$, on obtient  $3\eta_y-2\eta_x=3$. D'apr\`es le lemme \ref{le:varphi-no-irrE7}, on a $\varphi=\varphi_1^3$, o\`u $\varphi_1$ est irr\'eductible. Dans ce cas, on peut \'ecrire la Relation (\ref{eq:C1E7}) comme la somme de deux termes et on le traite comme les cas pr\'ec\'edents.\\

 D'apr\`es la Proposition \ref{pr:musygemiE7}, le vecteur $(\mu_x,\mu_y,\mu_z)$ satisfait une des relations suivantes: 

\begin{itemize}

 \item[i)] $2\mu_x=3\mu_y$ et $2\mu_y  <  3\mu_z$;

 \item[ii)] $2\mu_x=3\mu_z$ et $3\mu_y  <  2\mu_x$;
\item[iii)]$2\mu_x=\mu_y+3\mu_z$ et $2\mu_x\leq 3\mu_y$.
\end{itemize}

En utilisant que  $\mu_y-\eta_y=0$, on obtient le lemme suivant:

\begin{lemme}
\label{le:vmu}
 Le vecteur $(\mu_x,\mu_y,\mu_z)$ ne satisfait pas les  relations suivantes:

\begin{itemize}

 \item[i)] $2\mu_x=3\mu_y$ et $2\mu_y  <  3\mu_z$;

 \item[ii)] $2\mu_x=3\mu_z$ et $3\mu_y  <  2\mu_x$.
\end{itemize}
\end{lemme}
En vertu du Lemme \ref{le:vmu}, on peut supposer que le vecteur $(\mu_x,\mu_y,\mu_z)$ satisfait la relation  suivante: $2\mu_x=\mu_y+3\mu_z$ et $2\mu_x\leq 3\mu_y$.

Maintenant, en raisonnant par l'absurde, on suppose que la s\'erie $\chi$ ou la s\'erie $\psi$ n'est pas inversible. On peut donc \'etablir une liste de cas possibles pour les vecteurs $(\mu_x,\mu_y,\mu_z)$ et $(\eta_x,\eta_y,\eta_z)$. En utilisant la Relation (\ref{eq:C1E7}), la Proposition \ref{pr:musygemiE7} et les majorants du nombre de facteurs irr\'eductibles compt\'es avec multiplicit\'e des s\'eries formelles $\chi$ et $\psi$, on obtient une contradiction dans chaque cas de la liste, d'o\`u la proposition.     
 \end{proof}
La preuve des propositions suivantes est  analogue  \`a celle de la proposition \ref{pr:cas1etaE7}.
\begin{proposition}
\label{pr:cas2etaE7}
   On suppose que $2\eta_x=3\eta_y$ et $2\eta_y<3\eta_z$. Alors, les s\'eries  $\chi$, $\varphi$, $\psi$ sont inversibles.
 \end{proposition}
\begin{proposition}
\label{pr:cas3etaE7}
   On suppose que $2\eta_y=3\eta_z$ et $3\eta_y<2\eta_x$. Alors, les s\'eries  $\chi$, $\varphi$, $\psi$ sont inversibles.
 \end{proposition}

La d\'emonstration du cas $\EE_7$ du Th\'eor\`eme \ref{th:nashDP} r\'esulte  des propositions \ref{pr:etaE7}, \ref{pr:cas1etaE7}, \ref{pr:cas2etaE7} et \ref{pr:cas3etaE7}.
\section[Les singularit\'es de type $\DD_n$]{Une nouvelle preuve de la bijectivit\'e de l'application de Nash pour les singularit\'es de type $\DD_n$}

Dans l'article \cite{Ple08} l'auteur  d\'emontre la bijectivit\'e de l'application de Nash $\mathcal{N}_{\DD_n}$ associ\'ee \`a une singularit\'e de type $\DD_n$, $n\geq 4$. Dans cette section, en utilisant les m\^emes m\'ethodes que dans les preuves des Th\'eor\`emes \ref{th:nashBP} et \ref{th:nashDP}, on donne  une d\'emonstration de la bijectivit\'e de l'application $\mathcal{N}_{\DD_n}$  diff\'erente de celle de \cite{Ple08}.\\

 L'entier $n\geq 4$ \'etant fix\'e, soit $S$ l'hypersurface normale de $\AF_{\KK}^3$ donn\'ee par l'\'equation  $x^2+z(y^2+z^{n-2})=0$. L'hypersurface $S$ a un unique point singulier de type $\DD_n$.\\

Comme dans la preuve des Th\'eor\`emes \ref{th:nashBP} ou \ref{th:nashDP} pour prouver la bijectivit\'e de l'application de Nash associ\'ee \`a $S$, on a besoin de quelques r\'esultats sur la r\'esolution minimale de $S$.\\

On consid\`ere l'\'eventail de Newton  $\EN{\f}$ associ\'e \`a $\f:=x^2+z(y^2+z^{n-2})$. Soit $H$ un plan de $\RR^3$ qui ne contient pas l'origine de $\RR^3$ et tel que l'intersection de $H$ et $\RR^3_{\geq 0}$ soit un ensemble compact. La Figure \ref{fi:ENDn}  repr\'esente l'intersection de $H$ avec la subdivision $\EN{\f}$ de $\RR^3_{\geq 0}$. Chaque sommet du diagramme est identifi\'e avec le {\it vecteur extr\'emal}  correspondant.  On note $\tau_1$ (resp. $\tau_2$, $\tau_3$)  le c\^one engendr\'e par les vecteurs $(1,0,0)$ (resp. $(0,1,0)$, $(0,0,1)$) et $(n-1,n-2,2)$.\\

\begin{figure}[!h]
\begin{tikzpicture}[font=\small]
   \draw[thick] (0,0)  -- (330:3) (0,0)  -- (210:3) (0,0) -- (-0.6495,1.8750);
   \draw[thick] (330:3) --  (90:3)  -- (210:3) -- (330:3);
\draw (-1.2,1.8750) node[inner sep=-1pt,below=-1pt,rectangle,fill=white] {{\tiny $(1,0,2)$}};
\draw (0.85,0.2) node[inner sep=-1pt,below=-1pt,rectangle,fill=white] {{\tiny (n-1,n-2,2)}};
\draw (335:3.5) node[inner sep=-1pt,below=-1pt,rectangle,fill=white] {{\tiny $(0,1,0)$}};
\draw (340:1.5) node[inner sep=-1pt,below=-1pt,rectangle,fill=white] {$\tau_2$};
\draw (205:3.5) node[inner sep=-1pt,below=-1pt,rectangle,fill=white] {{\tiny $(1,0,0)$}};
\draw (200:1.5) node[inner sep=-1pt,below=-1pt,rectangle,fill=white] {$\tau_1$};
\draw (90:3.3) node[inner sep=-1pt,below=-1pt,rectangle,fill=white] {{\tiny $(0,0,1)$}};
\draw (100:1.5) node[inner sep=-1pt,below=-1pt,rectangle,fill=white] {$\tau_3$};
    \draw[fill=black]  (0,0) circle(0.7mm) (330:3) circle(0.7mm)  (90:3) circle(0.7mm)  (210:3) circle(0.7mm) (90:3) circle(0.7mm) 
(-0.6495,1.8750) circle(0.7mm);

  \end{tikzpicture}

  \caption{\'Eventail de Newton $\EN{f}$}
\label{fi:ENDn}
\end{figure}

Pour un c\^one  $\tau$ dans $\EN{\f}$, on note $G\tau$ le syst\`eme g\'en\'erateur minimal du semi-groupe $\tau\cap \ZZ^3$. Par un calcul direct, on obtient le r\'esultat suivant:
\begin{proposition} 
\label{pr:GtauDn}
Le syst\`eme g\'en\'erateur minimal du semi-groupe $\tau_3\cap \ZZ^3$ est l'ensemble 
 $G\tau_3=\{(j,j-1,2)\mid j\in \{1,2,...,n-1\}\}$.
De plus, soit $k\geq 2$ un entier,  on a:

\begin{itemize} 
 \item[i)] Si $n=2k$, alors le syst\`eme g\'en\'erateur minimal du semi-groupe $\tau_1\cap \ZZ^3$ est l'ensemble $G\tau_1=\{(1,0,0),(k,k-1,1),(n-1,n-2,2)\}$ et $\tau_2$ est un c\^one r\'egulier. 
\item[ii)] Si $n=2k-1$,  alors le syst\`eme g\'en\'erateur minimal du semi-groupe $\tau_2\cap \ZZ^3$ est l'ensemble  $G\tau_2=\{(0,1,0),(k-1,k-1,1),(n-1,n-2,2)\}$ et $\tau_1$ est un c\^one r\'egulier. 
\end{itemize}
\end{proposition}

 Soit  $\GD{\f}$ une $G$-{\it subdivision r\'eguli\`ere}  de  $\EN{\f}$. On note  $\pi_{\mathcal{N}}:X(\EN{\f})\rightarrow \AF^{3}_{\KK}$  (resp. $\pi_{\mathcal{G}}:X(\GD{\f})\rightarrow X(\EN{\f})$)  le morphisme torique induit par la subdivision $\EN{\f}$ de $\RR^3_{\geq 0}$ (resp. $\GD{\f}$ de  $\EN{\f}$) et $S_{\mathcal{G}}$ le transform\'e strict de $S$ associ\'e au morphisme $\pi:=\pi_{\mathcal{G}}\circ \pi_{\mathcal{N}}$.\\

La proposition suivante est un analogue de la Proposition \ref{pr:Gdes1}.

\begin{proposition}
 \label{pr:desingE6} $S_{\mathcal{G}}$ est la r\'esolution minimale de $S$.
\end{proposition}

Soient $\E$ un diviseur essentiel sur $S$ ($\E\in  \Ess(S))$  et  $\alpha_{\E}$ le point g\'en\'erique de $N_{\E}$. On note 
\begin{center} $(\mu_x,\mu_y,\mu_z):=(\ordt{\com{\alpha_{\E}}}(x),\ordt{\com{\alpha_{\E}}}(y), \ordt{\com{\alpha_{\E}}}(z))$,\end{center}
o\`u $\com{\alpha_{\E}}$ est le comorphisme de $\alpha_{\E}$.\\

La proposition suivante est un analogue du Corollaire \ref{co:musygemi}.
\begin{proposition}
\label{pr:musygemiDn}
Soit $k$  un entier, $k\geq 2$.
\begin{itemize}
 \item[]Si $n=2k$, alors le vecteur $(\mu_x,\mu_y,\mu_z)$ appartient \`a l'union de  $G\tau_1$ et $G\tau_3$.
\item[] Si $n=2k-1$, alors le vecteur $(\mu_x,\mu_y,\mu_z)$ appartient \`a l'union de  $G\tau_2$ et $G\tau_3$.
\end{itemize}
\end{proposition}

D\'emontrer que l'application de Nash bijective pour les singularit\'es de type $\DD_n$, $n\geq 4$, \'equivaut \`a montrer que tous les wedges admissibles se rel\`event \`a la r\'esolution minimale de $\DD_n$ (voir \cite{Reg06}). Dans la suite, on  montre que pour chaque diviseur essentiel $\E$ tous les $K$-wedges admissibles centr\'es en $N_{\E}$  se rel\`event \`a la r\'esolution minimale de $S$.\\

Soit $\E\in \Ess(S)$ et on consid\`ere un $K$-wedge $\omega:\spec K[[s,t]]\rightarrow S$ admissible centr\'e en $N_{\E}$. On pose: \begin{center} $(\eta_x,\eta_y,\eta_z): =(\ordt \omega^{\star }(x),\ordt \omega^{\star }(y),\ordt \omega^{\star }(z))$. \end{center}

  On peut \'ecrire le comorphisme  de  $\omega$  de la fa\c con  suivante:\\
 
\begin{center} 
$\omega^{\star }(x)=t^{\eta_x}\chi,\;\omega^{\star }(y)=t^{\eta_y}\varphi,\; \omega^{\star }(z)=t^{\eta_z}\psi,$
\end{center}
o\`u les s\'eries formelles $\chi$, $\varphi$, $\psi$ ne sont pas divisibles par $t$.\\

La  proposition  suivante est un analogue du Corollaire \ref{co:etacone}. On note $\tau^0$ l'int\'erieur du c\^one $\tau\in \EN{\f}$.

\begin{proposition}
\label{pr:etaconeDn}
Soit $k$  un entier, $k\geq 2$.
 S'il existe au moins  une s\'erie formelle parmi les s\'eries  $\chi$, $\varphi$, $\psi$ qui n'est pas  inversible, alors on a:

 \begin{itemize}
\item[i)]si $n=2k$, alors le vecteur $(\eta_x,\eta_y,\eta_z)$ appartient \`a l'union de $\tau_1^{0}$ et $\tau_3^{0}$;
\item[ii)]si $n=2k-1$, alors le vecteur $(\eta_x,\eta_y,\eta_z)$ appartient \`a l'union de $\tau_2^{0}$ et $\tau_3^{0}$.
  \end{itemize}

\end{proposition}

En vertu de la Proposition \ref{pr:musygemiDn}, on a  $\mu_x\leq n-1$, $\mu_y\leq n-2$ et $\mu_z\leq 2$. Comme 
$\eta_x\geq 1$, $\eta_y\geq 1$ et $\eta_z\geq 1$, on a $\mu_x-\eta_x\leq n-2$, $\mu_y-\eta_y\leq n-3$ et $\mu_z-\eta_z\leq 1$.\\

D'apr\`es la  Proposition \ref{pr:invserBP}, si les s\'eries formelles $\chi$, $\varphi$ et $\psi$ sont inversibles, alors le $K$-wedge admissible $\omega$ centr\'e en $N_{\E}$ se rel\`eve \`a la r\'esolution minimale de $S$.

En vertu de la Proposition \ref{pr:FI-g}, on peut majorer le nombre de facteurs irr\'eductibles compt\'es avec multiplicit\'e des s\'eries formelles $\chi$, $\varphi$ et $\psi$ \`a l'aide des $v$-ordres. Plus  pr\'ecis\'ement,  il existe un vecteur $v\in \QQ^{2}_{>0} $  tel que:\\

\begin{itemize}
   \item[]$\FI(\chi)\leq \degt{\chi_v}= \nu_v\chi =\mu_x-\eta_x\leq n-2$;
\item[] $\FI(\varphi)\leq \degt{\varphi_v}= \nu_v\varphi= \mu_y-\eta_y\leq n-3$;
 \item[]$\FI(\psi)\leq \degt{\psi_v}= \nu_v\psi= \mu_z-\eta_z\leq 1$.\\
\end{itemize}

Sauf mention du contraire, dans toute la suite le vecteur  $v\in \QQ_{>0}^2$ satisfait la propri\'et\'e ci-dessus.
Le r\'esultat suivant est un corollaire de la Proposition \ref{pr:FI-g}
\begin{corollaire}
\label{co:FIDn}
 La s\'erie formelle $\psi$ est irr\'eductible ou inversible.
\end{corollaire}

Maintenant, on donne notre preuve de la bijectivit\'e de l'application de Nash pour les singularit\'es de type $\DD_n$.

 \begin{theoreme} L'application de Nash $\mathcal{N}_S$ associ\'ee \`a $S$ est bijective.
\label{nashDn}
\end{theoreme}
\begin{proof}
 En vertu  de la Proposition \ref{pr:invserBP},  pour montrer que le $K$-wedge admissible $\omega$ se rel\`eve \`a la r\'esolution minimale  de $S$, il suffit de montrer que les s\'eries formelles $\chi$, $\varphi$ et $\psi$ sont inversibles. 
En raisonnant par l'absurde, on  suppose qu'il y a au moins l'une d'elles n'est pas inversible.

On remarque qu'une s\'erie formelle $\phi \in K[[s,t]]$ est inversible dans $K[[s,t]]$ si et seulement si elle est inversible dans $\overline{K}[[s,t]]$, o\`u $\overline{K}$ est la cl\^oture alg\'ebrique de $K$. Dans toute la suite,  on suppose que le corps $K$ est alg\'ebriquement clos.\\

Le $K$-wedge $\omega$ doit satisfaire l'\'equation $x^2+z(y^2+z^{n-2})=0$, d'o\`u la relation suivante:
             
\begin{equation} t^{2\eta_x}\chi^2+t^{2\eta_y+\eta_z}\varphi^2\psi+t^{(n-1)\eta_z}\psi^{n-1}=0 \label{eq:dn}
\end{equation}

En vertu de la Proposition \ref{pr:etaconeDn}, on a les cas suivants:\\

\begin{itemize}
 \item[Cas 1).] le vecteur $(\eta_x,\eta_y,\eta_z)$ appartient \`a l'int\'erieur $\tau_3^0$ du c\^one $\tau_3$; 
\item[Cas 2).] l'entier $n$ est \'egal \`a $ 2k$, o\`u $k$ est un entier sup\'erieur ou \'egal \`a $2$, et le vecteur $(\eta_x,\eta_y,\eta_z)$ appartient \`a l'int\'erieur $\tau_1^0$ du c\^one $\tau_1$; 
\item[Cas 3).] l'entier $n$ est \'egal \`a $ 2k-1$, o\`u $k$ est un entier sup\'erieur ou \'egal \`a $2$, et le vecteur $(\eta_x,\eta_y,\eta_z)$ appartient \`a l'int\'erieur $\tau_2^0$ du c\^one $\tau_2$.\\
\end{itemize}
Dans chaque cas ci-dessus, on va obtenir une contradiction, ce qui ach\`eve la preuve du Th\'eor\`eme.\\
 
Cas 1).  On suppose que $(\eta_x,\eta_y,\eta_z)\in\tau_3^0$.   Dans ce cas, on a  $\eta_z=2$ et $\eta_x=\eta_y+1$. 
Au moyen de la Relation (\ref{eq:dn}), on obtient la relation suivante:
$$\chi^2+\varphi^2\psi=-t^{2(n-2-\eta_y)}\psi^{n-1}.$$
 On remarque que $n-2-\eta_y>0$, car le vecteur $(\eta_x,\eta_y,\eta_z)$ appartient \`a l'int\'erieur du c\^one $\tau_3$.

D'apr\`es les Propositions \ref{pr:GtauDn} et \ref{pr:musygemiDn}, on a $\mu_z\leq 2$. Par cons\'equent, $\mu_z-\eta_z=0$, ce qui implique que la s\'erie formelle $\psi$ est inversible (voir la proposition \ref{pr:FI-g}).\\

Comme $\psi$ est inversible et $K$ alg\'ebriquement clos, il existe une s\'erie formelle inversible  $\psi_0$ tel que 
$\psi=-\psi_0^2$,  d'o\`u 
$$(\chi+\psi_0\varphi)(\chi-\psi_0\varphi)=\mbox{(-1)}^{n}t^{2(n-2-\eta_y)}\psi_0^{2(n-1)}.$$

Par cons\'equent, $t$ divise $\chi$ et $\varphi$ ou $\chi$ et $\varphi$ sont inversibles. Ce sont des contradictions.\\

 Cas 2). On suppose  que $n=2k$, o\`u $k$ est un entier sup\'erieur ou \'egal \`a $2$ et  $(\eta_x,\eta_y,\eta_z)\in \tau_1^0$.\\

Comme $(\eta_x,\eta_y,\eta_z)$ appartient \`a $\tau_1^0$ et $n=2k$, le vecteur $(\eta_x,\eta_y,\eta_z)$ satisfait les relations suivantes:
\begin{center}
 $(k-1)\eta_z=\eta_y$ et $2\eta_x>(2k-1)\eta_z$.
\end{center}
Comme $1\leq \eta_x\leq \mu_x\leq 2k-1$ (voir les Propositions \ref{pr:GtauDn} et \ref{pr:musygemiDn}), on obtient que $\eta_z=1$, $\eta_y=2k-1$ et $\eta_x\geq k$. On a donc $\mu_z=2$. En effet, si $\mu_z=1$, alors $(\mu_x,\mu_y,\mu_z)=(k,k-1,1)$. D'apr\`es la Proposition \ref{pr:FI-g}, les s\'eries formelles $\chi$, $\varphi$ et $\psi$ sont inversibles, ce qui est une contradiction, car on a suppos\'e que parmi ces s\'eries formelles il y en a au moins une qui n'est pas inversible.\\

Au moyen de la Relation (\ref{eq:dn}), on obtient la relation suivante: 

\begin{center} $t^{2\eta_x+1-n}\chi^2+\varphi^2\psi=-\psi^{n-1}.$\end{center}

On remarque que $\psi$ est irr\'eductible car  $\nu_v\psi=\mu_z-\eta_z=1$ (voir la Proposition \ref{pr:FI-g} et le Corollaire \ref{co:FIDn})
Au moyen de la relation ci-dessus, on obtient que $\chi=\chi_0\psi^{k}$ et $\varphi=\varphi_0\psi^{k-1}$, o\`u $\chi_0$ et  $\varphi_0$
sont deux s\'eries formelles qui satisfont la relation suivante:
\begin{center}
 $t^{2\eta_x+1-n}\chi_0^2\psi+\varphi_0^2=-1.$
\end{center}

Comme $\nu_v\psi=1$, on obtient que  $\nu_v\chi=\nu_v\chi_0\psi^k=\nu_v\chi_0+k$ . Comme  $\mu_x\leq 2k-1$ et $\eta_x\geq k$, on a $\nu_v\chi=\mu_x-\eta_x\leq k-1$ (voir la Proposition \ref{pr:FI-g}). Par cons\'equent, on a  $\nu_v\chi_0 \leq -1$, ce qui est une contradiction.\\

La preuve du Cas 3) est analogue \`a celle du Cas 2).\\

Cas 3). On suppose  que $n=2k-1$, o\`u $k$ est un entier sup\'erieur ou \'egal \`a $2$ et  $(\eta_x,\eta_y,\eta_z)\in \tau_2^0$.\\

Comme $(\eta_x,\eta_y,\eta_z)$ appartient \`a $\tau_2^0$ et $n=2k-1$, le vecteur $(\eta_x,\eta_y,\eta_z)$ satisfait les relations suivantes:
\begin{center}
 $(k-1)\eta_z=\eta_x$ et $2\eta_y>(2k-3)\eta_z$.
\end{center}
Comme $1\leq \eta_y\leq \mu_y\leq 2k-3$ (voir les Propositions \ref{pr:GtauDn} et \ref{pr:musygemiDn}), on obtient que $\eta_z=1$, $\eta_x=k-1$ et $\eta_y\geq k-1$. On a donc $\mu_z=2$. En effet, si $\mu_z=1$, alors $(\mu_x,\mu_y,\mu_z)=(k-1,k-1,1)$. D'apr\`es la Proposition \ref{pr:FI-g}, les s\'eries formelles $\chi$, $\varphi$ et $\psi$ sont inversibles, ce qui est une contradiction car on a
 suppos\'e que parmi ces s\'eries formelles il y en a au moins une qui n'est pas inversible.\\

Au moyen de la Relation (\ref{eq:dn}), on obtient la relation suivante: 
$$\chi^2+t^{2\eta_y+2-n}\varphi^2\psi=-\psi^{n-1}.$$

On remarque que $\psi$ est irr\'eductible car  $\nu_v\psi=\mu_z-\eta_z=1$ (voir la Proposition \ref{pr:FI-g} et le Corollaire \ref{co:FIDn})
Au moyen de la relation ci-dessus, on obtient que $\chi=\chi_0\psi^{k-1}$ et $\varphi=\varphi_0\psi^{k-1}$, o\`u $\chi_0$ et  $\varphi_0$
sont deux s\'eries formelles qui satisfont la relation suivante:
\begin{center}
 $\chi_0^2+t^{2\eta_y+2-n}\varphi_0^2\psi=-1$.
 
\end{center}

Comme $\nu_v\psi=1$, on obtient que  $\nu_v\varphi=\nu_v\varphi_0\psi^{k-1}=\nu_v\varphi_0+k-1$ . Comme  $\mu_y\leq 2k-3$ et $\eta_y\geq k-1$, on a $\nu_v\varphi=\mu_y-\eta_y\leq k-2$ (voir la Proposition \ref{pr:FI-g}). Par cons\'equent, on a  $\nu_v\varphi_0 \leq -1$, ce qui est une contradiction.\\

Dans le trois cas pr\'ec\'edent, on a obtenu une contradiction, d'o\`u le th\'eor\`eme.\end{proof}

\bibliographystyle{alpha}
\bibliography{Nash_toulouse}

\end{document}